\documentclass[11pt]{amsart}
\usepackage{amssymb}
\usepackage{amsmath}
\usepackage{amsthm}
\usepackage{verbatim}
\usepackage{url}
\usepackage{mathtools}
\usepackage{dsfont}

\usepackage[mathscr]{euscript}
\usepackage[T1]{fontenc}
\usepackage{tikz-cd}
\usetikzlibrary{calc}
\usepackage[top=2cm, bottom=2.1cm, left=2.5cm, right=2.7cm]{geometry}

\usepackage
{hyperref}

\DeclareSymbolFont{rsfs}{U}{rsfs}{m}{n}
\DeclareSymbolFontAlphabet{\mathrsfs}{rsfs}

\usepackage{color}

\definecolor{cldcol}{rgb}{0,0.5,0.4}

\renewcommand{\mathbb}{\mathds}

\theoremstyle{plain}
\newtheorem{theorem}{Theorem}[section]

\newtheorem{proposition}[theorem]{Proposition}
\newtheorem{lemma}[theorem]{Lemma}
\newtheorem{corollary}[theorem]{Corollary}

\theoremstyle{definition}

\newtheorem{remark}[theorem]{Remark}

\numberwithin{equation}{subsection}

\newtheorem*{notation*}{Notation}

\newcommand{\isoto}{{\,\overset\sim\longrightarrow\,}}

\newcommand{\hs}{\kern 0.8pt}
\newcommand{\hssh}{\kern 1.2pt}
\newcommand{\hshs}{\kern 1.6pt}
\newcommand{\hssss}{\kern 2.0pt}

\newcommand{\hm}{\kern -0.8pt}
\newcommand{\hmm}{\kern -1.2pt}

\begin{document}

\title[Central algebraic cohomology of Galois gerbes and its abelianization]%
{Central algebraic cohomology of Galois gerbes and its abelianization}

\author{Taeyeoup Kang}
\address{School of Mathematical Sciences,
Tel Aviv University, 6997801 Tel Aviv, Israel}

\address{Current address: Department of Mathematics, Ben-Gurion University of the
Negev, 8410501 Be’er Sheva, Israel }

\email{taeyeoupkang@gmail.com}

\thanks{This research was supported
by the Israel Science Foundation (grant 1030/22).}



\date{\today}

\begin{abstract}
Let $F$ be a local field of characteristic zero or a number field, and let $G$ be a connected reductive $F$-group with arbitrary central $F$-subgroup $Z$.
For a Galois gerbe $\mathcal{E}$ over $F$ with band $D$, we endow $H^1_{\mathrm{ab}}(\mathcal{E}, Z \to G)$ with a canonical abelian group structure.
The image of the Newton map is the kernel of an obstruction map to $H^2_{\mathrm{ab}}(F, G)$, and each nonempty fiber of the abelianized Newton map is a torsor under $H^1_{\mathrm{ab}}(F, G)$.
Given a morphism from Kaletha's rigid gerbe to $\mathcal{E}$, we express $H^1_{\mathrm{ab}}(\mathcal{E}, Z \to G)$ as a fiber product of abelianized rigid cohomology and $\operatorname{Hom}_F(D, Z)$.
Over number fields, abelianization and localization at the real places recover $H^1(\mathcal{E}, Z \to G)$ through a Cartesian square.
We obtain explicit formulas for Kottwitz's local and global gerbes and for Fargues's extended Kottwitz gerbe.
\end{abstract}

\maketitle

\tableofcontents

\section{Introduction}\label{s:Intro}

Let $F$ be a field of characteristic zero, let $\bar{F}$ be a fixed algebraic closure of $F$, and set $\Gamma \coloneqq \operatorname{Gal}(\bar{F} / F)$.
Let $G$ be a connected reductive $F$-group and let $Z \subseteq Z(G)$ be a central $F$-subgroup, not necessarily finite or connected.
Then $Z$ is of multiplicative type.

Let $D = \varprojlim_i D_i$ be a pro-algebraic $F$-group of multiplicative type, where each $D_i$ is an algebraic $F$-group of multiplicative type and the transition maps are surjective.
Following \cite{Kot97} and \cite{Kal16}, a Galois gerbe $\mathcal{E}$ banded by $D$ is an extension of topological groups
\[
1 \longrightarrow D(\bar{F}) \longrightarrow \mathcal{E} \longrightarrow \Gamma \longrightarrow 1
\]
satisfying the conditions of Section~\ref{s:Alg-cohomology}.
We consider the central algebraic cohomology set
\[
H^1(\mathcal{E}, Z \to G).
\]
Its elements are represented by \emph{algebraic cocycles} $(x, \nu)$, where $x \colon \mathcal{E} \to G(\bar{F})$ is a continuous $1$-cocycle and $\nu \colon D \to Z$ is a homomorphism of $F$-group schemes such that $x_d = \nu(d)$ for every $d \in D(\bar{F})$.

One motivation for studying these sets comes from the refined local Langlands correspondence.
For a local field $F$ and finite $Z$, the set $H^1(\mathcal{E}^{\mathrm{rig}}, Z \to G)$ for Kaletha's rigid gerbe $\mathcal{E}^{\mathrm{rig}}$ parametrizes isomorphism classes of rigid inner twists of $G$ realized by $Z$; see \cite[Section~5.1]{Kal16}.
For quasi-split $G$, Kaletha uses rigid inner twists to normalize endoscopic transfer factors and to formulate the conjectural internal parametrization of tempered $L$-packets on its inner forms; see \cite[Sections~5.3--5.4]{Kal16}.
For $p$-adic groups with connected center, Kaletha compares this formulation with the one using the basic set $B(G)_{\mathrm{bas}} = H^1(\mathcal{E}^{\mathrm{iso}}, Z(G) \to G)$ for Kottwitz's gerbe $\mathcal{E}^{\mathrm{iso}}$; see \cite[Sections~4.1--4.2]{Kal18b}.

Following \cite[Section~3.4]{Kal16} and \cite[Section~3.8]{Kal18a}, we consider the abelian central cohomology $H^1_{\mathrm{ab}}(\mathcal{E}, Z \to G)$ and the abelianization map
\[
\operatorname{ab} \colon H^1(\mathcal{E}, Z \to G) \longrightarrow H^1_{\mathrm{ab}}(\mathcal{E}, Z \to G)
\]
defined in \eqref{e:*}.

We say that $F$ is of \emph{Galois--Douai type} if, for every semisimple simply connected $F$-group $H$, the connecting map
\[
H^1(F, H^{\mathrm{ad}}) \longrightarrow H^2(F, Z(H))
\]
associated with $1 \to Z(H) \to H \to H^{\mathrm{ad}} \to 1$ is surjective; see \cite[Definition~5.2]{GA12}.
For instance, local fields of characteristic zero and number fields are fields of Galois--Douai type; see \cite[Lemma~5.7]{Bor93}.
Under this hypothesis, Proposition~\ref{prop:gerbe-abelian-group} gives $H^1_{\mathrm{ab}}(\mathcal{E}, Z \to G)$ a canonical abelian group structure, which is functorial in $(Z \to G)$ by Lemma~\ref{lem:gerbe-abelian-functoriality}.
With this group structure, Propositions~\ref{prop:gerbe-inflation-restriction} and~\ref{prop:general-newton-fibers}, together with \eqref{eq:gerbe-obstruction-compatibility}, give a commutative diagram with exact rows
\[
\begin{tikzcd}[column sep=small]
1 \arrow[r]
& H^1(F, G) \arrow[r, "\operatorname{inf}"] \arrow[d, "\operatorname{ab}_G^1"']
& H^1(\mathcal{E}, Z \to G) \arrow[r, "\mathcal{N}"] \arrow[d, "\operatorname{ab}"]
& \operatorname{Hom}_F(D, Z) \arrow[r, "\delta_{\mathcal{E}}"] \arrow[d, equal]
& H^2(F, G) \arrow[d, "\operatorname{ab}_G^2"] \\
1 \arrow[r]
& H^1_{\mathrm{ab}}(F, G) \arrow[r, "\operatorname{inf}_{\mathrm{ab}}"']
& H^1_{\mathrm{ab}}(\mathcal{E}, Z \to G) \arrow[r, "\mathcal{N}_{\mathrm{ab}}"']
& \operatorname{Hom}_F(D, Z) \arrow[r, "\delta_{\mathcal{E},\mathrm{ab}}"']
& H^2_{\mathrm{ab}}(F, G).
\end{tikzcd}
\]
The lower row is an exact sequence of abelian groups.
The upper row is exact in the sense of pointed sets, with exactness at $\operatorname{Hom}_F(D, Z)$ meaning that $\nu$ lies in the image of $\mathcal{N}$ if and only if $\delta_{\mathcal{E}}(\nu)$ is neutral in $H^2(F, G)$, as explained in Section~\ref{subsec:abelianization}.

Here $\operatorname{inf}$ and $\operatorname{inf}_{\mathrm{ab}}$ denote inflation along $\mathcal{E} \to \Gamma$, and the Newton maps satisfy $\mathcal{N}([x, \nu]) = \mathcal{N}_{\mathrm{ab}}(\operatorname{ab}[x, \nu]) = \nu$.
Section~\ref{sec:gerbes} recalls Borovoi's abelianization maps $\operatorname{ab}_G^1$ and $\operatorname{ab}_G^2$ and defines the obstruction maps $\delta_{\mathcal{E}}$ and $\delta_{\mathcal{E}, \mathrm{ab}}$.
For $\nu \in \operatorname{Hom}_F(D, Z)$, the fiber $\mathcal{N}^{-1}(\nu)$ is nonempty if and only if $\delta_{\mathcal{E}, \mathrm{ab}}(\nu) = 0$.
For each such $\nu$, the fiber $\mathcal{N}_{\mathrm{ab}}^{-1}(\nu)$ is a torsor under $H^1_{\mathrm{ab}}(F, G)$ via $\operatorname{inf}_{\mathrm{ab}}$.
For non-Archimedean local fields and number fields, Corollaries~\ref{cor:general-padic-newton-fibers} and~\ref{cor:general-global-newton-fibers} describe $H^1_{\mathrm{ab}}(F, G)$ and $H^2_{\mathrm{ab}}(F, G)$ in terms of Borovoi's algebraic fundamental group $\pi_1(G)$ \cite{Bor98}, following \cite{BK26}.

By Corollary~\ref{cor:abelianization-bijective}, the abelianization map $\operatorname{ab}$ is bijective if $F$ is a non-Archimedean local field, a totally imaginary number field, or $\mathbb{C}$.
In these cases, $H^1(\mathcal{E}, Z \to G)$ carries the abelian group structure for which $\operatorname{ab}$ is an isomorphism.
When $G$ is a torus, abelianization is bijective over every field $F$ of characteristic zero, and pointwise multiplication of cocycles gives $H^1(\mathcal{E}, Z \to G)$ its abelian group structure.

Over local fields and number fields, our aim is to express $H^1_{\mathrm{ab}}(\mathcal{E}, Z \to G)$ in terms of Kaletha's rigid cohomology and homomorphisms from the band $D$.
Combining this comparison with Kaletha's parametrizations yields formulas in terms of algebraic fundamental groups.

For a central $F$-subgroup $Z$ of $G$, let
\[
\tilde{\pi}_1(G, Z) \coloneqq \varinjlim_{A \subseteq Z} \pi_1(G / A),
\]
where $A$ runs over the finite $F$-subgroups of $Z$.
By Proposition~\ref{prop:colimit-quotient}, there is a natural exact sequence
\begin{equation}\label{e:exact-seq}
0 \longrightarrow \pi_1(G)
\longrightarrow \tilde{\pi}_1(G, Z)
\longrightarrow \operatorname{Hom}_{\mathbb{Z}}\bigl( X^*(Z), \mathbb{Q} / \mathbb{Z} \bigr)
\longrightarrow 0.
\end{equation}
For a finite Galois extension $E / F$ through which the $\Gamma$-action on $\tilde{\pi}_1(G, Z)$ factors, set $\Gamma_{E / F} \coloneqq \operatorname{Gal}(E / F)$ and
\[
\tilde{\pi}_1(G, Z)^{N_{E / F}}
\coloneqq
\ker \left[ N_{E / F} \colon \tilde{\pi}_1(G, Z) \longrightarrow \tilde{\pi}_1(G, Z) \right],
\qquad
N_{E / F} \coloneqq \sum_{\sigma \in \Gamma_{E / F}} \sigma.
\]
For a $\Gamma$-module $M$, write
$I_\Gamma M \coloneqq
\langle \sigma m - m \mid \sigma \in \Gamma,\ m \in M \rangle_{\mathbb{Z}}$ and 
$M_\Gamma \coloneqq M / I_\Gamma M$.

For local $F$, let $\mathcal{E}^{\mathrm{rig}}$ denote Kaletha's local rigid gerbe with band $u$, as in Section~\ref{s:Kaletha-local-gerbe}.

\begin{theorem}[Theorem~\ref{thm:main-bis}]\label{thm:main}
Let $F$ be a local field of characteristic zero.
Suppose that $f \colon \mathcal{E}^{\mathrm{rig}} \to \mathcal{E}$ is a morphism of gerbes over $\Gamma$ inducing an $F$-homomorphism $\phi \colon u \to D$ on the bands.
Then $f$ induces an isomorphism of abelian groups, natural in $(Z \to G)$,
\begin{equation}\label{e:intro-formula-local}
H^1_{\mathrm{ab}}(\mathcal{E}, Z \to G)
\xrightarrow{\sim}
\left( \varinjlim_{E / F}
\frac{\tilde{\pi}_1(G, Z)^{N_{E / F}}}{I_\Gamma \pi_1(G)} \right)
\underset{\operatorname{Hom}_F(u, Z)}{\times}
\operatorname{Hom}_F(D, Z),
\end{equation}
where $E / F$ runs over the finite Galois extensions in $\bar{F}$ through which the $\Gamma$-action on $\tilde{\pi}_1(G, Z)$ factors.
The first map to the base is induced by the quotient map in \eqref{e:exact-seq} under the identification of $\operatorname{Hom}_F(u, Z)$ in Lemma~\ref{lem:u-nonfinite}.
The second map sends $\nu$ to $\nu \circ \phi$.

Moreover, abelianization and pullback induce a bijection of pointed sets, natural in $(Z \to G)$,
\[
H^1(\mathcal{E}, Z \to G)
\xrightarrow{\sim}
H^1_{\mathrm{ab}}(\mathcal{E}, Z \to G)
\underset{H^1_{\mathrm{ab}}(\mathcal{E}^{\mathrm{rig}}, Z \to G)}{\times}
H^1(\mathcal{E}^{\mathrm{rig}}, Z \to G),
\]
where the maps to the base are $f^*$ and abelianization respectively.
\end{theorem}

Theorem~\ref{thm:local-rigid-formula} describes the first factor on the right-hand side of \eqref{e:intro-formula-local} for non-Archimedean, real, and complex fields.
When $F$ is non-Archimedean or complex, Corollary~\ref{cor:abelianization-bijective} shows that \eqref{e:intro-formula-local} also holds with $H^1(\mathcal{E}, Z \to G)$ on the left.

Now let $F$ be a number field, and let $V$ be its set of places.
For each $v \in V$, choose a place $\dot v$ of $\bar{F}$ above $v$ so that $\dot V \coloneqq \{ \dot v \mid v \in V \}$ satisfies \cite[Condition~3.3.1]{KT23}.
Let $\mathcal{E}^{\mathrm{rig}}_{\dot V}$ be Kaletha's global rigid gerbe with band $P_{\dot V}$, as in Section~\ref{s:Kaletha-global-gerbe}.
For a finite Galois extension $E / F$ and a finite subset $S \subseteq V$, let $S_E$ be the set of places of $E$ above $S$, and set $\dot S_E \coloneqq \{ \dot v \vert_E \mid v \in S \}$.
We call $(E, S)$ \emph{admissible} if $(E / F, S, \dot S_E)$ satisfies \cite[Conditions~2.4.1]{KT23}.
The following is the global counterpart of Theorem~\ref{thm:main}.

\begin{theorem}[Theorem~\ref{thm:global-abelianized-formula}]\label{thm:main-global}
Suppose that $f \colon \mathcal{E}^{\mathrm{rig}}_{\dot V} \to \mathcal{E}$ is a morphism of gerbes over $\Gamma$ inducing an $F$-homomorphism $\phi \colon P_{\dot V} \to D$ on the bands.
Then $f$ induces an isomorphism of abelian groups, natural in $(Z \to G)$,
\[
H^1_{\mathrm{ab}}(\mathcal{E}, Z \to G)
\xrightarrow{\sim}
\left( \varinjlim_{(E, S)}
\frac{\tilde{\pi}_1(G, Z)[S_E, \dot S_E]_0^{N_{E / F}}}
     {I_\Gamma \bigl( \pi_1(G)[S_E]_0 \bigr)} \right)
\underset{\operatorname{Hom}_F(P_{\dot V}, Z)}{\times}
\operatorname{Hom}_F(D, Z),
\]
where the limit runs over a cofinal tower of admissible pairs $(E, S)$ such that the $\Gamma$-action on $\tilde{\pi}_1(G, Z)$ factors through $\Gamma_{E / F}$.
The first map to the base is the restriction map described in Theorem~\ref{thm:global-rigid-formula}.
The second map sends $\nu$ to $\nu \circ \phi$.
\end{theorem}

Here $\tilde{\pi}_1(G, Z)[S_E, \dot S_E]_0^{N_{E / F}}$, $\pi_1(G)[S_E]_0$, and the transition maps are explained in Section~\ref{s:Kaletha-global-gerbe}, following \cite[Section~3.8]{Kal18a}.

In both the local and global cases, Lemma~\ref{lem:comparison-dependence} describes how pullback changes when $f$ varies with fixed band map $\phi$.
If $D$ is of finite type, a morphism from Kaletha's rigid gerbe with a suitable band map exists by Corollary~\ref{cor:comparison-exist-finite-type}.

Over a number field, abelianization and localization at the real places recover $H^1(\mathcal{E}, Z \to G)$ for every Galois gerbe $\mathcal{E}$ under consideration.

\begin{theorem}[Theorem~\ref{prop:global-local-cartesian}]\label{thm:main-real-localization}
Let $F$ be a number field and let $\mathcal{E}$ be a Galois gerbe over $F$.
For each real place $v$ of $F$, let $\mathcal{E}_v$ denote the localization of $\mathcal{E}$ over $F_v$, and set $\operatorname{loc}_\infty \coloneqq (\operatorname{loc}_v)_{v \text{ real}}$.
Abelianization and localization give a Cartesian square
\[
\begin{tikzcd}[column sep=large]
H^1(\mathcal{E}, Z \to G) \arrow[r, "\operatorname{ab}"] \arrow[d, "\operatorname{loc}_\infty"']
& H^1_{\mathrm{ab}}(\mathcal{E}, Z \to G) \arrow[d, "\operatorname{loc}_\infty"] \\
\displaystyle \prod_{v \text{ real}} H^1(\mathcal{E}_v, Z \to G) \arrow[r, "{\prod_v \operatorname{ab}_v}"']
& \displaystyle \prod_{v \text{ real}} H^1_{\mathrm{ab}}(\mathcal{E}_v, Z \to G).
\end{tikzcd}
\]
\end{theorem}

Over a local field $F$ of characteristic zero, we apply our formula to Kottwitz's gerbe $\mathcal{E}^{\mathrm{iso}}$ \cite{Kot14} and Fargues's extended Kottwitz gerbe $\mathcal{E}^{\mathrm{ext}}$ \cite{Far22}.
If $F$ is non-Archimedean, Theorem~\ref{thm:kottwitz-local-explicit} gives an isomorphism of abelian groups
\[
H^1(\mathcal{E}^{\mathrm{iso}}, Z \to G)
\xrightarrow{\sim}
\ker \left[ \pi_1(G)_\Gamma \longrightarrow \pi_1(G / Z)_{\Gamma,\mathrm{free}} \right],
\]
where the map is induced by $G \to G / Z$ and the subscript $\mathrm{free}$ denotes the quotient by the torsion subgroup.
Under the inclusion of this kernel in $\pi_1(G)_\Gamma$, this bijection is the restriction of Kottwitz's map $\kappa_G$ by Corollary~\ref{cor:kottwitz-local-map}.
The same theorem gives the corresponding descriptions for $F = \mathbb{R}$ and $F = \mathbb{C}$.
For finite $Z$, the left-hand side identifies with $H^1(F, G)$.
For $Z = Z(G)$, it is Kottwitz's basic set $B(G)_{\mathrm{bas}}$.

We set $\mathcal{E}^{\mathrm{ext}} \coloneqq \mathcal{E}^{\mathrm{iso}} \times_\Gamma \mathcal{E}^{\mathrm{rig}}$.
For non-Archimedean $F$, this product represents Fargues's extended Kottwitz gerbe by \cite[Theorem~5.6]{Far22}.
In this case, the morphism $f \colon \mathcal{E}^{\mathrm{rig}} \to \mathcal{E}^{\mathrm{iso}}$ fixed in \eqref{eq:kottwitz-comparison-map} gives, by Theorem~\ref{thm:extended-kottwitz-local-cases}, an isomorphism of abelian groups
\[
H^1(\mathcal{E}^{\mathrm{ext}}, Z \to G)
\xrightarrow{\sim}
\left( \frac{\tilde{\pi}_1(G, Z)}{I_\Gamma \pi_1(G)} \right)_{\mathrm{tors}}
\times \operatorname{Hom}_\Gamma(X^*(Z), \mathbb{Q}).
\]
The same theorem also gives the formulas for real and complex fields.
For non-Archimedean $F$ and $Z = Z(G)$, the left-hand side is $B^e(G)_{\mathrm{bas}}$.
Let $\kappa_G^e \colon B^e(G) \to \pi_1(G)^e_\Gamma$ denote the extended Kottwitz map constructed in \cite[Section~9.5]{Far22}.
By Lemma~\ref{lem:fargues-basic-formula}, the restriction of $\kappa_G^e$ to $B^e(G)_{\mathrm{bas}} \coloneqq H^1(\mathcal{E}^{\mathrm{ext}}, Z(G) \to G)$ is given by
$\kappa_G^e(r, \lambda) = r + [\iota(\lambda)]
\in \tilde{\pi}_1(G, Z(G))/I_\Gamma \pi_1(G)
\simeq \pi_1(G)^e_\Gamma$,
where the isomorphism is given by Lemma~\ref{lem:extended-fundamental-identification}, $\iota$ is the injection in \eqref{eq:kottwitz-central-cocharacters}, and $[\iota(\lambda)]$ denotes the class of $\iota(\lambda)$ modulo $I_\Gamma \pi_1(G)$.
This formula agrees with the parametrization of \cite[Theorem~9.12]{Far22}.
Under this parametrization, the comparison map from $B(G)_{\mathrm{bas}}$ to $H^1(\mathcal{E}^{\mathrm{rig}}, Z(G) \to G)$ corresponds to the inclusion $\pi_1(G)_\Gamma \hookrightarrow \pi_1(G)^e_\Gamma$ followed by projection onto the torsion subgroup, as in \cite[Proposition~10.6]{Far22}.

Over a number field $F$, let $D^{\mathrm{iso}}$ be the band of Kottwitz's gerbe and let $q \colon G \to G / Z$ be the quotient map.
Theorem~\ref{thm:kottwitz-global-explicit} yields a canonical isomorphism of abelian groups
\begin{equation}\label{e:formula-Kottwitz-global}
H^1_{\mathrm{ab}}(\mathcal{E}^{\mathrm{iso}}, Z \to G)
\xrightarrow{\sim}
\ker \left[
\bigl( \pi_1(G)[V_E]_0 \bigr)_{\Gamma_{E / F}}
\xrightarrow{N_{G / Z}^\diamond \circ q_*}
\bigl( \pi_1(G / Z) \otimes_{\mathbb{Z}} X^*(D^{\mathrm{iso}}) \bigr)^\Gamma
\right].
\end{equation}
Here $E / F$ is a finite Galois extension through which the $\Gamma$-actions on $\pi_1(G)$ and $\pi_1(G / Z)$ factor, and $N_{G / Z}^\diamond$ is Kottwitz's norm map \cite[Definition~11.3]{Kot14}, recalled in Section~\ref{subsec:global-kottwitz}.
We write $V_E$ for the set of places of $E$ and $\pi_1(G)[V_E]_0$ for the kernel of the summation map on $\pi_1(G)[V_E]$, with its diagonal $\Gamma_{E / F}$-action.
The group $\bigl( \pi_1(G)[V_E]_0 \bigr)_{\Gamma_{E / F}}$ identifies canonically with Kottwitz's $A(F, G)$ by \cite[Section~11.4]{Kot14}, so \eqref{e:formula-Kottwitz-global} is independent of $E$ under this identification.
After composition with abelianization and the inclusion of the kernel in $A(F, G)$, this isomorphism is the restriction to $B(Z \to G) \coloneqq H^1(\mathcal{E}^{\mathrm{iso}}, Z \to G)$ of Kottwitz's global map $\kappa_G \colon B(G) \to A(F, G)$ constructed in \cite[Section~11.5]{Kot14}.
The Cartesian square of Corollary~\ref{cor:kottwitz-global-cartesian} recovers $B(Z \to G)$ from its abelianization and its real localizations.
For totally imaginary $F$, Corollary~\ref{cor:abelianization-bijective} shows that \eqref{e:formula-Kottwitz-global} also holds with $B(Z \to G)$ on the left.

The rest of the paper is organized as follows.
Section~\ref{sec:gerbes} develops abelianization and comparison for Galois gerbes and constructs a canonical abelian group structure on $H^1_{\mathrm{ab}}(\mathcal{E}, Z \to G)$ over fields of Galois--Douai type.
It also expresses the abelianized inflation--restriction sequence in terms of $\pi_1(G)$ over non-Archimedean local fields and number fields.
Section~\ref{sec:local-formula} proves the formula over local fields and applies it to Kottwitz's local gerbe and Fargues's extended Kottwitz gerbe.
Section~\ref{sec:number-field} proves the formula over number fields, establishes recovery from real localizations, and applies the formula to Kottwitz's global gerbe.

\medskip
\noindent
\textsc{Acknowledgements.}
The author is grateful to Mikhail Borovoi for suggesting this problem and for many helpful discussions.
The author also thanks Tasho Kaletha for his suggestion to consider Fargues's extended Kottwitz gerbe.
AI tools were used to assist with typesetting and language editing.

\section{Central algebraic cohomology of a Galois gerbe}
\label{sec:gerbes}

\subsection{Galois gerbes and algebraic cohomology}\label{s:Alg-cohomology}

Let $F$ be a field of characteristic zero, fix an algebraic closure $\bar{F}$ of $F$, and set $\Gamma \coloneqq \operatorname{Gal}(\bar{F} / F)$.
By an $F$-group we mean an $F$-group scheme.
An $F$-group $D$ is a \emph{pro-algebraic group of multiplicative type} if $D \simeq \varprojlim_i D_i$, where each $D_i$ is an $F$-group of finite type and of multiplicative type and the transition maps are surjective.
Such a group $D$ need not be of finite type.

For an $F$-group $H$ of finite type, we equip $H(\bar{F})$ with the discrete topology.
For a pro-algebraic $F$-group $H = \varprojlim_i H_i$ of multiplicative type, we equip $H(\bar{F}) = \varprojlim_i H_i(\bar{F})$ with the inverse limit topology.
For $\sigma \in \Gamma$ and $a \in H(\bar{F})$, we denote the image of $a$ under $\sigma$ by ${}^\sigma a$.
In either case, we write $Z^1(F, H)$ for the set of continuous maps $a \colon \Gamma \to H(\bar{F})$ satisfying $a_{\sigma \tau} = a_\sigma \, {}^\sigma a_\tau$ for all $\sigma, \tau \in \Gamma$.
We write $H^1(F, H)$ for the set of orbits under the $H(\bar{F})$-action given by $(h \cdot a)_\sigma \coloneqq h a_\sigma \, {}^\sigma h^{-1}$.
When $H$ is commutative, we write $H^i(F, H)$ for the usual continuous cohomology of $\Gamma$ with coefficients in $H(\bar{F})$ for $i = 1, 2$.
In degree $1$, this agrees with the preceding definition.
For a noncommutative connected reductive $F$-group $G$, we recall the definition of the second nonabelian Galois cohomology set $H^2(F, G)$ in Section~\ref{subsec:abelianization}.

Let $D$ be a pro-algebraic $F$-group of multiplicative type.
A \emph{Galois gerbe banded by $D$} is an extension
\[
1 \to D(\bar{F}) \to \mathcal{E} \xrightarrow{\pi} \Gamma \to 1
\]
of topological groups such that the subspace topology on $D(\bar{F})$ agrees with its inverse limit topology, $\pi$ admits a continuous set-theoretic section, and $w d w^{-1} = {}^{\pi(w)} d$ for every $d \in D(\bar{F})$ and $w \in \mathcal{E}$.
A continuous section $s \colon \Gamma \to \mathcal{E}$ of $\pi$ is called \emph{normalized} if $s_1 = 1$.
Such a section always exists.
Any continuous section $s$ of $\pi$ induces a homeomorphism $D(\bar{F}) \times \Gamma \to \mathcal{E}$ by $(d, \sigma) \mapsto d s_\sigma$.
In particular, $\pi$ is open and the induced bijection $\mathcal{E} / D(\bar{F}) \to \Gamma$ is a homeomorphism.

Let $G$ be a connected reductive group over $F$, and denote its center by $Z(G)$.
An $F$-subgroup $Z \subseteq G$ is \emph{central} if it is a closed $F$-subgroup scheme contained in $Z(G)$.
Every central subgroup of a connected reductive group is of multiplicative type and of finite type.

We let $\mathcal{E}$ act on $G(\bar{F})$ via $\Gamma$ and write ${}^w g \coloneqq {}^{\pi(w)} g$.
Following \cite[Section~2]{Kot14}, we define an \emph{algebraic $1$-cocycle} of $\mathcal{E}$ in $G$ to be a pair $(x, \nu)$, where $x \colon \mathcal{E} \to G(\bar{F})$ is a continuous map satisfying
\begin{equation}\label{e:cocycle-cond}
x_{w_1 w_2} = x_{w_1} \, {}^{w_1} x_{w_2} \qquad \text{for all } w_1, w_2 \in \mathcal{E},
\end{equation}
and $\nu \colon D_{\bar{F}} \to G_{\bar{F}}$ is an algebraic $\bar{F}$-homomorphism such that $x_d = \nu(d)$ for every $d \in D(\bar{F})$.
We denote by $Z^1(\mathcal{E}, G)$ the set of algebraic $1$-cocycles.
For a central $F$-subgroup $Z$ of $G$, we write $Z^1(\mathcal{E}, Z \to G)$ for the subset of pairs $(x, \nu) \in Z^1(\mathcal{E}, G)$ for which $\nu$ factors through $Z_{\bar{F}}$.
For such a pair, $\nu \colon D_{\bar{F}} \to Z_{\bar{F}}$ is defined over $F$, since for every $d \in D(\bar{F})$ and every lift $w \in \mathcal{E}$ of $\sigma \in \Gamma$, the cocycle condition \eqref{e:cocycle-cond} and the centrality of $\nu(d)$ yield
\[
\nu({}^\sigma d) = x_{w d w^{-1}} = x_w \, {}^\sigma x_d \, x_w^{-1} = {}^\sigma \nu(d).
\]
We define the \emph{basic} subset by $Z^1_{\mathrm{bas}}(\mathcal{E}, G) \coloneqq Z^1(\mathcal{E}, Z(G) \to G)$.

Write $G^{\mathrm{ad}} \coloneqq G / Z(G)$.
For $(x, \nu) \in Z^1(\mathcal{E}, Z \to G)$, define the \emph{adjoint cocycle of $x$} by
\[
\bar{x} \colon \Gamma \longrightarrow G^{\mathrm{ad}}(\bar{F}), \qquad \sigma \longmapsto x_w Z(G)(\bar{F}),
\]
where $w \in \mathcal{E}$ is any lift of $\sigma$ under $\pi$.
The value $\bar{x}_\sigma$ is independent of $w$, since $x_{d w} = \nu(d) \, x_w$ and $\nu(d) \in Z(\bar{F}) \subseteq Z(G)(\bar{F})$ for every $d \in D(\bar{F})$.
The map $\bar{x}$ is continuous since $\pi$ is a quotient map, and \eqref{e:cocycle-cond} gives $\bar{x}_{\sigma \tau} = \bar{x}_\sigma \, {}^\sigma \bar{x}_\tau$.
Thus $\bar{x} \in Z^1(F, G^{\mathrm{ad}})$.

There is a natural left action of $G(\bar{F})$ on $Z^1(\mathcal{E}, G)$ defined by
\begin{equation}\label{eq:act-G-to-Z1}
g \cdot (x, \nu) \coloneqq \bigl(w \mapsto g x_w \, {}^w g^{-1}, \operatorname{inn}(g) \circ \nu\bigr)
\end{equation}
for $g \in G(\bar{F})$ and $(x, \nu) \in Z^1(\mathcal{E}, G)$.
This action preserves the subsets $Z^1(\mathcal{E}, Z \to G)$ and $Z^1_{\mathrm{bas}}(\mathcal{E}, G)$.
Taking $G(\bar{F})$-orbits in $Z^1(\mathcal{E}, G)$, $Z^1(\mathcal{E}, Z \to G)$, and $Z^1_{\mathrm{bas}}(\mathcal{E}, G)$ gives $H^1(\mathcal{E}, G)$, $H^1(\mathcal{E}, Z \to G)$, and $H^1_{\mathrm{bas}}(\mathcal{E}, G)$, respectively.
The set $H^1(\mathcal{E}, G)$ is denoted by $H^1_{\mathrm{alg}}(\mathcal{E}, G)$ in \cite{Kot14}.

\subsection{Comparison of gerbes}
\label{subsec:gerbe-comparison}

Let $\mathcal{E}'$ and $\mathcal{E}$ be Galois gerbes, banded by $D'$ and $D$ respectively.
For an $F$-homomorphism $\phi \colon D' \to D$, a \emph{morphism of gerbes inducing $\phi$} is a continuous homomorphism $f \colon \mathcal{E}' \to \mathcal{E}$ making the diagram
\[
\begin{tikzcd}[row sep=small]
1 \ar[r] & D'(\bar F) \ar[r] \ar[d,"\phi"] & \mathcal{E}' \ar[r,"\pi'"] \ar[d,"f"] & \Gamma \ar[r] \ar[d,equal] & 1 \\
1 \ar[r] & D(\bar F) \ar[r] & \mathcal{E} \ar[r,"\pi"] & \Gamma \ar[r] & 1
\end{tikzcd}
\]
commute.
For gerbes banded by $D$, we consider isomorphisms inducing the identity on $D$ and $\Gamma$.
A continuous section $s$ of $\pi$ gives the continuous $2$-cocycle $c_{\sigma, \tau} = s_\sigma s_\tau s_{\sigma \tau}^{-1}$, whose class is independent of $s$.
Conversely, a normalized continuous $2$-cocycle $c$ defines an extension on $D(\bar{F}) \times \Gamma$ with multiplication $(d, \sigma)(e, \tau) = (d \, {}^\sigma e \, c_{\sigma, \tau}, \sigma \tau)$.
These constructions identify $H^2(F, D)$ with the set of isomorphism classes of such gerbes.
Let $\xi_{\mathcal{E}} \in H^2(F, D)$ denote the class of $\mathcal{E}$.

For an algebraic $F$-homomorphism $\phi \colon D' \to D$, we denote by $\phi_* \colon H^2(F,D') \to H^2(F,D)$ the associated map.

\begin{lemma}\label{lem:existence-comparison}
Let $\phi \colon D' \to D$ be an algebraic $F$-homomorphism.
There exists a morphism $f \colon \mathcal{E}' \to \mathcal{E}$ of gerbes inducing $\phi$ if and only if $\phi_*(\xi_{\mathcal{E}'}) = \xi_{\mathcal{E}} \in H^2(F,D)$.
\end{lemma}
\begin{proof}
Choose normalized continuous sections $s' \colon \Gamma \to \mathcal{E}'$ and $s \colon \Gamma \to \mathcal{E}$ of $\pi'$ and $\pi$ respectively.
The associated $2$-cocycles representing $\xi_{\mathcal{E}'}$ and $\xi_{\mathcal{E}}$ are given by
\[
c'_{\sigma, \tau} \coloneqq s'_\sigma s'_\tau (s'_{\sigma \tau})^{-1} \in D'(\bar{F}), \qquad c_{\sigma, \tau} \coloneqq s_\sigma s_\tau s_{\sigma \tau}^{-1} \in D(\bar{F}).
\]

Suppose that $f \colon \mathcal{E}' \to \mathcal{E}$ is a morphism of gerbes inducing $\phi$.
Since $\pi(f(s'_\sigma)) = \sigma = \pi(s_\sigma)$, the element $b_\sigma \coloneqq f(s'_\sigma) s_\sigma^{-1}$ belongs to $D(\bar{F})$.
The map $b \colon \Gamma \to D(\bar{F})$ is continuous and satisfies $b_1 = 1$.
Every $w' \in \mathcal{E}'$ has a unique expression $w' = d' s'_\sigma$, where $\sigma = \pi'(w')$ and $d' = w' (s'_\sigma)^{-1} \in D'(\bar{F})$.
Since $f$ is a homomorphism and restricts to $\phi$ on $D'(\bar{F})$, we have $f(d' s'_\sigma) = f(d') f(s'_\sigma) = \phi(d') \, b_\sigma s_\sigma$.

Conversely, let $b \colon \Gamma \to D(\bar{F})$ be a normalized continuous map.
For $w' \in \mathcal{E}'$, set $\sigma \coloneqq \pi'(w')$ and define
\[
f \colon \mathcal{E}' \longrightarrow \mathcal{E}, \qquad f(w') \coloneqq \phi\bigl(w' (s'_\sigma)^{-1}\bigr) \, b_\sigma s_\sigma.
\]
The map $w' \mapsto w' (s'_{\pi'(w')})^{-1}$ is continuous with values in $D'(\bar{F})$, hence $f$ is continuous.
We have $\pi(f(w')) = \pi'(w')$ for every $w' \in \mathcal{E}'$, and $f(d') = \phi(d')$ for every $d' \in D'(\bar{F})$ since $s'_1 = s_1 = b_1 = 1$.

For $d', e' \in D'(\bar{F})$ and $\sigma, \tau \in \Gamma$, the multiplication law in $\mathcal{E}'$ yields $(d' s'_\sigma)(e' s'_\tau) = d' \, {}^\sigma e' \, c'_{\sigma, \tau} s'_{\sigma \tau}$.
Since $\phi$ is defined over $F$ and $D(\bar{F})$ is abelian, we obtain
\[
f\bigl((d' s'_\sigma)(e' s'_\tau)\bigr) = \phi(d') \, {}^\sigma \phi(e') \, \phi(c'_{\sigma, \tau}) \, b_{\sigma \tau} s_{\sigma \tau}, \qquad
f(d' s'_\sigma) f(e' s'_\tau) = \phi(d') \, {}^\sigma \phi(e') \, b_\sigma \, {}^\sigma b_\tau \, c_{\sigma, \tau} s_{\sigma \tau}.
\]
Thus $f$ is a homomorphism if and only if
\[
\phi(c'_{\sigma, \tau}) \, c_{\sigma, \tau}^{-1} = b_\sigma \, {}^\sigma b_\tau \, b_{\sigma \tau}^{-1}
\qquad \text{for all } \sigma, \tau \in \Gamma.
\]
A continuous map $b$ satisfying this identity exists if and only if $\phi_*(\xi_{\mathcal{E}'}) = \xi_{\mathcal{E}}$ in $H^2(F, D)$, since the right-hand side is the coboundary of $b$.
Any such $b$ satisfies $b_1 = 1$, as follows by setting $\sigma = \tau = 1$.
\end{proof}

Fix an $F$-homomorphism $\phi \colon D' \to D$ and a morphism $f \colon \mathcal{E}' \to \mathcal{E}$ of gerbes inducing $\phi$.
Pullback induces a map of pointed sets
\[
f^* \colon H^1(\mathcal{E},G) \longrightarrow H^1(\mathcal{E}',G),
\qquad
[x,\nu] \longmapsto [x \circ f,\nu \circ \phi_{\bar F}].
\]
We describe explicitly how pullback depends on the choice of $f$, following \cite[Section~3.1]{KT23}.

\begin{lemma}
\label{lem:comparison-dependence}
Let $f \colon \mathcal{E}' \to \mathcal{E}$ be a morphism of gerbes inducing $\phi$ on the bands.
\begin{enumerate}
\item
Every morphism $f' \colon \mathcal{E}' \to \mathcal{E}$ inducing $\phi$ is given by
\[
f' = a_\alpha \circ f,
\qquad
a_\alpha(w) = \alpha_{\pi(w)} w,
\]
for a unique $\alpha \in Z^1(F,D)$.
In particular, the set of morphisms $\mathcal{E}' \to \mathcal{E}$ of gerbes inducing
$\phi$ is a torsor under $Z^1(F,D)$, and modulo conjugation by $D(\bar F)$, these morphisms form a torsor under $H^1(F,D)$.

\item
Pullback along $a_\alpha$ defines an action
$[\alpha] \cdot [x,\nu] \coloneqq [x \circ a_\alpha,\nu]$
of $H^1(F,D)$ on $H^1(\mathcal{E},G)$.
If $f' = a_\alpha \circ f$, then
\[
(f')^*[x,\nu]
= f^*\bigl([\alpha] \cdot [x,\nu]\bigr).
\]
The action preserves $H^1(\mathcal{E},Z \to G)$ and
$H^1_{\mathrm{bas}}(\mathcal{E},G)$, and pullback maps
these subsets to $H^1(\mathcal{E}',Z \to G)$ and
$H^1_{\mathrm{bas}}(\mathcal{E}',G)$ respectively.

\item
If $H^1(F,D) = 1$, the map $f^*$ depends only on $\phi$.
In particular, the sets $H^1(\mathcal{E}, G)$ for gerbes $\mathcal{E}$ representing the same class in $H^2(F, D)$ are canonically identified.
The same statements hold for
$H^1(\mathcal{E},Z \to G)$ and
$H^1_{\mathrm{bas}}(\mathcal{E},G)$.
\end{enumerate}
\end{lemma}

\begin{proof}
Let $f' \colon \mathcal{E}' \to \mathcal{E}$ be another morphism of gerbes inducing $\phi$.
For $\sigma \in \Gamma$, choose a lift $w' \in \mathcal{E}'$ of $\sigma \in \Gamma$ and let
\[
\alpha_\sigma = f'(w') f(w')^{-1} \in D(\bar F).
\]
Since $f$ and $f'$ agree on $D'(\bar F)$, this expression is independent of the lift.
Evaluating along a continuous section of $\pi'$ shows that $\alpha$ is continuous.
For lifts $w'_1,w'_2$ of $\sigma,\tau$ respectively, we obtain
\[
\alpha_{\sigma \tau}
= f'(w'_1 w'_2) f(w'_1 w'_2)^{-1}
= \alpha_\sigma f(w'_1) \alpha_\tau f(w'_1)^{-1} = \alpha_\sigma \, {}^\sigma \alpha_\tau,
\]
hence $\alpha \in Z^1(F,D)$.
Conversely, for every $\alpha \in Z^1(F, D)$, the identity $\alpha_{\sigma \tau} = \alpha_\sigma \, {}^\sigma \alpha_\tau$ shows that $w' \mapsto \alpha_{\pi'(w')} f(w')$ defines a morphism of gerbes inducing $\phi$.

For $\alpha \in Z^1(F,D)$, define
\[
a_\alpha \colon \mathcal{E} \longrightarrow \mathcal{E},
\qquad
a_\alpha(w) = \alpha_{\pi(w)} w.
\]
These are precisely the automorphisms fixing both the band $D$ and $\Gamma$, and $f' = a_\alpha \circ f$.
Let $\alpha$ and $\beta$ be continuous cocycles.
Suppose that $[\alpha] = [\beta] \in H^1(F,D)$, that is $\beta_\sigma = d^{-1} \alpha_\sigma \, {}^\sigma d$ for some $d \in D(\bar F)$.
For a lift $w \in \mathcal{E}$ of $\sigma \in \Gamma$, we have
\[
a_\beta(w)
= \beta_\sigma w
= d^{-1} \alpha_\sigma {}^\sigma d w
= d^{-1} \alpha_\sigma w d
= \bigl(\operatorname{inn}(d^{-1}) \circ a_\alpha\bigr)(w),
\]
which implies that $a_\beta \circ f = \operatorname{inn}(d^{-1}) \circ a_\alpha \circ f$.
Conversely, if $a_\beta \circ f = \operatorname{inn}(d^{-1}) \circ a_\alpha \circ f$, evaluation at a lift of $\sigma$ yields $\beta_\sigma = d^{-1} \alpha_\sigma \, {}^\sigma d$.
This completes the proof of (1).

For an algebraic cocycle $(x,\nu)$ and $d \in D(\bar F)$, the cocycle condition \eqref{e:cocycle-cond} gives
\[
x_{d^{-1} w d}
= \nu(d)^{-1} x_w \, {}^{\pi(w)} \nu(d).
\]
Since $\nu(d)$ centralizes the image of $\nu$, we have $[x \circ \operatorname{inn}(d^{-1}),\nu] = [x,\nu] \in H^1(\mathcal{E},G)$.
Hence pullback along $\operatorname{inn}(d^{-1})$ acts trivially on $H^1(\mathcal{E},G)$, and
\[
[\alpha] \cdot [x,\nu]
 \coloneqq [x \circ a_\alpha,\nu]
\]
depends only on $[\alpha] \in H^1(F,D)$.
It defines an action since $a_{\alpha \beta} = a_\alpha \circ a_\beta = a_\beta \circ a_\alpha$.

The cocycle condition \eqref{e:cocycle-cond} gives
$(x \circ a_\alpha)_w = x_{\alpha_{\pi(w)} w} = \nu\bigl(\alpha_{\pi(w)}\bigr) x_w$.
For $f' = a_\alpha \circ f$, we obtain
\[
(f')^*[x,\nu]
= [x \circ a_\alpha \circ f,\nu \circ \phi_{\bar F}]
= f^*\bigl([\alpha] \cdot [x,\nu]\bigr).
\]
The action fixes $\nu$, while pullback replaces it by $\nu \circ \phi_{\bar F}$.
Both therefore preserve the subsets $H^1(\mathcal{E},Z \to G)$ and $H^1_{\mathrm{bas}}(\mathcal{E},G)$, and pullback maps them to the corresponding subsets for $\mathcal{E}'$.
This proves (2).

If $H^1(F,D) = 1$, the action is trivial, proving independence of the lift $f$.
For two extensions representing the same class in $H^2(F,D)$, choose an isomorphism inducing the identity on the band and $\Gamma$.
Its pullback is independent of this choice, and these identifications are compatible with composition.
This proves (3).
\end{proof}

For an $F$-homomorphism $\varphi \colon D' \to A$ and $\xi = [c] \in H^2(F,D')$, we write $\varphi_*(\xi) = [\varphi \circ c] \in H^2(F,A)$.
The following proposition gives a criterion for the existence of a morphism between gerbes.

\begin{proposition}\label{prop:comparison-exist-finite-type}
Let $\mathcal{E}'$ be a Galois gerbe banded by $D'$.
The following conditions are equivalent:
\begin{enumerate}
\item
For every finite $F$-group $A$ of multiplicative type,
the map
\[
\operatorname{Hom}_F(D',A) \longrightarrow H^2(F,A),
\qquad \varphi \longmapsto \varphi_*(\xi_{\mathcal{E}'})
\]
is surjective.

\item
Every Galois gerbe $\mathcal{E}$ banded by an $F$-group $D$ of multiplicative type and of finite type admits a morphism $\mathcal{E}' \to \mathcal{E}$ of gerbes inducing an algebraic $F$-homomorphism on bands.

\item
Every Galois gerbe $\mathcal{E}$ banded by an $F$-group $D$ of multiplicative type and of finite type admits a morphism $\mathcal{E}' \to \mathcal{E}$ inducing an algebraic $F$-homomorphism on bands, which factors through a finite $F$-subgroup of $D$.
\end{enumerate}
\end{proposition}
\begin{proof}
The implication $(3) \Rightarrow (2)$ is obvious.
We prove the implications $(1) \Rightarrow (3)$ and $(2) \Rightarrow (1)$.

Assume~(1), and let $\mathcal{E}$ be banded by an algebraic $F$-group $D$ of multiplicative type.
Let $\xi = \xi_{\mathcal{E}}$ and $T = D^\circ$.
Choose a positive integer $m$ annihilating $D/T$, so that multiplication by $m$ induces a surjective homomorphism $[m] \colon D \to T$.
The group $H^2(F,T)$ is torsion; see \cite[Chapter~II, Corollary~4.3]{Mil20}.
Hence we may choose a positive integer $n$ such that $n[m]_*(\xi) = 0$.

Consider the exact sequence
$1 \to A \xrightarrow{i} D
\xrightarrow{[m n]} T \to 1$
where $A = D[m n]$ is a finite $F$-group of multiplicative type.
Since $[m n]_*(\xi) = n[m]_*(\xi) = 0$, the induced exact sequence
\[
  H^2(F,A) \xrightarrow{i_*} H^2(F,D)
\xrightarrow{[m n]_*} H^2(F,T)
\]
gives a class $\xi_A \in H^2(F,A)$ with $i_*(\xi_A) = \xi$.
By~(1), there is an $F$-homomorphism $\varphi \colon D' \to A$ satisfying $\varphi_*(\xi_{\mathcal{E}'}) = \xi_A$.
The composite $i \circ \varphi$ therefore satisfies
\[
(i \circ \varphi)_*(\xi_{\mathcal{E}'})
= i_*(\xi_A)
= \xi_{\mathcal{E}}.
\]
Lemma~\ref{lem:existence-comparison} gives a morphism $\mathcal{E}' \to \mathcal{E}$, whose band map $i \circ \varphi$ factors through the finite subgroup $A$.
This proves~(3).

Assume~(2).
Let $A$ be a finite $F$-group of multiplicative type,
and let $\eta \in H^2(F,A)$.
Choose a Galois gerbe $\mathcal{E}_\eta$ banded by $A$, whose isomorphism class corresponds to $\eta$.
By~(2), there is a morphism $f \colon \mathcal{E}' \to \mathcal{E}_\eta$ of gerbes inducing an $F$-homomorphism $\varphi \colon D' \to A$ on the bands.
Then Lemma~\ref{lem:existence-comparison} implies that
$\varphi_*(\xi_{\mathcal{E}'})
=
\xi_{\mathcal{E}_\eta}
=
\eta$.
Since $\eta$ is arbitrary, the map
$\operatorname{Hom}_F(D',A) \to H^2(F,A)$ is surjective.
This proves~(1).
\end{proof}

\begin{remark}\label{rmk:Kaletha-gerbe_surj}
  \begin{enumerate}
    \item Kaletha's local and global rigid gerbes (which will be reviewed in Sections~\ref{sec:local-formula} and~\ref{sec:number-field}) satisfy the assumption of Lemma~\ref{lem:comparison-dependence}(3).
  For the local band $u$, one has $H^1(F,u) = 0$ by
  \cite[Theorem~3.1]{Kal16}; for $F = \mathbb{C}$, this follows from $u = 1$.
  For the global band $P_{\dot V}$, one has
  $H^1(F,P_{\dot V}) = 0$ by \cite[Proposition~3.4.6]{Kal18a}.
Thus Lemma~\ref{lem:comparison-dependence}(3) shows that, in both cases, the cohomology sets associated with the gerbes in the isomorphism class are canonically identified.

  \item Kaletha's local and global rigid gerbes satisfy the equivalent conditions of Proposition~\ref{prop:comparison-exist-finite-type}.
For the local gerbe, condition~(1) is \cite[Proposition~3.2]{Kal16}, where the case of $\mathbb{C}$ follows from $u = 1$ and $H^2(\mathbb{C},A) = 0$.
For the global gerbe, the surjectivity of the map
$\Theta_{P_{\dot V}}$ in \cite[Section~3.3, (3.10)]{Kal18a} follows from \cite[Corollary~3.2.9]{Kal18a}.
The paragraph following (3.10) identifies $\Theta_{P_{\dot V}}$ with evaluation on any lift of the distinguished finite-level classes, and hence on the canonical global rigid gerbe class.
Thus condition~(1) holds in the global case as well.
  \end{enumerate}

\end{remark}

\begin{corollary}\label{cor:comparison-exist-finite-type}
Let $F$ be a local field of characteristic zero or a number field, and let $\mathcal{E}$ be a Galois gerbe banded by an $F$-group $D$ of finite type and of multiplicative type.
If $F$ is local, let $\mathcal{E}^{\mathrm{rig}}$ be Kaletha's local rigid gerbe \cite[Sections~3.1--3.2]{Kal16}, with $\mathcal{E}^{\mathrm{rig}} = 1$ when $F = \mathbb{C}$.
If $F$ is a number field, fix an admissible set of lifts $\dot V$ and let $\mathcal{E}^{\mathrm{rig}}$ denote Kaletha's global rigid gerbe $\mathcal{E}^{\mathrm{rig}}_{\dot V}$ \cite[Sections~3.5--3.6]{Kal18a}.
Then there exists a morphism of gerbes
\[
f \colon \mathcal{E}^{\mathrm{rig}} \longrightarrow \mathcal{E}
\]
such that the induced $F$-homomorphism on bands factors through a finite $F$-subgroup of $D$.
\end{corollary}

\begin{proof}
By Remark~\ref{rmk:Kaletha-gerbe_surj}(2), Proposition~\ref{prop:comparison-exist-finite-type} applies to $\mathcal{E}^{\mathrm{rig}}$ and yields the desired morphism.
\end{proof}

\begin{lemma}[{cf.\ \cite[Fact~3.1.3]{KT23}}]
\label{lem:cartesian}
Let $f \colon \mathcal{E}' \to \mathcal{E}$ be a morphism of gerbes with band map $\phi \colon D' \to D$ defined over $F$.
For a connected reductive $F$-group $G$ and a central $F$-subgroup $Z \subseteq G$, the map
\begin{align*}
H^1(\mathcal{E},Z \to G)
\longrightarrow
H^1(\mathcal{E}',Z \to G)
\times_{\operatorname{Hom}_F(D',Z)} \operatorname{Hom}_F(D,Z), \qquad
[x,\nu] \longmapsto ([x \circ f,\nu \circ \phi],\nu) 
\end{align*}
is a bijection, natural in $(Z \to G)$.
The maps to the base are restriction to $D'$ and $\nu \mapsto \nu \circ \phi$ respectively.
\end{lemma}
\begin{proof}
Let $([x', \nu \circ \phi], \nu)$ belong to the fiber product, and choose a representative $(x', \nu \circ \phi)$.
For $w \in \mathcal{E}$, choose $w' \in \mathcal{E}'$ such that $\pi'(w') = \pi(w)$.
Since $f$ induces the identity on $\Gamma$, we have $d \coloneqq w f(w')^{-1} \in D(\bar{F})$.
Thus $w = d f(w')$, and we define
\[
x_w \coloneqq \nu(d) x'_{w'}.
\]

Let $w'' \in \mathcal{E}'$ be another lift of $\pi(w)$ and set $e' \coloneqq w'' (w')^{-1} \in D'(\bar{F})$.
Then $w'' = e' w'$ and $w = d \phi(e')^{-1} f(w'')$.
Since $x'_{w''} = x'_{e' w'} = \nu(\phi(e')) x'_{w'}$, we obtain
$\nu(d \phi(e')^{-1}) x'_{w''} = \nu(d) x'_{w'}$.
Thus $x_w$ is independent of the choice of $w'$.
Choosing $w' = s'_{\pi(w)}$ for a continuous section $s'$ of $\pi'$ shows that both $w'$ and $d = w f(w')^{-1}$ depend continuously on $w$, so $x$ is continuous.

To prove the cocycle identity, write $w_i = d_i f(w'_i)$ for $i = 1, 2$ and set $\sigma \coloneqq \pi(w_1)$.
Then $w_1 w_2 = d_1 \, {}^\sigma d_2 \, f(w'_1 w'_2)$.
Since $\nu$ is defined over $F$ and has central image, the cocycle identity for $x'$ yields
\[
x_{w_1 w_2}
= \nu(d_1 \, {}^\sigma d_2) x'_{w'_1} \, {}^\sigma x'_{w'_2}
= \bigl(\nu(d_1) x'_{w'_1}\bigr) \, {}^\sigma\bigl(\nu(d_2) x'_{w'_2}\bigr)
= x_{w_1} \, {}^\sigma x_{w_2}.
\]
Taking $w' = 1$ gives $x_d = \nu(d)$ for every $d \in D(\bar{F})$, and taking $d = 1$ gives $x_{f(w')} = x'_{w'}$ for every $w' \in \mathcal{E}'$.
Conversely, these two identities force $x_{d f(w')} = \nu(d) x'_{w'}$, so $x$ is unique.
Since $\nu$ has central image, replacing $x'$ by $g \cdot x'$ replaces $x$ by $g \cdot x$ for every $g \in G(\bar{F})$.
Hence $([x', \nu \circ \phi], \nu) \mapsto [x, \nu]$ is well defined and gives the inverse on cohomology classes.

Finally, for a morphism $h \colon (Z \to G) \to (Z' \to G')$, applying $h$ to the defining formula gives $h(x_{d f(w')}) = (h \circ \nu)(d) \, (h \circ x')_{w'}$.
This proves naturality.
\end{proof}

\subsection{Abelianization and the Newton map}\label{subsec:abelianization}\label{subsec:general-newton-fibers}

Let $H$ be a connected reductive $F$-group.
Write $\rho_H \colon H^{\mathrm{sc}} \to H$ for the natural map from the simply connected cover of its derived group.
Following \cite[Definition~2.2 and Lemma~2.4.1]{Bor98}, we define the \emph{first abelian Galois cohomology group} by
\[
H^1_{\mathrm{ab}}(F,H) \coloneqq \mathbb{H}^1(F,[Z(H^{\mathrm{sc}}) \xrightarrow{\rho_H} Z(H)]),
\]
where the complex is placed in degrees $-1$ and $0$.
We denote the ordinary abelianization map by $\operatorname{ab}_H^1 \colon H^1(F,H) \to H^1_{\mathrm{ab}}(F,H)$; see \cite[Section~3.10]{Bor98}.
By \cite[Corollary~3.17(i)]{Bor98}, there is an exact sequence of pointed sets
\begin{equation}\label{eq:ab-exact-seq}
H^1(F,H^{\mathrm{sc}}) \xrightarrow{\rho_{H,*}} H^1(F,H) \xrightarrow{\operatorname{ab}_H^1} H^1_{\mathrm{ab}}(F,H).
\end{equation}

We recall the constructions of \cite[Sections~2.4 and~2.6--2.7]{Kot14}.
There is a well-defined map
\[
\mathcal{N} \colon H^1(\mathcal{E},Z \to G) \longrightarrow \operatorname{Hom}_F(D,Z),\qquad [x,\nu] \longmapsto \nu,
\]
which is called the \textit{Newton map}.
We write $1$ for the trivial homomorphism $D \to Z$.

For $(x, \nu) \in Z^1(\mathcal{E}, Z \to G)$, let $G_x$ be the inner form of $G$ obtained by twisting by the adjoint cocycle $\bar{x}$.
Under the identification $G_x(\bar{F}) = G(\bar{F})$, the action of $\sigma \in \Gamma$ on $g$ is given by $x_w \, {}^\sigma g \, x_w^{-1}$ for any $w \in \mathcal{E}$ with $\pi(w) = \sigma$.
By \cite[Section~2.7]{Kot14}, there is a bijection
\begin{equation}\label{eq:twisting-restriction-fiber}
t_{(x,\nu)} \colon H^1(F,G_x) \xrightarrow{\sim} \mathcal{N}^{-1}(\nu),\qquad [a] \longmapsto [(a \circ \pi) \cdot x,\nu],
\end{equation}
where $((a \circ \pi) \cdot x)_w = a_{\pi(w)} x_w$.
Its inverse sends $[y,\nu]$ to the class of the cocycle $a_\sigma = y_w x_w^{-1}$, where $\pi(w) = \sigma$.
In particular, $t_{(x,\nu)}$ sends the trivial class to $[x,\nu]$.

We define the \emph{inflation map} as the composite
\[
\operatorname{inf} \colon H^1(F,G)
\xrightarrow[\sim]{t_{(1,1)}} \mathcal{N}^{-1}(1)
\hookrightarrow H^1(\mathcal{E},Z \to G),
\qquad
[a] \longmapsto [a \circ \pi,1].
\]

We recall the second nonabelian Galois cohomology set, following \cite[Section~1]{FSS98}.
Let $\operatorname{SAut}(G_{\bar{F}})$ be the group of semilinear automorphisms of $G_{\bar{F}}$, and let $\operatorname{SOut}(G_{\bar{F}})$ be its quotient by the subgroup of inner automorphisms.
Since $G$ is defined over $F$, the natural Galois action defines a continuous homomorphism $s_G \colon \Gamma \to \operatorname{SAut}(G_{\bar{F}})$ that splits the projection to $\Gamma$.
Composing $s_G$ with the quotient map gives the $F$-band $\beta_G \colon \Gamma \to \operatorname{SOut}(G_{\bar{F}})$ associated with $G$ \cite[Definition~(1.11) and Remark~(1.15)]{FSS98}.

A $2$-cocycle with band $\beta_G$ is a pair $(u, f)$ consisting of a locally constant map $u \colon \Gamma \times \Gamma \to G(\bar{F})$ and a continuous lift $f \colon \Gamma \to \operatorname{SAut}(G_{\bar{F}})$ of $\beta_G$ such that
\[
\operatorname{inn}(u_{\sigma, \tau}) \circ f_\sigma \circ f_\tau = f_{\sigma \tau},
\qquad
u_{\sigma, \tau \upsilon} \, f_\sigma(u_{\tau, \upsilon}) = u_{\sigma \tau, \upsilon} \, u_{\sigma, \tau}
\]
for all $\sigma, \tau, \upsilon \in \Gamma$.
Here $\operatorname{inn}(g)$ denotes conjugation by $g$, and continuity of $f$ is understood in the sense of \cite[Definition~(1.10)]{FSS98}.
We use the convention of \cite[Section~4]{BK26b}, so our pair $(u, f)$ corresponds to $(f, u^{-1})$ in \cite[(1.17)--(1.18)]{FSS98}.

Two such cocycles $(u, f)$ and $(u', f')$ are equivalent if there is a continuous map $b \colon \Gamma \to G(\bar{F})$ such that
\[
u'_{\sigma, \tau} = b_{\sigma \tau} \, u_{\sigma, \tau} \, f_\sigma(b_\tau)^{-1} \, b_\sigma^{-1},
\qquad
f'_\sigma = \operatorname{inn}(b_\sigma) \circ f_\sigma
\]
for all $\sigma, \tau \in \Gamma$.
We denote the set of equivalence classes by $H^2(F, G) \coloneqq H^2(F, G_{\bar{F}}, \beta_G)$, with distinguished class $[1, s_G]$.
A class is \emph{neutral} if it has a representative of the form $(1, f)$ \cite[Definition~(1.25)]{FSS98}.
When $G$ is commutative, the map $[u, s_G] \mapsto [u]$ identifies $H^2(F, G)$ with the usual second Galois cohomology group, and the trivial class is the only neutral class.

Choose a continuous $2$-cocycle $c$ representing the class $\xi_{\mathcal{E}} \in H^2(F, D)$ of $\mathcal{E}$.
For $\nu \in \operatorname{Hom}_F(D,Z)$, we define
\[
\delta_{\mathcal{E}}(\nu) \coloneqq [\nu \circ c,s_G] \in H^2(F,G),
\]
where $\nu$ is viewed as a homomorphism into $G$ through the inclusion $Z \hookrightarrow G$.
Since $\nu$ is defined over $F$ and has central image, $(\nu \circ c,s_G)$ is a $2$-cocycle with band $\beta_G$.
Its class is independent of the representative $c$, by the preceding equivalence relation.

\begin{proposition}\label{prop:gerbe-inflation-restriction}
There is an exact sequence
\[
1 \longrightarrow H^1(F,G) \xrightarrow{\operatorname{inf}} H^1(\mathcal{E},Z \to G) \xrightarrow{\mathcal{N}} \operatorname{Hom}_F(D,Z) \xrightarrow{\delta_{\mathcal{E}}} H^2(F,G),
\]
where exactness at $\operatorname{Hom}_F(D,Z)$ means that the image of $\mathcal{N}$ coincides with the preimage of the subset of neutral classes in $H^2(F,G)$ under $\delta_\mathcal{E}$.
If $G$ is commutative, this is the exact sequence of abelian groups in \cite[Fact~3.1.1]{KT23}.

In particular, for $\nu \in \operatorname{Hom}_F(D,Z)$, the fiber $\mathcal{N}^{-1}(\nu)$ is nonempty if and only if $\delta_{\mathcal{E}}(\nu)$ is neutral.
\end{proposition}

\begin{proof}
By definition, inflation identifies $H^1(F, G)$ with $\mathcal{N}^{-1}(1)$, as in \cite[Fact~3.1.1]{KT23}.
It remains to prove exactness at $\operatorname{Hom}_F(D,Z)$.

Choose a continuous section $s \colon \Gamma \to \mathcal{E}$ normalized by $s_1 = 1$ and let $c_{\sigma,\tau} = s_\sigma s_\tau s_{\sigma \tau}^{-1}$.
Then $c$ is a 2-cocycle representing $\xi_\mathcal{E}$, and the class $\delta_\mathcal{E}(\nu) = [\nu \circ c,s_G]$ is neutral if and only if there is a continuous map $b \colon \Gamma \to G(\bar F)$ satisfying
\begin{equation}\label{eq:gerbe-neutrality}
b_{\sigma \tau} \nu(c_{\sigma,\tau}) \, {}^\sigma b_\tau^{-1} b_\sigma^{-1} = 1.
\end{equation}

If $(x,\nu) \in Z^1(\mathcal{E},Z \to G)$, then $b_\sigma = x_{s_\sigma}$ satisfies \eqref{eq:gerbe-neutrality}, since $\nu(c_{\sigma,\tau})$ is central and we have
\[
  \nu(c_{\sigma,\tau}) b_{\sigma \tau}
  = \nu(c_{\sigma,\tau}) x_{s_{\sigma \tau}}
  = x_{c_{\sigma,\tau} s_{\sigma \tau}}
  = x_{s_\sigma s_\tau}
  = x_{s_\sigma} {}^\sigma x_{s_\tau}
  = b_\sigma {}^\sigma b_\tau
\]
by the cocycle condition \eqref{e:cocycle-cond}.
Thus the image of $[x,\nu] \in H^1(\mathcal{E},Z \to G)$ under $\delta_\mathcal{E} \circ \mathcal{N}$ is neutral in $H^2(F,G)$.
Conversely, suppose that $\delta_{\mathcal{E}}(\nu)$ is neutral, so that there exists $b$ satisfying \eqref{eq:gerbe-neutrality}.
Since $c_{1,1} = 1$, we have $b_1 = 1$.
If $w \in \mathcal{E}$ satisfies $\pi(w) = \sigma$, then there is a unique $d \in D(\bar F)$ such that $w = d s_\sigma$.
Define a continuous map $x \colon \mathcal{E} \to G(\bar F)$ by
\[
x_{d s_\sigma} \coloneqq \nu(d) b_\sigma,\qquad d \in D(\bar F),\quad \sigma \in \Gamma.
\]
For $d,d' \in D(\bar F)$ and $\sigma,\tau \in \Gamma$, since $\nu$ is defined over $F$ and its image is in $Z(G)$, we have
\[
x_{d s_\sigma} \, {}^\sigma x_{d' s_\tau}
= \nu(d \, {}^\sigma d') b_\sigma \, {}^\sigma b_\tau
= \nu(d \, {}^\sigma d' \, c_{\sigma,\tau}) b_{\sigma \tau}
= x_{(d s_\sigma)(d' s_\tau)},
\]
hence $x$ is a $1$-cocycle.
Its restriction to $D(\bar F)$ is $\nu$, since $s_1 = b_1 = 1$.
Thus $(x,\nu) \in Z^1(\mathcal{E},Z \to G)$, proving the assertion.

When $G$ is commutative, inflation and restriction to $D$ preserve pointwise multiplication of cocycles, and $\delta_{\mathcal{E}}$ sends $\nu$ to the class of $\nu \circ c$.
Thus all the arrows in the sequence are homomorphisms, and the only neutral class in $H^2(F, G)$ is the trivial class.
\end{proof}

We extend to $H^1(\mathcal{E}, Z \to G)$ the equivalence relation used by Kaletha for the local and global rigid gerbes in \cite[Section~3.4]{Kal16} and \cite[Section~3.8]{Kal18a} respectively.
The classes $[x,\nu],[y,\nu'] \in H^1(\mathcal{E},Z \to G)$ are equivalent if $\nu = \nu'$ and
\begin{equation}\label{eq:abelianization-relation}
[y,\nu] = t_{(x,\nu)}\bigl(\rho_{G_x,*}[a]\bigr) = [(\rho_{G_x} \circ a \circ \pi) \cdot x,\nu] \quad \text{for some } a \in Z^1(F,G_x^{\mathrm{sc}}).
\end{equation}

By \eqref{eq:ab-exact-seq}, the condition \eqref{eq:abelianization-relation} is equivalent to
\begin{equation}\label{eq:abelianization-relation-2}
\operatorname{ab}_{G_x}^1\bigl(t_{(x,\nu)}^{-1}[y,\nu]\bigr) = 0.
\end{equation}
For a fixed $(x,\nu)$ and classes $[a],[b] \in H^1(F,G_x)$, \cite[Corollary~3.17(ii)]{Bor98}, applied to $G_x$, shows that $t_{(x,\nu)}[a]$ and $t_{(x,\nu)}[b]$ are related by \eqref{eq:abelianization-relation} if and only if
$\operatorname{ab}_{G_x}^1[a] = \operatorname{ab}_{G_x}^1[b]$.
Hence \eqref{eq:abelianization-relation} defines an equivalence relation, independent of representatives and preserving $\nu$.

We denote the set of equivalence classes under \eqref{eq:abelianization-relation} by $H^1_{\mathrm{ab}}(\mathcal{E},Z \to G)$ and call it the \textit{abelian central cohomology}.
The natural surjective map
\begin{equation}\label{e:*}
\operatorname{ab} \colon H^1(\mathcal{E},Z \to G) \longrightarrow H^1_{\mathrm{ab}}(\mathcal{E},Z \to G)
\end{equation}
is called the \textit{abelianization}.
Its fibers are given by
\begin{equation}\label{eq:abelianization-fiber}
\operatorname{ab}^{-1}(\operatorname{ab}[x, \nu])
= t_{(x, \nu)} \! \left(\operatorname{im}\left[H^1(F, G_x^{\mathrm{sc}}) \to H^1(F, G_x)\right]\right)
= t_{(x, \nu)} \! \left(\operatorname{ker}\left[H^1(F, G_x) \to H^1_{\mathrm{ab}}(F, G_x)\right]\right).
\end{equation}
In particular, the fiber over $\operatorname{ab}[1,1]$ is the image of $H^1(F,G^{\mathrm{sc}})$ under the composite of $\rho_{G,*}$ and inflation.
The Newton map descends to
\begin{equation}\label{eq:Newton-ab}
\mathcal{N}_{\mathrm{ab}} \colon H^1_{\mathrm{ab}}(\mathcal{E},Z \to G) \longrightarrow \operatorname{Hom}_F(D,Z),\qquad \operatorname{ab}[x,\nu] \longmapsto \nu,
\end{equation}
with $\mathcal{N} = \mathcal{N}_{\mathrm{ab}} \circ \operatorname{ab}$.
We call $\mathcal{N}_{\mathrm{ab}}$ the \textit{abelianized Newton map}.

The sets $H^1(\mathcal{E}, Z \to G)$ and $H^1_{\mathrm{ab}}(\mathcal{E}, Z \to G)$, the abelianization map, and the two Newton maps are functorial in $(Z \to G)$.
By a morphism $(Z \to G) \to (Z' \to G')$ we mean an $F$-homomorphism $h \colon G \to G'$ satisfying $h(Z) \subseteq Z'$.
For $(x, \nu) \in Z^1(\mathcal{E}, Z \to G)$, the map $h$ intertwines the Galois actions defined by $\bar{x}$ and $\overline{h \circ x}$, since $h(x_w \, {}^\sigma g \, x_w^{-1}) = h(x_w) \, {}^\sigma h(g) \, h(x_w)^{-1}$ whenever $\pi(w) = \sigma$.
It therefore induces an $F$-homomorphism $h_x \colon G_x \to G'_{h \circ x}$ satisfying
\[
h_*\bigl(t_{(x,\nu)}[a]\bigr) = t_{(h \circ x,h \circ \nu)}\bigl(h_{x,*}[a]\bigr),\qquad [a] \in H^1(F,G_x).
\]
The functoriality of ordinary abelianization $\operatorname{ab}^1$ in \cite[Proposition~3.11]{Bor98} shows that $h$ preserves \eqref{eq:abelianization-relation-2}.
It therefore induces the map $\operatorname{ab}[x,\nu] \mapsto \operatorname{ab}[h \circ x,h \circ \nu]$ on abelian central cohomology.

\begin{corollary}\label{cor:abelianization-bijective}
If $H^1(F,G_x^{\mathrm{sc}}) = 1$ for every $(x,\nu) \in Z^1(\mathcal{E},Z \to G)$, then $\operatorname{ab}$ is bijective.
This holds when $F$ is a non-Archimedean local field, a totally imaginary number field, or $\mathbb{C}$.
It also holds over any $F$ when $G$ is a torus.
\end{corollary}

\begin{proof}
The first assertion follows from \eqref{eq:abelianization-fiber}.
Each $G_x^{\mathrm{sc}}$ is semisimple and simply connected.
Its first cohomology is trivial over a non-Archimedean local field by the Kneser--Bruhat--Tits theorem, and over a totally imaginary number field by the Kneser--Harder--Chernousov Hasse principle and vanishing at every completion.
Over $\mathbb{C}$, ordinary Galois cohomology is trivial.
If $G$ is a torus, then $G_x^{\mathrm{sc}} = 1$.
\end{proof}

We prove the analogue of Lemma~\ref{lem:cartesian} for abelian central cohomology.

\begin{proposition}\label{prop:comparison-formal-ab}
Let $f \colon \mathcal{E}' \to \mathcal{E}$ be a morphism of gerbes inducing an $F$-homomorphism $\phi \colon D' \to D$ on the bands.
There is a bijection of pointed sets, natural in $(Z \to G)$,
\begin{equation}\label{eq:comparison-formal-ab}
H^1_{\mathrm{ab}}(\mathcal{E},Z \to G) \xrightarrow{\sim} H^1_{\mathrm{ab}}(\mathcal{E}',Z \to G) \underset{\operatorname{Hom}_F(D',Z)}{\times} \operatorname{Hom}_F(D,Z),
\end{equation}
given by $\operatorname{ab}[x,\nu] \mapsto (\operatorname{ab}[x \circ f,\nu \circ \phi],\nu)$.
The maps to the base are the Newton map $\mathcal{N}'_{\mathrm{ab}}$ for $\mathcal{E}'$ and $\nu \mapsto \nu \circ \phi$ respectively.
Moreover, abelianization and pullback induce a bijection of pointed sets, natural in $(Z \to G)$,
\begin{equation}\label{eq:comparison-full-recovery}
H^1(\mathcal{E},Z \to G) \xrightarrow{\sim} H^1_{\mathrm{ab}}(\mathcal{E},Z \to G) \underset{H^1_{\mathrm{ab}}(\mathcal{E}',Z \to G)}{\times} H^1(\mathcal{E}',Z \to G),
\end{equation}
given by $[x,\nu] \mapsto (\operatorname{ab}[x,\nu],f^*[x,\nu])$.
The maps to the base are $f^*$ and abelianization respectively.
\end{proposition}

\begin{proof}
Let $\mathcal{N}$ and $\mathcal{N}'$ be the Newton maps for $\mathcal{E}$ and $\mathcal{E}'$ respectively.
For each $\nu \in \operatorname{Hom}_F(D,Z)$, Lemma~\ref{lem:cartesian} gives a bijection $f^* \colon \mathcal{N}^{-1}(\nu) \xrightarrow{\sim} {\mathcal{N}'}^{-1}(\nu \circ \phi)$.
If this fiber is nonempty, choose $(x,\nu)$ representing a class in it.
Since $f$ induces the identity on $\Gamma$, we have $\overline{x \circ f} = \bar{x}$.
Thus $G_{x \circ f} = G_x$, and
\[
f^* \circ t_{(x,\nu)} = t_{(x \circ f,\nu \circ \phi)} \colon H^1(F,G_x) \longrightarrow {\mathcal{N}'}^{-1}(\nu \circ \phi).
\]
Under these twisting bijections, the equivalence relations on the two fibers are both given by equality of images under $\operatorname{ab}_{G_x}^1$.
Taking quotients for every $\nu$ proves \eqref{eq:comparison-formal-ab}, including the description of its map.
Naturality follows from the functoriality of ordinary abelianization and twisting.

Substituting \eqref{eq:comparison-formal-ab} into the right-hand side of \eqref{eq:comparison-full-recovery} identifies it with
\[
H^1(\mathcal{E}',Z \to G) \underset{\operatorname{Hom}_F(D',Z)}{\times} \operatorname{Hom}_F(D,Z).
\]
Lemma~\ref{lem:cartesian} identifies this fiber product with $H^1(\mathcal{E},Z \to G)$ through the stated map.
\end{proof}

\subsection{The abelian group structure on \texorpdfstring{$H^1_{\mathrm{ab}}(\mathcal{E}, Z \to G)$}{abelianized gerbe cohomology} over fields of Galois--Douai type}\label{subsec:gerbe-group-law}

As in Section~\ref{subsec:abelianization}, we identify $H^1_{\mathrm{ab}}(F, G)$ with $\mathbb{H}^1(F, [Z(G^{\mathrm{sc}}) \xrightarrow{\rho_G} Z(G)])$.
In this description, the group structure is induced by componentwise multiplication of cocycles.
For the product and inverse in terms of $[G^{\mathrm{sc}} \xrightarrow{\rho_G} G]$-valued cocycles, see \cite[Lemma~B.2 and Corollary~B.7]{BK26b}.

When $F$ is of Galois--Douai type, we construct a canonical abelian group structure on $H^1_{\mathrm{ab}}(\mathcal{E}, Z \to G)$ and give cocycle formulas for the product and inverse.
With this structure, inflation from $H^1_{\mathrm{ab}}(F, G)$ is an injective homomorphism.

Write $C^i(F, M)$ for normalized continuous cochains of $\Gamma$ with values in $M(\bar{F})$ and $Z^i(F, M)$ for the corresponding cocycles whenever defined.
For $t \in M(\bar{F})$, set $(\partial t)_\sigma \coloneqq t^{-1} \, {}^\sigma t$, and for $h \in C^1(F, M)$, set
\begin{equation}\label{eq:cochain-differential}
(\partial h)_{\sigma, \tau} \coloneqq h_\sigma \, {}^\sigma h_\tau \, h_{\sigma \tau}^{-1}.
\end{equation}

For a semisimple simply connected $F$-group $H$, the connecting map for $1 \to Z(H) \to H \to H^{\mathrm{ad}} \to 1$ is
\begin{equation}\label{eq:adjoint-connecting-map}
\partial_H \colon H^1(F, H^{\mathrm{ad}}) \longrightarrow H^2(F, Z(H)), \qquad [\bar{h}] \longmapsto [\partial h],
\end{equation}
where $h \in C^1(F, H)$ lifts $\bar{h} \in Z^1(F, H^{\mathrm{ad}})$.
We say that $F$ is \emph{of Galois--Douai type} if $\partial_H$ is surjective for every such $H$.
In characteristic zero this is the condition that $\operatorname{Spec}(F)$ be of Douai type in the sense of \cite[Definition~5.2]{GA12}, since the groups $H^{\mathrm{ad}}$ and $Z(H)$ are smooth.
By \cite[Lemma~5.7]{Bor93}, this condition holds for number fields and for local fields of characteristic zero, including $\mathbb{R}$.

\begin{lemma}\label{lem:douai-central-cocycle-lift}
Suppose that $F$ is of Galois--Douai type and that $H$ is semisimple and simply connected.
For every normalized $u \in Z^2(F, Z(H))$, there exists $h \in C^1(F, H)$ such that $\partial h = u$.
\end{lemma}

\begin{proof}
By surjectivity of \eqref{eq:adjoint-connecting-map}, choose $\bar{h} \in Z^1(F, H^{\mathrm{ad}})$ with $\partial_H[\bar{h}] = [u]$.
Since $\bar{h}$ has finite image, it admits a normalized continuous lift $h_0 \in C^1(F, H)$.
Then $u = (\partial k)(\partial h_0)$ for some $k \in C^1(F, Z(H))$, hence $h \coloneqq k h_0$ satisfies $\partial h = u$.
\end{proof}

Let $j \colon Z \hookrightarrow Z(G)$ be the inclusion.
Choose a normalized continuous section $s \colon \Gamma \to \mathcal{E}$ of $\pi$ and set $c_{\sigma, \tau} \coloneqq s_\sigma s_\tau s_{\sigma \tau}^{-1}$.
Let $\mathcal{Z}_s$ be the abelian group, under componentwise multiplication, of triples
\[
(\nu, z, u) \in \operatorname{Hom}_F(D, Z) \times C^1(F, Z(G)) \times Z^2(F, Z(G^{\mathrm{sc}}))
\]
satisfying
\begin{equation}\label{eq:gerbe-abelian-central-cochains}
(\partial z)_{\sigma, \tau} \, \rho_G(u_{\sigma, \tau}) = j\bigl(\nu(c_{\sigma, \tau})\bigr)
\qquad \text{for all } \sigma, \tau \in \Gamma.
\end{equation}
Let $\mathcal{B}$ be the image of the homomorphism
\[
Z(G)(\bar{F}) \times C^1(F, Z(G^{\mathrm{sc}})) \longrightarrow \mathcal{Z}_s, \qquad
(t, k) \longmapsto \bigl(1, \, \partial t \, \rho_G(k)^{-1}, \, \partial k\bigr).
\]
This map takes values in $\mathcal{Z}_s$ since $\partial\bigl(\partial t \, \rho_G(k)^{-1}\bigr) = \rho_G(\partial k)^{-1}$.
We consider the abelian group
$\mathcal{H}_s \coloneqq \mathcal{Z}_s / \mathcal{B}$.

For another normalized continuous section $s'$ of $\pi$, set $d_\sigma \coloneqq s'_\sigma s_\sigma^{-1} \in D(\bar{F})$.
The cocycle $c'$ associated with $s'$ satisfies $c' = (\partial d) c$.
Thus, for $(\nu, z, u) \in \mathcal{Z}_s$, we have
\[
\partial\bigl((j \circ \nu \circ d) z\bigr) \, \rho_G(u)
= (j \circ \nu \circ (\partial d)) \, (j \circ \nu \circ c)
= j \circ \nu \circ c'.
\]
The map $(\nu, z, u) \mapsto (\nu, (j \circ \nu \circ d) z, u)$ preserves componentwise multiplication and fixes $\mathcal{B}$ pointwise.
It therefore induces a homomorphism
\[
T_{s, s'} \colon \mathcal{H}_s \longrightarrow \mathcal{H}_{s'}, \qquad
[\nu, z, u] \longmapsto [\nu, (j \circ \nu \circ d) z, u].
\]
If $s''_\sigma = e_\sigma s'_\sigma$, then $s''_\sigma = e_\sigma d_\sigma s_\sigma$, so $T_{s', s''} \circ T_{s, s'} = T_{s, s''}$.
Since $T_{s, s}$ is the identity, $T_{s, s'}$ is an isomorphism with inverse $T_{s', s}$.

Define the abelian group
\[
\mathcal{H} \coloneqq \varinjlim_s \mathcal{H}_s,
\]
where the indexing category has the normalized continuous sections of $\pi$ as objects and one morphism $s \to s'$ for each ordered pair, with transition map $T_{s, s'}$.
The canonical maps $\iota_s \colon \mathcal{H}_s \to \mathcal{H}$ satisfy $\iota_{s'} \circ T_{s, s'} = \iota_s$.
Each $\iota_s$ is an isomorphism, with inverse sending the class represented by $\xi' \in \mathcal{H}_{s'}$ to $T_{s', s}(\xi')$.

\begin{proposition}\label{prop:gerbe-abelian-group}
Let $(x, \nu) \in Z^1(\mathcal{E}, Z \to G)$ and choose $h \in C^1(F, G^{\mathrm{sc}})$ lifting $\bar{x}$ under $G^{\mathrm{sc}} \to G^{\mathrm{ad}}$.
Choose a normalized continuous section $s$ of $\pi$ and set $z_\sigma \coloneqq x_{s_\sigma} \, \rho_G(h_\sigma)^{-1}$.
The assignment
\begin{equation}\label{eq:gerbe-abelian-cochain-map}
\alpha \colon H^1_{\mathrm{ab}}(\mathcal{E}, Z \to G) \longrightarrow \mathcal{H}, \qquad
\operatorname{ab}[x, \nu] \longmapsto \iota_s\bigl([\nu, z, \partial h]\bigr)
\end{equation}
is a canonical injection.
If $F$ is of Galois--Douai type, then $\alpha$ is bijective.
In this case $H^1_{\mathrm{ab}}(\mathcal{E}, Z \to G)$ has a unique abelian group structure for which $\alpha$ is an isomorphism, with identity $\operatorname{ab}[1, 1]$.
\end{proposition}

\begin{proof}
Since $h$ lifts $\bar{x} \in Z^1(F, G^{\mathrm{ad}})$, its differential $u \coloneqq \partial h$ belongs to $Z^2(F, Z(G^{\mathrm{sc}}))$.
For each $\sigma \in \Gamma$, the elements $x_{s_\sigma}$ and $\rho_G(h_\sigma)$ have the same image $\bar{x}_\sigma$ in $G^{\mathrm{ad}}(\bar{F})$, so $z_\sigma \in Z(G)(\bar{F})$.
Substituting $x_{s_\sigma} = z_\sigma \, \rho_G(h_\sigma)$ into the identity
\[
x_{s_\sigma} \, {}^\sigma x_{s_\tau} = j\bigl(\nu(c_{\sigma, \tau})\bigr) x_{s_{\sigma \tau}}
\]
yields \eqref{eq:gerbe-abelian-central-cochains}.
If $s'_\sigma = d_\sigma s_\sigma$, the cochain computed using $s'$ and the same lift $h$ is $(j \circ \nu \circ d) z$, since $x_{s'_\sigma} = j(\nu(d_\sigma)) x_{s_\sigma}$.
Hence
\[
\iota_{s'}\bigl([\nu, (j \circ \nu \circ d) z, u]\bigr)
= \iota_{s'}\bigl(T_{s, s'}([\nu, z, u])\bigr)
= \iota_s\bigl([\nu, z, u]\bigr).
\]
Thus \eqref{eq:gerbe-abelian-cochain-map} is independent of $s$.

Any other lift is of the form $k h$ with $k \in C^1(F, Z(G^{\mathrm{sc}}))$.
Replacing $h$ by $k h$ multiplies the triple by $\bigl(1, \rho_G(k)^{-1}, \partial k\bigr) \in \mathcal{B}$.

If $x'_w = g x_w \, {}^{\pi(w)} g^{-1}$, write $g = t \, \rho_G(q)$ with $t \in Z(G)(\bar{F})$ and $q \in G^{\mathrm{sc}}(\bar{F})$.
Taking $h'_\sigma = q h_\sigma \, {}^\sigma q^{-1}$ gives $u' = u$ and $z' = z \, \partial(t^{-1})$.
If $x'_w = \rho_G(a_{\pi(w)}) x_w$ with $a \in Z^1(F, G_x^{\mathrm{sc}})$, take $h'_\sigma = a_\sigma h_\sigma$.
Then $z' = z$, and the twisted cocycle identity $a_{\sigma \tau} = a_\sigma h_\sigma \, {}^\sigma a_\tau \, h_\sigma^{-1}$ yields $u' = u$.
Thus \eqref{eq:abelianization-relation} shows that $\alpha$ is well defined.

Conversely, suppose that $\alpha(\operatorname{ab}[x, \nu]) = \alpha(\operatorname{ab}[x', \nu'])$.
Computing both images with the same $s$ and using the injectivity of $\iota_s$, we obtain $[\nu, z, u] = [\nu', z', u']$ in $\mathcal{H}_s$.
Then $\nu' = \nu$, and there exist $t \in Z(G)(\bar{F})$ and $k \in C^1(F, Z(G^{\mathrm{sc}}))$ such that
\[
z' = z \, \partial t \, \rho_G(k)^{-1}, \qquad u' = u \, \partial k.
\]
Replacing $h'$ by $k^{-1} h'$ and $x'_w$ by $t x'_w \, {}^{\pi(w)} t^{-1}$, we may assume that $z' = z$ and $\partial h' = \partial h$.
For $a_\sigma \coloneqq h'_\sigma h_\sigma^{-1}$, the latter equality gives
\[
a_{\sigma \tau} = a_\sigma h_\sigma \, {}^\sigma a_\tau \, h_\sigma^{-1}.
\]
Hence $a \in Z^1(F, G_x^{\mathrm{sc}})$.
The equality $z' = z$ gives $x'_{s_\sigma} = \rho_G(a_\sigma) x_{s_\sigma}$, and the equality $\nu' = \nu$ gives $x'_w = \rho_G(a_{\pi(w)}) x_w$ for every $w \in \mathcal{E}$.
Thus \eqref{eq:abelianization-relation} proves injectivity.

Suppose now that $F$ is of Galois--Douai type, and let $\xi \in \mathcal{H}$.
Since $\iota_s$ is surjective, choose $(\nu, z, u) \in \mathcal{Z}_s$ such that $\xi = \iota_s([\nu, z, u])$.
By Lemma~\ref{lem:douai-central-cocycle-lift}, choose $h \in C^1(F, G^{\mathrm{sc}})$ with $\partial h = u$.
Set $b_\sigma \coloneqq z_\sigma \, \rho_G(h_\sigma)$ and define
\begin{equation}\label{eq:gerbe-abelian-cochain-reconstruction}
x_{d s_\sigma} \coloneqq j(\nu(d)) b_\sigma \qquad (d \in D(\bar{F}), \ \sigma \in \Gamma).
\end{equation}
By \eqref{eq:gerbe-abelian-central-cochains}, we have $b_\sigma \, {}^\sigma b_\tau = j\bigl(\nu(c_{\sigma, \tau})\bigr) b_{\sigma \tau}$.
Since $(d s_\sigma)(d' s_\tau) = d \, {}^\sigma d' \, c_{\sigma, \tau} s_{\sigma \tau}$ and $\nu$ is $\Gamma$-equivariant on $\bar{F}$-points, $x$ is a continuous cocycle with restriction $j \circ \nu$ to $D(\bar{F})$.
Since $x_{s_\sigma} = z_\sigma \, \rho_G(h_\sigma)$ with $z_\sigma \in Z(G)(\bar{F})$, the cochain $h$ lifts $\bar{x}$ and gives
\[
\alpha(\operatorname{ab}[x, \nu]) = \iota_s\bigl([\nu, z, u]\bigr) = \xi.
\]
Thus $\alpha$ is bijective, and transporting multiplication on $\mathcal{H}$ by $\alpha$ gives the asserted canonical abelian group structure.
Its identity is $\operatorname{ab}[1, 1]$, since $\alpha(\operatorname{ab}[1, 1]) = \iota_s([1, 1, 1])$.
\end{proof}

Suppose for the rest of this subsection that $F$ is of Galois--Douai type.
We now describe the product and inverse for the group structure of Proposition~\ref{prop:gerbe-abelian-group} in terms of cocycles on $\mathcal{E}$.
For $i = 1, 2$, let $(x_i, \nu_i) \in Z^1(\mathcal{E}, Z \to G)$ and choose $h_i \in C^1(F, G^{\mathrm{sc}})$ lifting $\overline{x_i}$ under $G^{\mathrm{sc}} \to G^{\mathrm{ad}}$.
Set
\[
u_i \coloneqq \partial h_i, \qquad b_i(w) \coloneqq x_{i, w} \, \rho_G(h_{i, \pi(w)})^{-1} \in Z(G)(\bar{F}) \qquad (w \in \mathcal{E}).
\]
By Lemma~\ref{lem:douai-central-cocycle-lift}, choose $h_{12} \in C^1(F, G^{\mathrm{sc}})$ such that
\begin{equation}\label{eq:gerbe-product-lift}
\partial h_{12} = u_1 u_2.
\end{equation}
The product is
\begin{equation}\label{eq:gerbe-abelian-group-operation}
\operatorname{ab}[x_1, \nu_1] * \operatorname{ab}[x_2, \nu_2] = \operatorname{ab}[x_{12}, \nu_1 \nu_2],
\qquad x_{12, w} \coloneqq b_1(w) \, b_2(w) \, \rho_G(h_{12, \pi(w)}).
\end{equation}
For the inverse, choose $h' \in C^1(F, G^{\mathrm{sc}})$ such that $\partial h' = u_1^{-1}$.
Then
\begin{equation}\label{eq:gerbe-abelian-group-inverse}
\bigl(\operatorname{ab}[x_1, \nu_1]\bigr)^{-1} = \operatorname{ab}[x', \nu_1^{-1}],
\qquad x'_w \coloneqq b_1(w)^{-1} \, \rho_G(h'_{\pi(w)}).
\end{equation}
To see both formulas, set $z_{i, \sigma} \coloneqq b_i(s_\sigma)$, so that $b_i(d s_\sigma) = j(\nu_i(d)) z_{i, \sigma}$.
Since $\iota_s$ is a group isomorphism, the product and inverse can be computed componentwise in $\mathcal{H}_s$.
Applying \eqref{eq:gerbe-abelian-cochain-reconstruction} to $(\nu_1 \nu_2, z_1 z_2, u_1 u_2)$ with lift $h_{12}$ gives $x_{12}$, and applying it to $(\nu_1^{-1}, z_1^{-1}, u_1^{-1})$ with lift $h'$ gives $x'$.
If $G$ is a torus, then $G^{\mathrm{sc}} = 1$, so \eqref{eq:gerbe-abelian-group-operation} reduces to pointwise multiplication of cocycles.

For a fixed section $s$, we also write $[\nu, z, u]$ for the class $\alpha^{-1}\bigl(\iota_s([\nu, z, u])\bigr)$ in $H^1_{\mathrm{ab}}(\mathcal{E}, Z \to G)$.

\begin{lemma}\label{lem:gerbe-abelian-functoriality}
Suppose that $F$ is of Galois--Douai type.
For every morphism of pairs $\beta \colon (Z \to G) \to (Z' \to G')$, the induced map
\[
\beta_* \colon H^1_{\mathrm{ab}}(\mathcal{E}, Z \to G) \longrightarrow H^1_{\mathrm{ab}}(\mathcal{E}, Z' \to G')
\]
is a homomorphism of abelian groups.
For every morphism of gerbes $f \colon \mathcal{E}' \to \mathcal{E}$ inducing an $F$-homomorphism $\phi \colon D' \to D$ on the bands, pullback induces a homomorphism of abelian groups
\[
f^* \colon H^1_{\mathrm{ab}}(\mathcal{E}, Z \to G) \longrightarrow H^1_{\mathrm{ab}}(\mathcal{E}', Z \to G).
\]
\end{lemma}

\begin{proof}
Write $\beta^{\mathrm{sc}} \colon G^{\mathrm{sc}} \to (G')^{\mathrm{sc}}$ for the homomorphism induced by $\beta$.
Choose maximal $F$-tori $T \subseteq G$ and $T' \subseteq G'$ with $\beta(T) \subseteq T'$, as in \cite[Section~2.5]{Bor98}.
Fix a normalized continuous section $s$ of $\pi$.
For $(x, \nu) \in Z^1(\mathcal{E}, Z \to G)$, choose $h \in C^1(F, G^{\mathrm{sc}})$ lifting $\bar{x}$.
Set $z_\sigma \coloneqq x_{s_\sigma} \, \rho_G(h_\sigma)^{-1}$ and $u \coloneqq \partial h$, so that $[\nu, z, u]$ represents $\operatorname{ab}[x, \nu]$.
We have $\beta(z_\sigma) \in T'(\bar{F})$ since $Z(G) \subseteq T$.
Choose a normalized continuous cochain $k$ in $(T')^{\mathrm{sc}} \coloneqq \rho_{G'}^{-1}(T')$ lifting the image of $\beta(z)$ in $(G')^{\mathrm{ad}}$.
For each $\sigma$, the commutator of $k_\sigma$ with $\beta^{\mathrm{sc}}(g)$ belongs to $Z((G')^{\mathrm{sc}})$ for every $g \in G^{\mathrm{sc}}(\bar{F})$, since their images in $(G')^{\mathrm{ad}}$ commute.
This commutator is trivial, since it defines a morphism from the connected group $G^{\mathrm{sc}}_{\bar{F}}$ to a finite group and takes the value $1$ at $g = 1$.
Since $\beta(x_{s_\sigma}) = \beta(z_\sigma) \, \rho_{G'}(\beta^{\mathrm{sc}}(h_\sigma))$, the cochain $\sigma \mapsto k_\sigma \, \beta^{\mathrm{sc}}(h_\sigma)$ lifts $\overline{\beta \circ x}$ under $(G')^{\mathrm{sc}} \to (G')^{\mathrm{ad}}$.
Computing its differential yields
\begin{equation}\label{eq:gerbe-pair-homomorphism}
\beta_*[\nu, z, u] =
[\beta \circ \nu, \, (\beta \circ z) \, \rho_{G'}(k)^{-1}, \, (\partial k)(\beta^{\mathrm{sc}} \circ u)].
\end{equation}
The third entry belongs to $Z^2(F, Z((G')^{\mathrm{sc}}))$ since it is the differential of a lift of $\overline{\beta \circ x}$.
For two triples with cochains $z_1$ and $z_2$, choose lifts $k_1$ and $k_2$ as above and use $k_1 k_2$ for $z_1 z_2$.
Since $(T')^{\mathrm{sc}}$ is an $F$-torus, $\partial(k_1 k_2) = (\partial k_1)(\partial k_2)$.
The values of $k_2$ and their Galois conjugates centralize $\beta^{\mathrm{sc}}(G^{\mathrm{sc}})$ since $\beta^{\mathrm{sc}}$ is defined over $F$, so the values of $\partial k_2$ commute with those of $\beta^{\mathrm{sc}} \circ u_1$.
It follows from \eqref{eq:gerbe-pair-homomorphism} that $\beta_*$ preserves products.

Choose a normalized continuous section $s'$ for $\mathcal{E}'$ and write $f(s'_\sigma) = e_\sigma s_\sigma$ with $e_\sigma \in D(\bar{F})$.
Since $\overline{x \circ f} = \bar{x}$, we may use the same lift $h$ for both adjoint cocycles and obtain
\[
f^*[\nu, z, u] = [\nu \circ \phi, (j \circ \nu \circ e) z, u].
\]
Since $(\nu_1 \nu_2) \circ e = (\nu_1 \circ e)(\nu_2 \circ e)$, \eqref{eq:gerbe-abelian-group-operation} shows that $f^*$ preserves products.
\end{proof}

By \eqref{eq:gerbe-abelian-group-operation}, the Newton map $[\nu, z, u] \mapsto \nu$ is also a homomorphism.
It follows from Lemma~\ref{lem:gerbe-abelian-functoriality} that the bijection in \eqref{eq:comparison-formal-ab} is an isomorphism of abelian groups.

\begin{lemma}\label{lem:gerbe-abelian-twisting-inflation}
The map
\[
\operatorname{inf}_{\mathrm{ab}} \colon H^1_{\mathrm{ab}}(F, G) \longrightarrow H^1_{\mathrm{ab}}(\mathcal{E}, Z \to G), \qquad [(u, z)] \longmapsto [1, z, u]
\]
defines an injective homomorphism with image $\mathcal{N}_{\mathrm{ab}}^{-1}(1)$.
\end{lemma}

\begin{proof}
For $\nu = 1$, \eqref{eq:gerbe-abelian-central-cochains} becomes $\partial z \, \rho_G(u) = 1$.
A pair $(u, z) \in Z^2(F, Z(G^{\mathrm{sc}})) \times C^1(F, Z(G))$ satisfying this equation represents a class in $H^1_{\mathrm{ab}}(F, G)$, and two such pairs represent the same class if and only if they differ by the relation
\[
(u, z) \sim (u \, \partial k, z \, \partial t \, \rho_G(k)^{-1}), \qquad t \in Z(G)(\bar{F}), \quad k \in C^1(F, Z(G^{\mathrm{sc}})).
\]
This is precisely the equivalence relation on triples $(1, z, u)$, hence $\operatorname{inf}_{\mathrm{ab}}$ is well defined and identifies $H^1_{\mathrm{ab}}(F, G)$ with $\mathcal{N}_{\mathrm{ab}}^{-1}(1)$.
It preserves products since both group structures are induced by componentwise multiplication.
\end{proof}

For $a \in Z^1(F, G)$, choose $h \in C^1(F, G^{\mathrm{sc}})$ lifting the image of $a$ in $G^{\mathrm{ad}}(\bar{F})$ and set $z_\sigma \coloneqq a_\sigma \, \rho_G(h_\sigma)^{-1}$.
By \cite[(3.13.1.2)]{Bor98}, the pair $(\partial h, z)$ represents $\operatorname{ab}_G^1[a]$.
Since $(a \circ \pi)(s_\sigma) = a_\sigma$, the same lift $h$ in \eqref{eq:gerbe-abelian-cochain-map} gives
\[
\operatorname{inf}_{\mathrm{ab}}(\operatorname{ab}_G^1[a]) = [1, z, \partial h] = \operatorname{ab}[a \circ \pi, 1].
\]

Following \cite[Section~5.1]{Bor93}, set $H^2_{\mathrm{ab}}(F, G) \coloneqq \mathbb{H}^2(F, [Z(G^{\mathrm{sc}}) \xrightarrow{\rho_G} Z(G)])$.
Consider the morphisms of complexes
\begin{equation*}
[1 \to Z] \longrightarrow [1 \to Z(G)] \longrightarrow [Z(G^{\mathrm{sc}}) \xrightarrow{\rho_G} Z(G)],
\end{equation*}
where the maps in degree $0$ are the inclusion $j \colon Z \hookrightarrow Z(G)$ and the identity on $Z(G)$ respectively.
Their composite induces a homomorphism $j_{\mathrm{ab}}^2 \colon H^2(F, Z) \to H^2_{\mathrm{ab}}(F, G)$.
Define
\begin{equation}\label{eq:gerbe-obstruction}
\delta_{\mathcal{E}, \mathrm{ab}} \colon \operatorname{Hom}_F(D, Z) \longrightarrow H^2_{\mathrm{ab}}(F, G), \qquad \nu \longmapsto j_{\mathrm{ab}}^2\bigl(\nu_*(\xi_{\mathcal{E}})\bigr).
\end{equation}

By \cite[Section~1.8 and Lemma~1.9]{Bor93}, the map $[b] \mapsto [b, s_G]$ identifies $H^2(F, Z(G))$ with $H^2(F, G)$.
Under this identification, the second morphism of complexes induces Borovoi's second abelianization map $\operatorname{ab}_G^2 \colon H^2(F, G) \to H^2_{\mathrm{ab}}(F, G)$ \cite[Section~5]{Bor93}.
Since $\delta_{\mathcal{E}}(\nu) = [j \circ \nu \circ c, s_G]$, it follows that
\begin{equation}\label{eq:gerbe-obstruction-compatibility}
\operatorname{ab}_G^2\bigl(\delta_{\mathcal{E}}(\nu)\bigr) = \delta_{\mathcal{E}, \mathrm{ab}}(\nu).
\end{equation}

\begin{proposition}\label{prop:general-newton-fibers}
Suppose that $F$ is of Galois--Douai type.
There is an exact sequence of abelian groups
\[
1 \longrightarrow H^1_{\mathrm{ab}}(F, G) \xrightarrow{\operatorname{inf}_{\mathrm{ab}}} H^1_{\mathrm{ab}}(\mathcal{E}, Z \to G) \xrightarrow{\mathcal{N}_{\mathrm{ab}}} \operatorname{Hom}_F(D, Z) \xrightarrow{\delta_{\mathcal{E}, \mathrm{ab}}} H^2_{\mathrm{ab}}(F, G).
\]
For $\nu \in \operatorname{Hom}_F(D, Z)$, the fiber $\mathcal{N}^{-1}(\nu)$ is nonempty if and only if $\delta_{\mathcal{E}, \mathrm{ab}}(\nu) = 0$.
For each such $\nu$, the fiber $\mathcal{N}_{\mathrm{ab}}^{-1}(\nu)$ is a torsor under $H^1_{\mathrm{ab}}(F, G)$ via $\eta \cdot \zeta \coloneqq \operatorname{inf}_{\mathrm{ab}}(\eta) * \zeta$.
\end{proposition}

\begin{proof}
In the description of \eqref{eq:gerbe-abelian-cochain-map}, the Newton map sends $[\nu, z, u]$ to $\nu$ and hence is a homomorphism.
By Lemma~\ref{lem:gerbe-abelian-twisting-inflation}, its kernel is the image of $\operatorname{inf}_{\mathrm{ab}}$.
The hypercohomology sequence of $[Z(G^{\mathrm{sc}}) \to Z(G)]$ is exact at the middle term of
\[
H^2(F, Z(G^{\mathrm{sc}})) \xrightarrow{\rho_{G, *}} H^2(F, Z(G)) \longrightarrow H^2_{\mathrm{ab}}(F, G).
\]
Thus $\delta_{\mathcal{E}, \mathrm{ab}}(\nu) = 0$ if and only if $[j \circ \nu \circ c] = \rho_{G, *}[u]$ for some $u \in Z^2(F, Z(G^{\mathrm{sc}}))$.
This equality holds if and only if there is a cochain $z \in C^1(F, Z(G))$ satisfying \eqref{eq:gerbe-abelian-central-cochains} for that $u$.
This proves exactness at $\operatorname{Hom}_F(D, Z)$.
Since $\operatorname{ab}$ is surjective and $\mathcal{N} = \mathcal{N}_{\mathrm{ab}} \circ \operatorname{ab}$, it also gives the criterion for $\mathcal{N}^{-1}(\nu)$.
Each nonempty fiber of $\mathcal{N}_{\mathrm{ab}}$ is a coset of its kernel, which gives the stated torsor action.
\end{proof}

Combining Propositions~\ref{prop:gerbe-inflation-restriction} and~\ref{prop:general-newton-fibers} with \eqref{eq:gerbe-obstruction-compatibility} yields the commutative diagram
\[
\begin{tikzcd}[column sep=small]
1 \arrow[r] & H^1(F, G) \arrow[r, "\operatorname{inf}"] \arrow[d, "\operatorname{ab}_G^1"'] & H^1(\mathcal{E}, Z \to G) \arrow[r, "\mathcal{N}"] \arrow[d, "\operatorname{ab}"] & \operatorname{Hom}_F(D, Z) \arrow[r, "\delta_{\mathcal{E}}"] \arrow[d, equal] & H^2(F, G) \arrow[d, "\operatorname{ab}_G^2"] \\
1 \arrow[r] & H^1_{\mathrm{ab}}(F, G) \arrow[r, "\operatorname{inf}_{\mathrm{ab}}"'] & H^1_{\mathrm{ab}}(\mathcal{E}, Z \to G) \arrow[r, "\mathcal{N}_{\mathrm{ab}}"'] & \operatorname{Hom}_F(D, Z) \arrow[r, "\delta_{\mathcal{E}, \mathrm{ab}}"'] & H^2_{\mathrm{ab}}(F, G).
\end{tikzcd}
\]
In the upper row, exactness at $\operatorname{Hom}_F(D, Z)$ means that $\nu$ belongs to $\operatorname{im}(\mathcal{N})$ if and only if $\delta_{\mathcal{E}}(\nu)$ is neutral.
The lower row is exact as a sequence of abelian groups.

\subsection{Local and global exact sequences in terms of the algebraic fundamental group}

We express the exact sequence of Proposition~\ref{prop:general-newton-fibers} for non-Archimedean local fields of characteristic zero and for number fields.

For a maximal $F$-torus $T \subseteq G$, let $T^{\mathrm{sc}} \coloneqq \rho_G^{-1}(T)$, which is a maximal $F$-torus of $G^{\mathrm{sc}}$.
Borovoi's algebraic fundamental group of $G$ is
\[
\pi_1(G) \coloneqq X_*(T)/\rho_{G,*}\bigl(X_*(T^{\mathrm{sc}})\bigr).
\]
It is a finitely generated abelian group with a natural $\Gamma$-action and is independent of $T$ up to canonical $\Gamma$-equivariant isomorphism, as shown in \cite[Lemma~1.2 and Section~1.4]{Bor98} and \cite[Proposition~2.2.1]{Fu25}.

Let $\pi_1(G)_{\Gamma,\mathrm{tors}}$ denote the torsion subgroup of $\pi_1(G)_\Gamma$.
If $F$ is a non-Archimedean local field of characteristic zero, \cite[Propositions~2.8 and~4.1]{Bor98} give canonical isomorphisms
\begin{equation}\label{eq:identification-H1ab-H2ab}
H^1_{\mathrm{ab}}(F,G) \simeq \pi_1(G)_{\Gamma,\mathrm{tors}},
\qquad
H^2_{\mathrm{ab}}(F,G) \simeq \pi_1(G)_\Gamma \otimes_{\mathbb{Z}} \mathbb{Q}/\mathbb{Z}.
\end{equation}

\begin{corollary}\label{cor:general-padic-newton-fibers}
Suppose that $F$ is a non-Archimedean local field of characteristic zero.
Endow $H^1(\mathcal{E},Z \to G)$ with the abelian group structure induced by the bijection of Corollary~\ref{cor:abelianization-bijective}.
Then there is an exact sequence of abelian groups
\[
1 \longrightarrow \pi_1(G)_{\Gamma,\mathrm{tors}}
\longrightarrow H^1(\mathcal{E},Z \to G)
\xrightarrow{\mathcal{N}} \operatorname{Hom}_F(D,Z)
\xrightarrow{\delta_{\mathcal{E},\mathrm{ab}}}
\pi_1(G)_\Gamma \otimes_{\mathbb{Z}} \mathbb{Q}/\mathbb{Z}.
\]
For $\nu \in \operatorname{Hom}_F(D,Z)$, the fiber $\mathcal{N}^{-1}(\nu)$ is nonempty if and only if $\delta_{\mathcal{E},\mathrm{ab}}(\nu) = 0$.
Each nonempty fiber is canonically a torsor under $\pi_1(G)_{\Gamma,\mathrm{tors}}$.
\end{corollary}

\begin{proof}
By Corollary~\ref{cor:abelianization-bijective}, abelianization identifies $H^1(\mathcal{E},Z \to G)$ with $H^1_{\mathrm{ab}}(\mathcal{E},Z \to G)$ and carries $\mathcal{N}$ to $\mathcal{N}_{\mathrm{ab}}$.
Under this identification and the isomorphisms in \eqref{eq:identification-H1ab-H2ab}, the exact sequence of Proposition~\ref{prop:general-newton-fibers} gives the asserted exact sequence.
Exactness yields the nonemptiness criterion and identifies $\pi_1(G)_{\Gamma,\mathrm{tors}}$ with $\ker(\mathcal{N})$.
Translation by this kernel acts simply transitively on each nonempty fiber of $\mathcal{N}$.
\end{proof}

We next specialize Proposition~\ref{prop:general-newton-fibers} to a number field $F$.
Let $\mathcal{F}^1(\pi_1(G)/F)$ and $\mathcal{F}^2(\pi_1(G)/F)$ denote the fiber products of Borovoi and Kaletha defined in \cite[Definitions~4.3.7 and~4.3.10]{BK26}.
By \cite[Theorem~5.5.4(1)--(2)]{BK26}, there are canonical isomorphisms
\begin{equation}\label{eq:identification-global-H1ab-H2ab}
H^1_{\mathrm{ab}}(F,G) \simeq \mathcal{F}^1(\pi_1(G)/F),
\qquad
H^2_{\mathrm{ab}}(F,G) \simeq \mathcal{F}^2(\pi_1(G)/F).
\end{equation}

\begin{corollary}\label{cor:general-global-newton-fibers}
Let $F$ be a number field.
Under the identifications in \eqref{eq:identification-global-H1ab-H2ab}, there is an exact sequence of abelian groups
\[
1 \longrightarrow \mathcal{F}^1(\pi_1(G)/F)
\xrightarrow{\operatorname{inf}_{\mathrm{ab}}} H^1_{\mathrm{ab}}(\mathcal{E},Z \to G)
\xrightarrow{\mathcal{N}_{\mathrm{ab}}} \operatorname{Hom}_F(D,Z)
\xrightarrow{\delta_{\mathcal{E},\mathrm{ab}}} \mathcal{F}^2(\pi_1(G)/F).
\]
For $\nu \in \operatorname{Hom}_F(D,Z)$, the fiber $\mathcal{N}^{-1}(\nu)$ is nonempty if and only if $\delta_{\mathcal{E},\mathrm{ab}}(\nu) = 0$.
Each nonempty fiber of $\mathcal{N}_{\mathrm{ab}}$ is canonically a torsor under $\mathcal{F}^1(\pi_1(G)/F)$.

Suppose further that $F$ is totally imaginary.
Choose a finite Galois extension $E/F$ through which the $\Gamma$-action on $\pi_1(G)$ factors, and let $V_{E,\mathrm{f}}$ denote the set of finite places of $E$.
Set
\[
\pi_1(G)[V_{E,\mathrm{f}}]_0
\coloneqq \ker\left[\bigoplus_{w \in V_{E,\mathrm{f}}} \pi_1(G) \xrightarrow{\sum} \pi_1(G)\right],
\]
with the diagonal action of $\Gamma_{E/F} \coloneqq \operatorname{Gal}(E/F)$.
Then the preceding exact sequence takes the form
\begin{multline}\label{eq:gerbe-totally-imaginary-obstruction}
1 \longrightarrow \bigl(\pi_1(G)[V_{E,\mathrm{f}}]_0\bigr)_{\Gamma_{E/F},\mathrm{tors}}
\longrightarrow H^1(\mathcal{E},Z \to G) \\
\xrightarrow{\mathcal{N}} \operatorname{Hom}_F(D,Z)
\xrightarrow{\delta_{\mathcal{E},\mathrm{ab}}}
\bigl(\pi_1(G)[V_{E,\mathrm{f}}]_0\bigr)_{\Gamma_{E/F}} \otimes_{\mathbb{Z}} \mathbb{Q}/\mathbb{Z}.
\end{multline}
Each nonempty fiber of $\mathcal{N}$ is canonically a torsor under $\bigl(\pi_1(G)[V_{E,\mathrm{f}}]_0\bigr)_{\Gamma_{E/F},\mathrm{tors}}$.
\end{corollary}

\begin{proof}
Since every number field is of Galois--Douai type, the first exact sequence and the assertions about its fibers follow from Proposition~\ref{prop:general-newton-fibers} and \eqref{eq:identification-global-H1ab-H2ab}.

Suppose that $F$ is totally imaginary.
By \cite[Corollary~5.5.7(1)--(2)]{BK26}, there are canonical isomorphisms
\begin{align*}
H^1_{\mathrm{ab}}(F,G) \simeq \bigl(\pi_1(G)[V_{E,\mathrm{f}}]_0\bigr)_{\Gamma_{E/F},\mathrm{tors}}, \qquad
H^2_{\mathrm{ab}}(F,G) \simeq \bigl(\pi_1(G)[V_{E,\mathrm{f}}]_0\bigr)_{\Gamma_{E/F}} \otimes_{\mathbb{Z}} \mathbb{Q}/\mathbb{Z}.
\end{align*}
By Corollary~\ref{cor:abelianization-bijective}, abelianization is bijective and thus transports the abelian group structure to $H^1(\mathcal{E},Z \to G)$.
Since $\mathcal{N} = \mathcal{N}_{\mathrm{ab}} \circ \operatorname{ab}$, these identifications give \eqref{eq:gerbe-totally-imaginary-obstruction} and the stated torsor structure on each nonempty fiber of $\mathcal{N}$.
\end{proof}

\subsection{Central quotients and direct limits}

For a pro-algebraic $F$-group $H$ of multiplicative type, let
\[
X^*(H) = \operatorname{Hom}_{\bar F}(H_{\bar F},\mathbb{G}_{m,\bar F}),
\qquad
X_*(H) = \operatorname{Hom}_{\bar F}(\mathbb{G}_{m,\bar F},H_{\bar F}).
\]
Both carry natural $\Gamma$-actions.

\begin{lemma}\label{lem:finite-quotient}
For a finite central $F$-subgroup $A \subseteq G$, there is a natural exact sequence of $\Gamma$-modules
\[
0 \longrightarrow \pi_1(G) \longrightarrow \pi_1(G/A) \longrightarrow \operatorname{Hom}_{\mathbb{Z}}(X^*(A),\mathbb{Q}/\mathbb{Z}) \longrightarrow 0.
\]
These sequences are compatible with inclusions of finite central subgroups.
Moreover, for finite central $F$-subgroups $A \subseteq A' \subseteq G$, the induced map $\pi_1(G/A) \to \pi_1(G/A')$ is injective.
\end{lemma}

\begin{proof}
Every maximal torus of $G$ contains $A$.
Fix a maximal $F$-torus $T \subseteq G$ and set $T^{\mathrm{sc}} \coloneqq \rho_G^{-1}(T)$.
The isogeny $T \to T/A$ induces an inclusion $X_*(T) \hookrightarrow X_*(T/A)$ of finite index.
The composite $G^{\mathrm{sc}} \xrightarrow{\rho_G} G \to G/A$ is the simply connected cover of the derived group of $G/A$, and the inverse image of $T/A$ in $G^{\mathrm{sc}}$ is $T^{\mathrm{sc}}$.
Thus, under the inclusion $X_*(T) \hookrightarrow X_*(T/A)$, the two images of $X_*(T^{\mathrm{sc}})$ agree.
Quotienting by this common subgroup yields an exact sequence
\[
0 \longrightarrow \pi_1(G) \longrightarrow \pi_1(G/A) \longrightarrow X_*(T/A)/X_*(T) \longrightarrow 0.
\]

It remains to identify $X_*(T/A)/X_*(T)$ with $\operatorname{Hom}_{\mathbb{Z}}(X^*(A),\mathbb{Q}/\mathbb{Z})$.
The exact sequence
\[
0 \longrightarrow X^*(T/A) \longrightarrow X^*(T) \longrightarrow X^*(A) \longrightarrow 0
\]
identifies $X^*(A)$ with $X^*(T)/X^*(T/A)$.
The $\mathbb{Z}$-pairing between $X_*(T/A)$ and $X^*(T/A)$ extends canonically to a $\mathbb{Q}$-pairing between $X_*(T/A)$ and $X^*(T)$.
Reduction modulo $\mathbb{Z}$ induces a $\Gamma$-equivariant perfect pairing
\[
X_*(T/A)/X_*(T) \otimes X^*(T)/X^*(T/A) \longrightarrow \mathbb{Q}/\mathbb{Z};
\]
see \cite[Section~4.1]{Kal16}.
We therefore obtain a $\Gamma$-equivariant isomorphism
\[
X_*(T/A)/X_*(T) \isoto \operatorname{Hom}_{\mathbb{Z}}(X^*(A),\mathbb{Q}/\mathbb{Z}),
\]
which proves the asserted exact sequence.
The construction is independent of $T$ under the canonical identifications of the algebraic fundamental groups.

For finite central $F$-subgroups $A \subseteq A'$, precomposition with the restriction map $X^*(A') \to X^*(A)$ defines a homomorphism $\operatorname{Hom}_{\mathbb{Z}}(X^*(A),\mathbb{Q}/\mathbb{Z}) \to \operatorname{Hom}_{\mathbb{Z}}(X^*(A'),\mathbb{Q}/\mathbb{Z})$.
The pairings show that this homomorphism is compatible with the map $\pi_1(G/A) \to \pi_1(G/A')$ induced by $G/A \to G/A'$.
This proves compatibility with inclusions.
Applying the exact sequence proved above to the finite central subgroup $A'/A \subseteq G/A$ shows that $\pi_1(G/A) \to \pi_1(G/A')$ is injective.
\end{proof}

For a central $F$-subgroup $Z \subseteq G$, define
\begin{equation}\label{eq:tilde-pi}
\tilde{\pi}_1(G,Z) \coloneqq \varinjlim_{A \subseteq Z} \pi_1(G/A),
\end{equation}
where $A$ runs over the finite $F$-subgroups of $Z$, ordered by inclusion.

\begin{proposition}\label{prop:colimit-quotient}
For a central $F$-subgroup $Z$ of $G$, there is a natural exact sequence of $\Gamma$-modules
\begin{align*}
  0 \longrightarrow \pi_1(G) \longrightarrow
  \tilde \pi_1(G,Z)
  \longrightarrow
  \operatorname{Hom}_{\mathbb{Z}}(X^*(Z),\mathbb{Q}/\mathbb{Z})
  \longrightarrow 0,
\end{align*}
hence $\tilde \pi_1(G,Z)/\pi_1(G) \simeq \operatorname{Hom}_{\mathbb{Z}}(X^*(Z),\mathbb{Q}/\mathbb{Z})$.
\end{proposition}

\begin{proof}
The subgroups $Z[n] = \ker([n] \colon Z \to Z)$ are finite and defined over $F$.
They are cofinal among the finite subgroup schemes of $Z$, and their character groups are given by $X^*(Z[n]) = X^*(Z)/n X^*(Z)$.
Since $X^*(Z)$ is finitely generated, every homomorphism from $X^*(Z)$ to
$\mathbb{Q}/\mathbb{Z}$ has finite image and factors through one of these quotients.
Consequently, we obtain
\[
\varinjlim_{A \subseteq Z,\, A \text{ finite}/F}
\operatorname{Hom}_{\mathbb{Z}}(X^*(A),\mathbb{Q}/\mathbb{Z})
= \operatorname{Hom}_{\mathbb{Z}}(X^*(Z),\mathbb{Q}/\mathbb{Z}).
\]
Since filtered direct limits are exact, taking the direct limit of the exact sequences of Lemma~\ref{lem:finite-quotient} proves the assertion.
\end{proof}

Let $Z \subseteq G$ be a central $F$-subgroup.
Since every cocharacter of $Z$ factors through $Z^\circ$, we identify $X_*(Z^\circ)$ with $\operatorname{Hom}_{\mathbb{Z}}(X^*(Z), \mathbb{Z})$.
By \cite[(11.6)]{Kot14}, the inclusion $Z^\circ \hookrightarrow G$ induces a $\Gamma$-equivariant injection
\begin{equation}\label{eq:def-i}
i \colon \operatorname{Hom}_{\mathbb{Z}}(X^*(Z), \mathbb{Z}) \hookrightarrow \pi_1(G).
\end{equation}
We extend $i$ to rational-valued homomorphisms by passing to finite central quotients.
For $n \geq 1$, set $A_n \coloneqq Z[n]$.
The quotient map identifies $X^*(Z / A_n)$ with $n X^*(Z)$, and the canonical isomorphism
\[
\operatorname{Hom}_{\mathbb{Z}}\left(X^*(Z), \tfrac{1}{n} \mathbb{Z}\right)
\xrightarrow{\sim}
\operatorname{Hom}_{\mathbb{Z}}(n X^*(Z), \mathbb{Z})
\]
sends $\lambda$ to the homomorphism $n \chi \mapsto n \lambda(\chi)$.
Applying \cite[(11.6)]{Kot14} to the central torus $(Z / A_n)^\circ \subseteq G / A_n$ yields an injection
\[
\iota_n \colon \operatorname{Hom}_{\mathbb{Z}}\left(X^*(Z), \tfrac{1}{n} \mathbb{Z}\right) \hookrightarrow \pi_1(G / A_n).
\]
For $n \mid m$, the pullback on character groups induced by $Z / A_n \to Z / A_m$ is the inclusion $m X^*(Z) \subseteq n X^*(Z)$.
Thus the map $\pi_1(G / A_n) \to \pi_1(G / A_m)$ sends $\iota_n(\lambda)$ to $\iota_m(\lambda)$.
Since $X^*(Z)$ is finitely generated and the subgroups $Z[n]$ are cofinal among the finite $F$-subgroups of $Z$, passage to the direct limit yields a $\Gamma$-equivariant injection
\begin{equation}\label{eq:kottwitz-central-cocharacters}
\iota \colon \operatorname{Hom}_{\mathbb{Z}}(X^*(Z), \mathbb{Q}) \hookrightarrow \varinjlim_n \pi_1(G / A_n) = \tilde{\pi}_1(G, Z).
\end{equation}
Its restriction to $\operatorname{Hom}_{\mathbb{Z}}(X^*(Z), \mathbb{Z})$ is the composite of $i$ with $\pi_1(G) \hookrightarrow \tilde{\pi}_1(G, Z)$.

To describe the class of $\iota(\lambda)$ modulo $\pi_1(G)$, choose $n$ such that $\lambda$ takes values in $\tfrac{1}{n} \mathbb{Z}$ and choose a maximal $F$-torus $T \subseteq G$.
For $\chi \in X^*(Z)$, choose a lift $\widetilde{\chi} \in X^*(T)$.
Then $n \widetilde{\chi}$ belongs to $X^*(T / A_n)$ and pairs with the cocharacter of $(Z / A_n)^\circ$ associated with $\lambda$ to give $n \lambda(\chi)$.
Hence, by the pairing in the proof of Lemma~\ref{lem:finite-quotient}, the class of $\iota(\lambda)$ corresponds under Proposition~\ref{prop:colimit-quotient} to the homomorphism
\[
X^*(Z) \longrightarrow \mathbb{Q} / \mathbb{Z}, \qquad \chi \longmapsto \lambda(\chi) + \mathbb{Z}.
\]

Let $q \colon G \to G / Z$ be the quotient homomorphism.
The kernel of $X_*(T) \to X_*(T / Z)$ is $X_*(Z^\circ)$, and the image of the coroot lattice of $(G, T)$ is the coroot lattice of $(G / Z, T / Z)$.
Since $Z^\circ \cap G^{\mathrm{der}}$ is finite, $X_*(Z^\circ)$ meets the coroot lattice of $(G, T)$ trivially.
Passing to the quotients by the coroot lattices yields the exact sequence of $\Gamma$-modules
\begin{equation}\label{eq:central-quotient-fundamental-group}
0 \longrightarrow X_*(Z^\circ) \xrightarrow{i} \pi_1(G) \xrightarrow{q_*} \pi_1(G / Z).
\end{equation}

Under the finite-image assumption below, $H^1(\mathcal{E},Z \to G)$ is a direct limit over finite central subgroups.

\begin{lemma}\label{lem:abelianization-colimit}
Let $\mathcal{E}$ be a Galois gerbe banded by $D$ and let $Z$ be a central $F$-subgroup of a connected reductive $F$-group $G$.
Suppose that every $F$-homomorphism $D \to Z$ factors through a finite $F$-subgroup of $Z$.
Then the natural maps induce bijections of pointed sets
\begin{align*}
\varinjlim_{\substack{A \subseteq Z \\ A \text{ finite}/F}} H^1(\mathcal{E},A \to G) \isoto H^1(\mathcal{E},Z \to G)
\quad \text{and} \quad
\varinjlim_{\substack{A \subseteq Z \\ A \text{ finite}/F}} H^1_{\mathrm{ab}}(\mathcal{E},A \to G) \isoto H^1_{\mathrm{ab}}(\mathcal{E},Z \to G).
\end{align*}
The transition maps in both direct systems are injective.
If $F$ is of Galois--Douai type, the second bijection is an isomorphism of abelian groups.
If $X^*(D)$ is torsion, every $F$-homomorphism $D \to Z$ factors through a finite $F$-subgroup of $Z$.
\end{lemma}

\begin{proof}
For each finite $F$-subgroup $A \subseteq Z$, the set $Z^1(\mathcal{E}, A \to G)$ consists of the pairs $(x, \nu) \in Z^1(\mathcal{E}, Z \to G)$ for which $\nu$ factors through $A$.
Since $Z$ is central, the $G(\bar{F})$-action preserves $\nu$, so $H^1(\mathcal{E}, A \to G) \to H^1(\mathcal{E}, Z \to G)$ is injective.
The relation in \eqref{eq:abelianization-relation} also preserves $\nu$, since it replaces $x$ by $(\rho_{G_x} \circ a \circ \pi) \cdot x$ for $a \in Z^1(F, G_x^{\mathrm{sc}})$.
Since $G_x$ depends only on the adjoint cocycle $\bar{x}$, the same cocycles $a$ define the relation for $A$ and for $Z$, hence
$H^1_{\mathrm{ab}}(\mathcal{E}, A \to G) \to H^1_{\mathrm{ab}}(\mathcal{E}, Z \to G)$ is injective as well.

The transition maps in both systems are therefore injective.
By hypothesis, every $(x, \nu) \in Z^1(\mathcal{E}, Z \to G)$ belongs to $Z^1(\mathcal{E}, A \to G)$ for some finite $A \subseteq Z$.
Since any two finite $F$-subgroups of $Z$ are contained in a common finite $F$-subgroup, both direct-limit maps in the statement are bijective.
If $F$ is of Galois--Douai type, Lemma~\ref{lem:gerbe-abelian-functoriality} applied to the inclusions of pairs shows that the second bijection is an isomorphism of abelian groups.

Finally, if $X^*(D)$ is torsion, the image of $\nu^* \colon X^*(Z) \to X^*(D)$ is annihilated by some positive integer $n$, since $X^*(Z)$ is finitely generated.
Thus $\nu^*$ factors through $X^*(Z) / n X^*(Z) = X^*(Z[n])$, so $\nu$ factors through the finite $F$-subgroup $Z[n]$.
\end{proof}

We fix the notation for coinvariants and norms.
Let $M$ be a discrete $\Gamma$-module on which $\Gamma$ acts through a finite quotient.
This condition holds for every finitely generated discrete $\Gamma$-module with continuous action, since the common stabilizer of a finite set of generators is open.
The action on $\tilde{\pi}_1(G,Z)$ also factors through a finite quotient.
To see this, choose a maximal $F$-torus $T \subseteq G$ and a finite Galois extension $E/F$ splitting $T$.
For every finite $F$-subgroup $A \subseteq Z$, the character lattice $X^*(T/A)$ is a $\Gamma$-stable sublattice of $X^*(T)$.
Thus $\Gamma_E$ acts trivially on $X^*(T/A)$, hence also on $X_*(T/A)$ and its quotient $\pi_1(G/A)$.
It follows that $\Gamma_E$ acts trivially on $\tilde{\pi}_1(G,Z)$.

Let $I_\Gamma \subseteq \mathbb{Z}[\Gamma]$ be the augmentation ideal, generated by $\sigma - 1$ for $\sigma \in \Gamma$.
Thus we have
\[
I_\Gamma M = \langle \sigma m - m \mid \sigma \in \Gamma,\ m \in M \rangle_{\mathbb{Z}},\qquad M_\Gamma = M/I_\Gamma M.
\]

Consider the finite Galois extensions $E/F$ in $\bar F$ such that the $\Gamma$-action on $M$ factors through $\Gamma_{E/F} \coloneqq \operatorname{Gal}(E/F)$.
Ordered by inclusion, these extensions form a directed set cofinal among all finite Galois extensions of $F$ in $\bar F$.

For each $E/F$ in this directed set, define
\[
N_{E/F} \coloneqq \sum_{\sigma \in \Gamma_{E/F}} \sigma,\qquad M^{N_{E/F}} \coloneqq \ker \bigl(N_{E/F} \colon M \to M\bigr).
\]
The norm $N_{E/F}$ annihilates $I_\Gamma M$.
If $E \subseteq E'$, then $N_{E'/F} = [E' : E] N_{E/F}$ as endomorphisms of $M$, hence there is an inclusion
\[
M^{N_{E/F}} \hookrightarrow M^{N_{E'/F}}.
\]

\section{Formulas over local fields}
\label{sec:local-formula}

Throughout this section, let $F$ be a local field of characteristic zero.
Let $G$ be a connected reductive $F$-group, and let
$Z \subseteq G$ be a central $F$-subgroup.
Let $\mathcal{E}$ be a Galois gerbe over $F$ banded by $D$.

\subsection{Kaletha's local rigid gerbe}\label{s:Kaletha-local-gerbe}

For a finite Galois extension $E/F$ in $\bar F$ and a positive integer $n$, set
\[
u_{E/F,n} \coloneqq \operatorname{Res}_{E/F} \mu_n/\mu_n,
\]
where $\mu_n$ is embedded diagonally.
For finite Galois extensions $E \subseteq K$ and positive integers $n \mid m$, the composite
\[
\operatorname{Res}_{K/F} \mu_m \xrightarrow{(-)^{m/n}} \operatorname{Res}_{K/F} \mu_n \xrightarrow{N_{K/E}} \operatorname{Res}_{E/F} \mu_n
\]
induces a surjective transition map $u_{K/F,m} \to u_{E/F,n}$.
Following \cite[Section~3.1]{Kal16}, define
\[
u \coloneqq \varprojlim_{E/F,n} u_{E/F,n}.
\]
The limit runs over all such pairs $(E/F,n)$, with the transition maps described above.

Its character group is naturally identified with
\begin{equation}\label{eq:character-group-u}
  X^*(u) \simeq
  \begin{cases}
  \{h \colon \Gamma \to \mathbb{Q}/\mathbb{Z} \mid h \text{ is locally constant}\} & \text{if $F$ is non-Archimedean},\\
  \{h \colon \Gamma \to \mathbb{Q}/\mathbb{Z} \mid \sum_{\sigma \in \Gamma} h(\sigma) = 0\} & \text{if $F = \mathbb{R}$},\\
  0 & \text{if $F = \mathbb{C}$},
  \end{cases}
\end{equation}
with $\Gamma$-action given by $(\tau h)(\sigma) = h(\tau^{-1} \sigma)$; see \cite[Section~2.2]{KT23}.

By \cite[Theorem~3.1]{Kal16}, with the complex case following from $u = 1$, we have
\[
H^1(F,u) = 0,\qquad H^2(F,u) \simeq
\begin{cases}
\widehat{\mathbb{Z}} & \text{if $F$ is non-Archimedean},\\
\mathbb{Z}/2 \mathbb{Z} & \text{if $F = \mathbb{R}$},\\
0 & \text{if $F = \mathbb{C}$}.
\end{cases}
\]
These isomorphisms are canonical.
For $F \ne \mathbb{C}$, let $\xi^{\mathrm{rig}} \in H^2(F,u)$ be the class corresponding to $-1$, and choose a Galois gerbe
\[
1 \longrightarrow u(\bar F) \longrightarrow \mathcal{E}^{\operatorname{rig}} \longrightarrow \Gamma \longrightarrow 1
\]
whose isomorphism class corresponds to $\xi^{\mathrm{rig}}$.
For $F = \mathbb{C}$, take $\xi^{\mathrm{rig}} = 0$ and $\mathcal{E}^{\operatorname{rig}} = 1$.
Since $H^1(F,u) = 0$, the set $H^1(\mathcal{E}^{\mathrm{rig}},Z \to G)$ is independent of the choice of $\mathcal{E}^{\mathrm{rig}}$ in its isomorphism class by Lemma~\ref{lem:comparison-dependence}(3).

For an $F$-homomorphism $\nu \colon u \to Z$, we denote by $\nu^* \colon X^*(Z) \to X^*(u)$ the corresponding homomorphism of $\Gamma$-modules.

\begin{lemma}\label{lem:u-nonfinite}
Using the description of $X^*(u)$ in \eqref{eq:character-group-u}, evaluation at $1 \in \Gamma$ induces a natural injection
\[
\operatorname{ev}_1 \colon \operatorname{Hom}_F(u,Z) \hookrightarrow \operatorname{Hom}_{\mathbb{Z}}(X^*(Z),\mathbb{Q}/\mathbb{Z}),
\qquad \nu \longmapsto \bigl(\chi \mapsto (\nu^* \chi)(1)\bigr),
\]
whose image is
\begin{equation}\label{eq:local-rigid-band}
\begin{cases}
\operatorname{Hom}_{\mathbb{Z}}(X^*(Z),\mathbb{Q}/\mathbb{Z}) & \text{if $F$ is non-Archimedean},\\
\{\lambda \in \operatorname{Hom}_{\mathbb{Z}}(X^*(Z),\mathbb{Q}/\mathbb{Z}) \mid \lambda + \lambda \circ c = 0\} & \text{if $F = \mathbb{R}$},\\
0 & \text{if $F = \mathbb{C}$},
\end{cases}
\end{equation}
where $c$ denotes complex conjugation.
In particular, for non-Archimedean $F$, evaluation induces a natural isomorphism
\begin{equation}\label{eq:u-nonfinite}
\operatorname{Hom}_F(u,Z) \simeq \operatorname{Hom}_{\mathbb{Z}}(X^*(Z),\mathbb{Q}/\mathbb{Z}).
\end{equation}
\end{lemma}

\begin{proof}
An $F$-homomorphism $\nu \colon u \to Z$ corresponds to a $\Gamma$-equivariant homomorphism $\nu^* \colon X^*(Z) \to X^*(u)$.
Set $\lambda(\chi) \coloneqq (\nu^* \chi)(1)$.
Since $\Gamma$ acts on $X^*(u)$ by left translation, equivariance yields
\[
(\nu^* \chi)(\sigma) = \bigl(\nu^*(\sigma^{-1} \chi)\bigr)(1) = \lambda(\sigma^{-1} \chi).
\]
Thus $\lambda$ determines $\nu$, and evaluation at $1$ is natural in $Z$ since it commutes with precomposition on character groups.

Conversely, let $\lambda \in \operatorname{Hom}_{\mathbb{Z}}(X^*(Z), \mathbb{Q}/\mathbb{Z})$ and set $h_\chi(\sigma) \coloneqq \lambda(\sigma^{-1} \chi)$.
Each $h_\chi$ is locally constant since the action on $X^*(Z)$ factors through a finite quotient of $\Gamma$.
The assignment $\chi \mapsto h_\chi$ is additive and $\Gamma$-equivariant, since $\lambda$ is additive and $h_{\tau \chi}(\sigma) = h_\chi(\tau^{-1} \sigma)$.
Since $h_\chi(1) = \lambda(\chi)$, the map $\lambda$ lies in the image of evaluation if and only if $h_\chi \in X^*(u)$ for every $\chi \in X^*(Z)$.

By \eqref{eq:character-group-u}, this condition holds for every $\lambda$ if $F$ is non-Archimedean.
If $F = \mathbb{R}$, the condition is $h_\chi(1) + h_\chi(c) = \lambda(\chi) + \lambda(c \chi) = 0$ for every $\chi$, or equivalently $\lambda + \lambda \circ c = 0$.
If $F = \mathbb{C}$, then $X^*(u) = 0$, hence the image is zero.
\end{proof}

Since $X^*(u)$ is torsion by \eqref{eq:character-group-u}, Lemma~\ref{lem:abelianization-colimit} gives canonical bijections
\begin{equation}\label{eq:rigid-colimit}
\varinjlim_{\substack{A \subseteq Z \\ A \text{ finite}/F}} H^1(\mathcal{E}^{\mathrm{rig}}, A \to G)
\xrightarrow{\sim} H^1(\mathcal{E}^{\mathrm{rig}}, Z \to G),
\varinjlim_{\substack{A \subseteq Z \\ A \text{ finite}/F}} H^1_{\mathrm{ab}}(\mathcal{E}^{\mathrm{rig}}, A \to G)
\xrightarrow{\sim} H^1_{\mathrm{ab}}(\mathcal{E}^{\mathrm{rig}}, Z \to G)
\end{equation}
with injective transition maps.
For non-Archimedean $F$, this agrees with the definition in \cite[(3.8)]{Kal18b}.

If $E / F$ is a finite Galois extension splitting a maximal $F$-torus of $G$, the actions on $\pi_1(G)$ and $\tilde{\pi}_1(G, Z)$ both factor through $\Gamma_{E / F}$.

\begin{theorem}[{\cite[Theorem~4.11]{Kal16} for finite $Z$}]\label{thm:local-rigid-formula}
There is a natural isomorphism of abelian groups
\begin{equation}\label{eq:local-rigid-formula}
\kappa_{\mathrm{rig}} \colon H^1_{\mathrm{ab}}(\mathcal{E}^{\operatorname{rig}},Z \to G) \isoto
\varinjlim_{E/F} \frac{\tilde \pi_1(G,Z)^{N_{E/F}}}{I_\Gamma \pi_1(G)},
\end{equation}
where $E/F$ runs over the finite Galois extensions in $\bar F$ such that the $\Gamma$-action on $\tilde \pi_1(G,Z)$ factors through $\Gamma_{E/F}$.
The target is canonically identified with
\begin{equation}\label{eq:local-rigid-targets}
\begin{cases}
\displaystyle\left(\frac{\tilde \pi_1(G,Z)}{I_\Gamma \pi_1(G)}\right)_{\mathrm{tors}} & \text{if $F$ is non-Archimedean},\\[7pt]
\displaystyle \frac{\ker(1 + c \colon \tilde \pi_1(G,Z) \to \tilde \pi_1(G,Z))}{(c - 1) \pi_1(G)} & \text{if $F = \mathbb{R}$},\\[7pt]
0 & \text{if $F = \mathbb{C}$},
\end{cases}
\end{equation}
where $c$ denotes complex conjugation.
Moreover, the following diagram commutes:
\[
\begin{tikzcd}[column sep=large, row sep=small]
H^1_{\mathrm{ab}}(\mathcal{E}^{\operatorname{rig}},Z \to G) \arrow[r,"\kappa_{\mathrm{rig}}","\sim"'] \arrow[d,"\mathcal{N}_{\mathrm{ab}}^{\mathrm{rig}}"'] & \displaystyle \varinjlim_{E/F} \frac{\tilde \pi_1(G,Z)^{N_{E/F}}}{I_\Gamma \pi_1(G)} \arrow[d] \\
\operatorname{Hom}_F(u,Z) \arrow[r,hook,"\operatorname{ev}_1"] & \operatorname{Hom}_{\mathbb{Z}}(X^*(Z),\mathbb{Q}/\mathbb{Z})
\end{tikzcd}.
\]
Here the bottom arrow is the injection of Lemma~\ref{lem:u-nonfinite}, and the right vertical map is induced by the quotient map $\tilde \pi_1(G,Z) \to \tilde \pi_1(G,Z)/\pi_1(G)$ and Proposition~\ref{prop:colimit-quotient}.
\end{theorem}

\begin{proof}
Suppose first that $F \ne \mathbb{C}$.
The finite Galois extensions splitting a fixed maximal $F$-torus of $G$ form a cofinal subsystem of the index set in the statement.
For each finite $F$-subgroup $A \subseteq Z$, the equivalence relation in \eqref{eq:abelianization-relation} is the one used in \cite[Section~3.4]{Kal16}.
Thus the inverse of the bijection in \cite[Theorem~4.11]{Kal16}, with the target described in \cite[(4.9)]{Kal16}, gives a natural bijection
\begin{equation}\label{eq:norm-target}
H^1_{\mathrm{ab}}(\mathcal{E}^{\operatorname{rig}},A \to G) \simeq \varinjlim_{E/F} \frac{\pi_1(G/A)^{N_{E/F}}}{I_\Gamma \pi_1(G)},
\end{equation}
where $E/F$ runs over the index set in the statement.
This bijection preserves the group structure defined in Proposition~\ref{prop:gerbe-abelian-group}.
To see this, let $T \subseteq G$ be a fundamental maximal $F$-torus.
By \cite[Corollary~3.7(2)]{Kal16}, every class on the left of \eqref{eq:norm-target} is the image of a class in $H^1(\mathcal{E}^{\mathrm{rig}}, A \to T)$.
For $T$, the group structure defined in Proposition~\ref{prop:gerbe-abelian-group} is induced by pointwise multiplication of cocycles, since $T^{\mathrm{sc}} = 1$, and Kaletha's bijection is a group isomorphism by \cite[Theorem~4.8]{Kal16}.
The inclusion $T \hookrightarrow G$ induces a homomorphism on cohomology by Lemma~\ref{lem:gerbe-abelian-functoriality} and a homomorphism on the lattice quotients through $X_*(T / A) \to \pi_1(G / A)$.
Lifting two classes to $T$ and applying the naturality in \cite[Theorem~4.11]{Kal16} therefore proves additivity in \eqref{eq:norm-target}.
The same naturality applied to $(A \to G) \to (A' \to G)$ gives compatibility with every inclusion $A \subseteq A' \subseteq Z$ of finite $F$-subgroups.

By Lemma~\ref{lem:finite-quotient}, the transition maps between the groups $\pi_1(G/A)$ are injective.
Thus, for each $E/F$ in the index set,
\[
\tilde \pi_1(G,Z)^{N_{E/F}} = \varinjlim_{\substack{A \subseteq Z \\ A \text{ finite}/F}} \pi_1(G/A)^{N_{E/F}}.
\]
Taking the direct limit over $A$ in \eqref{eq:norm-target}, interchanging the two filtered direct limits, and applying \eqref{eq:rigid-colimit} for $H^1_{\mathrm{ab}}$ proves \eqref{eq:local-rigid-formula} for $F \ne \mathbb{C}$.

Suppose that $F$ is non-Archimedean.
For each finite $F$-subgroup $A \subseteq Z$, the proof of \cite[Proposition~5.3]{Kal16} shows that there is a natural bijection
\[
\varinjlim_{E/F} \frac{\pi_1(G/A)^{N_{E/F}}}{I_\Gamma \pi_1(G)} \simeq \left(\frac{\pi_1(G/A)}{I_\Gamma \pi_1(G)}\right)_{\mathrm{tors}}
\]
with the index set of the statement.
These isomorphisms are compatible with inclusions of finite subgroups $A \subseteq Z$.
Taking the direct limit over $A$, and using the compatibility of filtered direct limits with quotients and torsion subgroups, yields the desired formula in \eqref{eq:local-rigid-targets}.

For $F = \mathbb{R}$, the extension $\mathbb{C}/\mathbb{R}$ is the largest extension in the index set.
Its norm is $1 + c$, and $I_\Gamma \pi_1(G) = (c - 1) \pi_1(G)$, which gives the real formula in \eqref{eq:local-rigid-targets}.
For $F = \mathbb{C}$, we have $u = 1$ and $\Gamma = 1$, so $H^1_{\mathrm{ab}}(\mathcal{E}^{\operatorname{rig}},Z \to G)$ consists of a single point.
The only norm is the identity, whose kernel is zero.
This proves \eqref{eq:local-rigid-formula} and \eqref{eq:local-rigid-targets} in the complex case.

Finally, for finite $A$, \cite[Theorem~4.11 and (4.4)]{Kal16} identifies $\mathcal{N}_{\mathrm{ab}}^{\mathrm{rig}}$, followed by the injection of Lemma~\ref{lem:u-nonfinite}, with the map induced by
\[
\pi_1(G/A) \longrightarrow \pi_1(G/A)/\pi_1(G) \simeq \operatorname{Hom}_{\mathbb{Z}}(X^*(A),\mathbb{Q}/\mathbb{Z}).
\]
This map descends to the quotient by $I_\Gamma \pi_1(G)$.
Passing to the direct limit over $A$ and using Lemma~\ref{lem:u-nonfinite} and Proposition~\ref{prop:colimit-quotient} proves the commutativity of the diagram.
\end{proof}

\begin{corollary}\label{cor:rigid}
Suppose that $F$ is non-Archimedean.
With the group structure induced by abelianization, there is a natural isomorphism of abelian groups
\begin{equation}\label{eq:rigid-formula}
\kappa_{\mathrm{rig}} \circ \operatorname{ab} \colon H^1(\mathcal{E}^{\operatorname{rig}},Z \to G) \isoto \left(\frac{\tilde \pi_1(G,Z)}{I_\Gamma \pi_1(G)}\right)_{\mathrm{tors}},
\end{equation}
for which the following diagram commutes
\begin{equation}\label{eq:first-factor-map}
\begin{tikzcd}[column sep=large, row sep=small]
H^1(\mathcal{E}^{\operatorname{rig}},Z \to G) \arrow[r,"\kappa_{\mathrm{rig}} \circ \operatorname{ab}","\sim"'] \arrow[d,"\mathcal{N}^{\mathrm{rig}}"'] & \displaystyle\left(\frac{\tilde \pi_1(G,Z)}{I_\Gamma \pi_1(G)}\right)_{\mathrm{tors}} \arrow[d] \\
\operatorname{Hom}_F(u,Z) \arrow[r,"\sim"] & \operatorname{Hom}_{\mathbb{Z}}(X^*(Z),\mathbb{Q}/\mathbb{Z})
\end{tikzcd}.
\end{equation}
Here the bottom arrow is the isomorphism of Lemma~\ref{lem:u-nonfinite}, and the right vertical map is induced by the quotient map $\tilde \pi_1(G,Z) \to \tilde \pi_1(G,Z)/\pi_1(G)$ and Proposition~\ref{prop:colimit-quotient}.
\end{corollary}

\begin{proof}
By Corollary~\ref{cor:abelianization-bijective}, the abelianization map is bijective.
The assertions follow from Theorem~\ref{thm:local-rigid-formula} and the identity $\mathcal{N}^{\mathrm{rig}} = \mathcal{N}_{\mathrm{ab}}^{\mathrm{rig}} \circ \operatorname{ab}$.
\end{proof}

\subsection{Comparison with the local rigid gerbe}

We describe $H^1_{\mathrm{ab}}(\mathcal{E}, Z \to G)$ and $H^1(\mathcal{E}, Z \to G)$ in terms of a morphism $f \colon \mathcal{E}^{\mathrm{rig}} \to \mathcal{E}$.

\begin{theorem}\label{thm:main-bis}
Let $f \colon \mathcal{E}^{\mathrm{rig}} \to \mathcal{E}$ be a morphism of gerbes inducing an $F$-homomorphism $\phi \colon u \to D$ on the bands.
Then $f$ induces an isomorphism of abelian groups, natural in $(Z \to G)$,
\begin{equation}\label{eq:main}
H^1_{\mathrm{ab}}(\mathcal{E},Z \to G) \isoto \left(\varinjlim_{E/F} \frac{\tilde \pi_1(G,Z)^{N_{E/F}}}{I_\Gamma \pi_1(G)}\right) \underset{\operatorname{Hom}_F(u,Z)}{\times} \operatorname{Hom}_F(D,Z),
\end{equation}
sending $\operatorname{ab}[x,\nu]$ to $\bigl(\kappa_{\mathrm{rig}}(\operatorname{ab}[x \circ f,\nu \circ \phi]),\nu\bigr)$.
The limit runs over the finite Galois extensions $E/F$ in $\bar F$ such that the $\Gamma$-action on $\tilde \pi_1(G,Z)$ factors through $\Gamma_{E/F}$.
The maps to the base are $\mathcal{N}_{\mathrm{ab}}^{\mathrm{rig}} \circ \kappa_{\mathrm{rig}}^{-1}$, described in Theorem~\ref{thm:local-rigid-formula}, and $\nu \mapsto \nu \circ \phi$ respectively.

Moreover, abelianization and pullback induce a bijection of pointed sets, natural in $(Z \to G)$,
\begin{equation}\label{eq:local-full-recovery}
H^1(\mathcal{E},Z \to G) \isoto H^1_{\mathrm{ab}}(\mathcal{E},Z \to G) \underset{H^1_{\mathrm{ab}}(\mathcal{E}^{\mathrm{rig}},Z \to G)}{\times} H^1(\mathcal{E}^{\mathrm{rig}},Z \to G),
\end{equation}
given by $[x,\nu] \mapsto (\operatorname{ab}[x,\nu],f^*[x,\nu])$.
The maps to the base are $f^*$ and abelianization respectively.

If $F$ is non-Archimedean or complex, $\operatorname{ab}$ is bijective, and composition with \eqref{eq:main} gives a formula for $H^1(\mathcal{E}, Z \to G)$.
For $F = \mathbb{C}$, the Newton map is an isomorphism of abelian groups
\[
\mathcal{N} \colon H^1(\mathcal{E},Z \to G) \isoto \operatorname{Hom}_{\mathbb{C}}(D,Z).
\]

The dependence of these bijections on $f$ is described by Lemma~\ref{lem:comparison-dependence}.
In particular, for fixed $\phi$, they are independent of $f$ whenever $H^1(F,D) = 0$.
If $D$ is of finite type, a morphism of gerbes $f$ with a suitable band map $\phi$ exists by Corollary~\ref{cor:comparison-exist-finite-type}.
\end{theorem}

\begin{proof}
By Lemma~\ref{lem:gerbe-abelian-functoriality} and the additivity of $\mathcal{N}_{\mathrm{ab}}$, the bijection \eqref{eq:comparison-formal-ab} for $f$ is an isomorphism of abelian groups
\[
H^1_{\mathrm{ab}}(\mathcal{E},Z \to G) \isoto H^1_{\mathrm{ab}}(\mathcal{E}^{\mathrm{rig}},Z \to G) \underset{\operatorname{Hom}_F(u,Z)}{\times} \operatorname{Hom}_F(D,Z),
\]
where the maps to the base are $\mathcal{N}_{\mathrm{ab}}^{\mathrm{rig}}$ and $\nu \mapsto \nu \circ \phi$.
Applying Theorem~\ref{thm:local-rigid-formula} to the first factor proves \eqref{eq:main}, including the description of the bijection and the maps to the base.
Formula \eqref{eq:local-full-recovery} follows directly from Proposition~\ref{prop:comparison-formal-ab}.

Corollary~\ref{cor:abelianization-bijective} gives the desired bijectivity of abelianization for non-Archimedean or complex $F$.
For $F = \mathbb{C}$, the first factor and the base in \eqref{eq:main} are trivial, so the isomorphism is the Newton map.

The dependence on $f$ follows from Lemma~\ref{lem:comparison-dependence} and the functoriality of abelianization.
Existence of a comparison when $D$ is of finite type follows from Corollary~\ref{cor:comparison-exist-finite-type}.
All constructions are natural in $(Z \to G)$.
\end{proof}

For a given $\nu \in \operatorname{Hom}_F(D, Z)$, Theorems~\ref{thm:local-rigid-formula} and~\ref{thm:main-bis} give a criterion in terms of $\tilde{\pi}_1(G, Z)$.
Namely, $\nu$ lies in the image of $\mathcal{N}$ if and only if there are $E / F$ as in \eqref{eq:main} and $y \in \tilde{\pi}_1(G, Z)^{N_{E / F}}$ whose image in $\tilde{\pi}_1(G, Z) / \pi_1(G)$ corresponds to $\operatorname{ev}_1(\nu \circ \phi)$ under Proposition~\ref{prop:colimit-quotient}.
By Proposition~\ref{prop:general-newton-fibers}, this is equivalent to $\delta_{\mathcal{E}, \mathrm{ab}}(\nu) = 0$.
When it holds, the possible first coordinates in \eqref{eq:main} form a coset of the subgroup obtained by taking $Z = 1$ in \eqref{eq:local-rigid-formula}.
The difference of two such representatives at a common extension belongs to $\pi_1(G)$ and has norm zero.
Conversely, adding any such element preserves the image in the quotient.
Writing $M \coloneqq \pi_1(G)$, we obtain
\[
\# \mathcal{N}_{\mathrm{ab}}^{-1}(\nu) =
\begin{cases}
\lvert M_{\Gamma, \mathrm{tors}} \rvert & \text{if $F$ is non-Archimedean}, \\
\bigl\lvert \ker(1 + c \colon M \to M) / (c - 1) M \bigr\rvert & \text{if $F = \mathbb{R}$}, \\
1 & \text{if $F = \mathbb{C}$},
\end{cases}
\]
by \eqref{eq:local-rigid-targets}.
This also counts the classes in $\mathcal{N}^{-1}(\nu)$ when $F$ is non-Archimedean or complex.
Over $\mathbb{R}$, a choice of cocycle $(x, \nu)$ representing a class in $\mathcal{N}^{-1}(\nu)$ identifies this fiber with $H^1(\mathbb{R}, G_x)$ by \eqref{eq:twisting-restriction-fiber}, where $G_x$ is the twist of $G$ by $\bar{x}$.

\begin{remark}
For $\mathcal{E} = \mathcal{E}^{\mathrm{rig}}$ and $f = \operatorname{id}$, the first isomorphism specializes to Theorem~\ref{thm:local-rigid-formula}.
For finite $Z$, this is the parametrization of \cite[Theorem~4.11]{Kal16}.
\end{remark}

\subsection{Specialization to Kottwitz's local gerbe}\label{subsec:local-kottwitz}

We recall Kottwitz's local gerbe from \cite[Sections~5 and~10]{Kot14}, using the extension over $\Gamma$ constructed in \cite[Section~3.6]{KT23}.
Its band is the pro-torus
\[
D^{\mathrm{iso}} = \varprojlim_{K/F} \mathbb{G}_m,
\]
where $K/F$ runs over the finite Galois extensions in $\bar F$ and the transition map associated with $K \subseteq K'$ is $t \mapsto t^{[K' : K]}$.
Its character module has trivial $\Gamma$-action and is canonically identified with
\begin{equation}\label{eq:Kott-band-char-mod}
X^*(D^{\mathrm{iso}}) =
\begin{cases}
\mathbb{Q} & \text{if $F$ is non-Archimedean},\\
\frac12 \mathbb{Z} & \text{if $F = \mathbb{R}$},\\
\mathbb{Z} & \text{if $F = \mathbb{C}$};
\end{cases}
\end{equation}
see \cite[Section~2.2]{KT23}.

The local fundamental classes determine a distinguished class $\xi^{\mathrm{iso}} \in H^2(F,D^{\mathrm{iso}})$; see \cite[Section~3.6]{KT23}.
Choose a Galois gerbe
\[
1 \longrightarrow D^{\mathrm{iso}}(\bar F)
\longrightarrow \mathcal{E}^{\mathrm{iso}}
\longrightarrow \Gamma
\longrightarrow 1
\]
whose isomorphism class corresponds to $\xi^{\mathrm{iso}}$.
Since the band at each finite stage is $\mathbb{G}_m$, Hilbert's theorem~90 and \cite[Section~3.1]{KT23} show that the sets $H^1(\mathcal{E}^{\mathrm{iso}},G)$ and $H^1(\mathcal{E}^{\mathrm{iso}},Z \to G)$ are independent of the choice of $\mathcal{E}^{\mathrm{iso}}$ up to canonical bijection.
By \cite[Lemma~3.12.1]{KT23}, the algebraic cohomology of $\mathcal{E}^{\mathrm{iso}}$ agrees with Kottwitz's $B(G)$.
We write
\[
B(Z \to G) \coloneqq H^1(\mathcal{E}^{\mathrm{iso}},Z \to G).
\]

Since $D^{\mathrm{iso}}$ is a pro-torus, every homomorphism $D^{\mathrm{iso}} \to Z$ factors through $Z^\circ$.
Thus
\begin{equation}\label{eq:kottwitz-connected-center}
B(Z \to G) = B(Z^\circ \to G).
\end{equation}
For finite $Z$, \eqref{eq:kottwitz-connected-center} and Proposition~\ref{prop:gerbe-inflation-restriction} identify $B(Z \to G)$ with $H^1(F, G)$.
For $Z = Z(G)$, it is Kottwitz's basic set $B(G)_{\mathrm{bas}}$; see \cite[Section~10.8]{Kot14}.

By \eqref{eq:Kott-band-char-mod}, we have a natural identification
\begin{equation}\label{eq:kottwitz-local-hom-identification}
\operatorname{Hom}_F(D^{\mathrm{iso}},Z)
\simeq
\operatorname{Hom}_\Gamma\bigl(X^*(Z),X^*(D^{\mathrm{iso}})\bigr)
\subseteq
\operatorname{Hom}_\Gamma(X^*(Z),\mathbb{Q}).
\end{equation}
For $\nu \in \operatorname{Hom}_F(D^{\mathrm{iso}},Z)$, we therefore regard its pullback $\nu^*$ as a homomorphism $X^*(Z) \to \mathbb{Q}$ and write
\[
\iota(\nu^*) \in \tilde \pi_1(G,Z)^\Gamma
\]
for its image under \eqref{eq:kottwitz-central-cocharacters}.

Let $\phi \colon u \to D^{\mathrm{iso}}$ be the band map described in \cite[proof of Proposition~3.2]{Kal18b} and \cite[Section~2.2]{KT23}.
Under \eqref{eq:Kott-band-char-mod} and \eqref{eq:character-group-u}, its map on character modules is
\begin{equation}\label{eq:kottwitz-band-characters}
\phi^* \colon X^*(D^{\mathrm{iso}}) \longrightarrow X^*(u),
\qquad
q \longmapsto \bigl(\sigma \longmapsto -q + \mathbb{Z}\bigr).
\end{equation}
Fix a morphism of gerbes
\begin{equation}\label{eq:kottwitz-comparison-map}
  f \colon \mathcal{E}^{\mathrm{rig}} \longrightarrow \mathcal{E}^{\mathrm{iso}}
\end{equation}
inducing $\phi$.
Its existence follows from \cite[(3.13)]{Kal18b} and \cite[Section~3.6]{KT23}.

The induced maps on central cohomology are independent of the choice of $f$ inducing $\phi$.
Every $\nu \in \operatorname{Hom}_F(D^{\mathrm{iso}},Z)$ factors through a finite-stage quotient $\mathbb{G}_m$.
The image of any $\alpha \in Z^1(F,D^{\mathrm{iso}})$ in this quotient is a coboundary by Hilbert's theorem~90, so $\nu \circ \alpha$ is a coboundary in $Z$.
Lemma~\ref{lem:comparison-dependence} gives the desired independence, both before and after abelianization.

There is a commutative diagram
\begin{equation}\label{eq:kottwitz-local-band-map}
\begin{tikzcd}[ampersand replacement=\&,column sep=large]
\& \operatorname{Hom}_F(D^{\mathrm{iso}},Z) \& \operatorname{Hom}_F(u,Z) \\
\operatorname{Hom}_{\mathbb{Z}}(X^*(Z),\mathbb{Z}) \& \operatorname{Hom}_{\mathbb{Z}}(X^*(Z),\mathbb{Q}) \& \operatorname{Hom}_{\mathbb{Z}}(X^*(Z),\mathbb{Q}/\mathbb{Z}) \\
\pi_1(G) \& \tilde \pi_1(G,Z) \& \tilde \pi_1(G,Z)/\pi_1(G) \arrow["{\nu \mapsto \nu \circ \phi}",from=1-2,to=1-3] \arrow["{\nu \mapsto -\nu^*}"',hook,from=1-2,to=2-2] \arrow["{\operatorname{ev}_1}",hook,from=1-3,to=2-3] \arrow[hook,from=2-1,to=2-2] \arrow["i"',hook,from=2-1,to=3-1] \arrow["{\lambda \mapsto \lambda \bmod \mathbb{Z}}",from=2-2,to=2-3] \arrow["\iota"',hook,from=2-2,to=3-2] \arrow["\sim",from=2-3,to=3-3] \arrow[hook,from=3-1,to=3-2] \arrow[from=3-2,to=3-3]
\end{tikzcd}
\end{equation}
Here $\operatorname{ev}_1$ is the injection of Lemma~\ref{lem:u-nonfinite}, and the isomorphism in the right column is the inverse of the identification in Proposition~\ref{prop:colimit-quotient}.
The lower squares commute by the construction of $\iota$.
The upper square commutes by \eqref{eq:kottwitz-band-characters}, since both composites send $\nu$ to the homomorphism $\chi \mapsto -\nu^*(\chi) + \mathbb{Z}$.

Applying Theorem~\ref{thm:main-bis} to $f$ yields a natural isomorphism of abelian groups
\begin{equation}\label{eq:kottwitz-local-fiber-product}
H^1_{\mathrm{ab}}(\mathcal{E}^{\mathrm{iso}},Z \to G)
\xrightarrow{\sim}
\left(\varinjlim_{E/F}
\frac{\tilde \pi_1(G,Z)^{N_{E/F}}}{I_\Gamma \pi_1(G)}\right)
\underset{\operatorname{Hom}_F(u,Z)}{\times}
\operatorname{Hom}_F(D^{\mathrm{iso}},Z).
\end{equation}
The index set and the map from the first factor are those of Theorem~\ref{thm:local-rigid-formula}, and the map from the second factor is $\nu \mapsto \nu \circ \phi$.
For $y \in \tilde \pi_1(G,Z)^{N_{E/F}}$ and $\nu \in \operatorname{Hom}_F(D^{\mathrm{iso}},Z)$, the diagram \eqref{eq:kottwitz-local-band-map} shows that $([y],\nu)$ belongs to this fiber product if and only if
\begin{equation}\label{eq:kottwitz-local-integrality}
y + \iota(\nu^*) \in \pi_1(G).
\end{equation}
This condition is independent of the representative $y$, since $I_\Gamma \pi_1(G) \subseteq \pi_1(G)$.

Kottwitz's norm map is given by
\begin{equation}\label{eq:kottwitz-local-norm}
N_G^\diamond \colon \pi_1(G)_\Gamma
\longrightarrow
\bigl(\pi_1(G) \otimes_{\mathbb{Z}} X^*(D^{\mathrm{iso}})\bigr)^\Gamma,
\qquad
[m] \longmapsto N_{E/F} m \otimes \frac1{[E : F]},
\end{equation}
where $E/F$ is a finite Galois extension through which the $\Gamma$-action on $\pi_1(G)$ factors.
This expression is independent of $E$ and of the representative $m$; see \cite[Definition~11.3 and (11.12)]{Kot14}.
Since $X^*(D^{\mathrm{iso}})$ is torsion-free, tensoring $i$ with this module and taking $\Gamma$-invariants yields an injection
\[
j_Z \colon \operatorname{Hom}_F(D^{\mathrm{iso}},Z)
\simeq
\bigl(\operatorname{Hom}_{\mathbb{Z}}(X^*(Z),\mathbb{Z})
\otimes_{\mathbb{Z}} X^*(D^{\mathrm{iso}})\bigr)^\Gamma
\xrightarrow{(i \otimes 1)^\Gamma}
\bigl(\pi_1(G) \otimes_{\mathbb{Z}} X^*(D^{\mathrm{iso}})\bigr)^\Gamma.
\]
Define
\begin{equation}\label{eq:kottwitz-central-target}
A(F,Z \to G) \coloneqq
\pi_1(G)_\Gamma
\underset{\left(\pi_1(G) \otimes_{\mathbb{Z}} X^*(D^{\mathrm{iso}})\right)^\Gamma}{\times}
\operatorname{Hom}_F(D^{\mathrm{iso}},Z),
\end{equation}
where the maps to the base are $N_G^\diamond$ and $j_Z$.
Since $j_Z$ is injective, the first projection gives the identification
\begin{equation}\label{eq:kottwitz-central-target-2}
A(F,Z \to G)
=
\{a \in \pi_1(G)_\Gamma \mid N_G^\diamond(a) \in \operatorname{im}(j_Z)\}.
\end{equation}
We use this identification below.
For $a \in A(F,Z \to G)$, denote by $\nu_a \in \operatorname{Hom}_F(D^{\mathrm{iso}},Z)$ the unique homomorphism satisfying $j_Z(\nu_a) = N_G^\diamond(a)$.

\begin{proposition}\label{prop:kottwitz-local-recovery}
There is an isomorphism of abelian groups, natural in $(Z \to G)$,
\[
\left(\varinjlim_{E/F}
\frac{\tilde \pi_1(G,Z)^{N_{E/F}}}{I_\Gamma \pi_1(G)}\right)
\underset{\operatorname{Hom}_F(u,Z)}{\times}
\operatorname{Hom}_F(D^{\mathrm{iso}},Z)
\xrightarrow{\sim}
A(F,Z \to G),
\qquad
([y],\nu) \longmapsto [y + \iota(\nu^*)].
\]
Its inverse sends $a = [m] \in A(F,Z \to G)$ to $\bigl([m - \iota(\nu_a^*)],\nu_a\bigr)$.
Consequently, \eqref{eq:kottwitz-local-fiber-product} induces an isomorphism of abelian groups, natural in $(Z \to G)$,
\begin{equation}\label{eq:kottwitz-local-ab}
H^1_{\mathrm{ab}}(\mathcal{E}^{\mathrm{iso}},Z \to G)
\xrightarrow{\sim} A(F,Z \to G).
\end{equation}
\end{proposition}

\begin{proof}
Since $X^*(D^{\mathrm{iso}})$ is torsion-free, the inclusion $\pi_1(G) \hookrightarrow \tilde \pi_1(G,Z)$ remains injective after tensoring with $X^*(D^{\mathrm{iso}})$.
We use this injection to regard $N_G^\diamond$ and $j_Z$ as maps to
\[
\bigl(\tilde \pi_1(G,Z) \otimes_{\mathbb{Z}} X^*(D^{\mathrm{iso}})\bigr)^\Gamma.
\]
For $\nu \in \operatorname{Hom}_F(D^{\mathrm{iso}},Z)$, choose $n$ such that $1/n \in X^*(D^{\mathrm{iso}})$ and $\lambda \coloneqq n \nu^* \in \operatorname{Hom}_{\mathbb{Z}}(X^*(Z),\mathbb{Z})$.
Since $\iota$ extends $i$, we have
\[
j_Z(\nu)
=
i(\lambda) \otimes \frac1n
=
\iota(\nu^*) \otimes 1.
\]

Let $([y], \nu)$ belong to the fiber product and choose $y$ representing $[y]$ with $N_{E / F} y = 0$.
By \eqref{eq:kottwitz-local-integrality},
\[
m \coloneqq y + \iota(\nu^*) \in \pi_1(G).
\]
Writing $[m] \in \pi_1(G)_\Gamma$ for its class, we obtain
\[
N_G^\diamond([m])
=
N_{E/F} m \otimes \frac1{[E : F]}
=
\iota(\nu^*) \otimes 1
=
j_Z(\nu),
\]
since $N_{E/F} y = 0$ and $\iota(\nu^*) \in \tilde \pi_1(G,Z)^\Gamma$.
Thus $[m] \in A(F,Z \to G)$ and $\nu_{[m]} = \nu$.

Conversely, let $a = [m] \in A(F,Z \to G)$, with $m \in \pi_1(G)$, and set
\[
y \coloneqq m - \iota(\nu_a^*) \in \tilde \pi_1(G,Z).
\]
Choose a finite Galois extension $E/F$ through which the $\Gamma$-action on $\tilde \pi_1(G,Z)$ factors.
Then
\begin{equation}\label{eq:kottwitz-local-norm-comparison}
N_{E/F} y \otimes \frac1{[E : F]}
=
N_G^\diamond(a) - j_Z(\nu_a)
=
0 \in \bigl(\tilde \pi_1(G,Z) \otimes_{\mathbb{Z}} X^*(D^{\mathrm{iso}})\bigr)^\Gamma.
\end{equation}

Such extensions form a cofinal subset of the finite Galois extensions of $F$, ordered by inclusion.
Using $X^*(D^{\mathrm{iso}}) \simeq \varinjlim_{E/F} \mathbb{Z}$ and the fact that tensor products commute with filtered colimits, we obtain
\[
\tilde \pi_1(G,Z) \otimes_{\mathbb{Z}} X^*(D^{\mathrm{iso}})
\simeq
\varinjlim_{E/F} \tilde \pi_1(G,Z),
\]
where the transition maps in both colimits associated with $E \subseteq E'$ are multiplication by $[E' : E]$.
This isomorphism is induced by the compatible maps
\[
\tilde \pi_1(G,Z)
\longrightarrow
\tilde \pi_1(G,Z) \otimes_{\mathbb{Z}} X^*(D^{\mathrm{iso}}),
\qquad
v \longmapsto v \otimes \frac1{[E : F]}.
\]
Thus \eqref{eq:kottwitz-local-norm-comparison} implies that, for some finite Galois extension $E'/F$ containing $E$,
\[
N_{E'/F} y = [E' : E] N_{E/F} y = 0
\qquad \text{in } \tilde \pi_1(G,Z).
\]
It follows that the class $[y] \in \tilde \pi_1(G,Z)/I_\Gamma \pi_1(G)$ belongs to
\[
\varinjlim_{E/F}
\frac{\tilde \pi_1(G,Z)^{N_{E/F}}}{I_\Gamma \pi_1(G)}.
\]
Moreover, $y + \iota(\nu_a^*) = m \in \pi_1(G)$, hence \eqref{eq:kottwitz-local-integrality} shows that $([y],\nu_a)$ belongs to the fiber product.

The norms annihilate $I_\Gamma \pi_1(G)$, and changing either representative $y$ or $m$ by an element of this subgroup changes the corresponding representative by the same element.
Thus both maps are well defined.
Their formulas and the equality $\nu_{[m]} = \nu$ show that they are mutually inverse.
The forward map is a homomorphism since $\nu \mapsto \nu^*$ and $\iota$ are additive.
All the constructions are natural in $(Z \to G)$.
\end{proof}

\begin{lemma}\label{lem:kottwitz-central-target}
Let $q \colon G \to G/Z$ be the quotient homomorphism.
Then
\[
A(F,Z \to G)
=
\ker \left[
\pi_1(G)_\Gamma
\xrightarrow{q_*} \pi_1(G/Z)_\Gamma
\xrightarrow{N_{G/Z}^\diamond}
\bigl(\pi_1(G/Z) \otimes_{\mathbb{Z}} X^*(D^{\mathrm{iso}})\bigr)^\Gamma
\right].
\]
\end{lemma}

\begin{proof}
Tensoring \eqref{eq:central-quotient-fundamental-group} with the torsion-free module $X^*(D^{\mathrm{iso}})$ and taking $\Gamma$-invariants identifies $\operatorname{im}(j_Z)$ with the kernel of $(q_* \otimes 1)^\Gamma$.
For $a \in \pi_1(G)_\Gamma$, functoriality of the norm gives
\[
(q_* \otimes 1)^\Gamma\bigl(N_G^\diamond(a)\bigr)
=
N_{G/Z}^\diamond(q_* a) \in \bigl(\pi_1(G/Z) \otimes_{\mathbb{Z}} X^*(D^{\mathrm{iso}})\bigr)^\Gamma.
\]
By \eqref{eq:kottwitz-central-target-2}, it follows that $a \in A(F,Z \to G)$ if and only if $N_{G/Z}^\diamond(q_* a) = 0$.
\end{proof}

\begin{theorem}\label{thm:kottwitz-local-explicit}
  \begin{enumerate}
    \item If $F$ is non-Archimedean, then there is an isomorphism of abelian groups, natural in $(Z \to G)$,
\begin{equation}\label{eq:kottwitz-padic-central-formula}
B(Z \to G)
\xrightarrow{\sim}
\ker \left[
\pi_1(G)_\Gamma
\longrightarrow
\pi_1(G/Z)_{\Gamma,\mathrm{free}}
\right],
\end{equation}
where the map in the kernel is induced by $G \to G/Z$, and the subscript $\mathrm{free}$ denotes the quotient by the torsion subgroup.

\item For $F = \mathbb{R}$, let $c$ denote complex conjugation.
Then
\begin{equation}\label{eq:kottwitz-real-central-target}
A(\mathbb{R},Z \to G)
=
\left\{
[m] \in \frac{\pi_1(G)}{(c - 1) \pi_1(G)}
\; \middle| \;
(1 + c) m \in i\bigl(\operatorname{Hom}_\Gamma(X^*(Z),\mathbb{Z})\bigr)
\right\}.
\end{equation}
Moreover, abelianization and pullback induce a bijection of pointed sets, natural in $(Z \to G)$,
\begin{equation}\label{eq:kottwitz-real-full}
B(Z \to G)
\xrightarrow{\sim}
A(\mathbb{R},Z \to G)
\underset{H^1_{\mathrm{ab}}(\mathcal{E}^{\mathrm{rig}},Z \to G)}{\times}
H^1(\mathcal{E}^{\mathrm{rig}},Z \to G).
\end{equation}
The map from $A(\mathbb{R},Z \to G)$ to the base is
\[
a = [m] \longmapsto
\kappa_{\mathrm{rig}}^{-1}\bigl([m - \iota(\nu_a^*)]\bigr),
\]
and the other map is abelianization.

\item For $F = \mathbb{C}$, the Newton map induces an isomorphism of abelian groups, natural in $(Z \to G)$,
\[
B(Z \to G)
\xrightarrow{\sim}
\operatorname{Hom}_{\mathbb{Z}}(X^*(Z),\mathbb{Z}),
\qquad
[x,\nu] \longmapsto \nu^*.
\]
  \end{enumerate}
\end{theorem}

\begin{proof}
Suppose first that $F$ is non-Archimedean, so that $X^*(D^{\mathrm{iso}}) = \mathbb{Q}$.
Choose a finite Galois extension $E/F$ through which the $\Gamma$-action on $\pi_1(G/Z)$ factors.
The map
\[
\pi_1(G/Z)_\Gamma \otimes_{\mathbb{Z}} \mathbb{Q}
\xrightarrow{\sim}
\bigl(\pi_1(G/Z) \otimes_{\mathbb{Z}} \mathbb{Q}\bigr)^\Gamma,
\qquad
[m] \otimes r \longmapsto
N_{E/F} m \otimes \frac{r}{[E : F]},
\]
is an isomorphism, whose inverse is the restriction of the natural projection
\[
\pi_1(G/Z) \otimes_{\mathbb{Z}} \mathbb{Q}
\longrightarrow
\pi_1(G/Z)_\Gamma \otimes_{\mathbb{Z}} \mathbb{Q}.
\]
Here $N_{E/F}$ acts as multiplication by $[E : F]$ on both the invariants and the coinvariants.
By definition, $N_{G/Z}^\diamond$ is the composite of the natural map
\[
\pi_1(G/Z)_\Gamma
\longrightarrow
\pi_1(G/Z)_\Gamma \otimes_{\mathbb{Z}} \mathbb{Q}
\]
with this isomorphism.
Its kernel is therefore $\pi_1(G/Z)_{\Gamma,\mathrm{tors}}$.
Lemma~\ref{lem:kottwitz-central-target}, Proposition~\ref{prop:kottwitz-local-recovery}, and Corollary~\ref{cor:abelianization-bijective} therefore give the isomorphism in \eqref{eq:kottwitz-padic-central-formula}, with the group structure on $B(Z \to G)$ induced by abelianization.

Suppose next that $F = \mathbb{R}$.
Under the $\Gamma$-equivariant identification
\[
\pi_1(G) \otimes_{\mathbb{Z}} \tfrac12 \mathbb{Z}
\xrightarrow{\sim} \pi_1(G),
\qquad
m \otimes \tfrac n2 \longmapsto n m,
\]
the elements $N_G^\diamond([m])$ and $j_Z(\nu)$ correspond to $(1 + c) m$ and $i(2 \nu^*)$ respectively, in $\pi_1(G)^\Gamma$.
Moreover, the map $\nu \mapsto 2 \nu^*$ identifies $\operatorname{Hom}_{\mathbb{R}}(D^{\mathrm{iso}},Z)$ with $\operatorname{Hom}_\Gamma(X^*(Z),\mathbb{Z})$.
Formula \eqref{eq:kottwitz-real-central-target} therefore follows from \eqref{eq:kottwitz-central-target-2}.
Its condition is independent of the representative $m$, since $(1 + c)(c - 1) = 0$.
Combining \eqref{eq:local-full-recovery} with Proposition~\ref{prop:kottwitz-local-recovery} yields \eqref{eq:kottwitz-real-full}, with the stated maps to the base.

Finally, suppose that $F = \mathbb{C}$.
Then $\Gamma = 1$ and $D^{\mathrm{iso}} = \mathbb{G}_m$.
By Theorem~\ref{thm:main-bis}, the Newton map identifies $B(Z \to G)$ with $\operatorname{Hom}_{\mathbb{C}}(\mathbb{G}_m,Z) \simeq \operatorname{Hom}_{\mathbb{Z}}(X^*(Z),\mathbb{Z})$.
\end{proof}

Let $\kappa_G \colon B(G) \to \pi_1(G)_\Gamma$ denote Kottwitz's map of \cite[Section~11.5]{Kot14}.
It is functorial in $G$.
For an $F$-torus $T$, the inverse of $\kappa_T$ is the isomorphism $X_*(T)_\Gamma \xrightarrow{\sim} B(T)$ obtained from the local fundamental classes by \cite[Lemma~5.1]{Kot14} and passage to the direct limit.

\begin{corollary}\label{cor:kottwitz-local-map}
Let $\Theta_{G, Z}$ denote the isomorphism in \eqref{eq:kottwitz-local-ab}.
With the inclusion on the right given by \eqref{eq:kottwitz-central-target-2}, the following diagram commutes.
\[
\begin{tikzcd}[column sep=large]
B(Z \to G) \arrow[r, "\operatorname{ab}"] \arrow[d, hook]
  & H^1_{\mathrm{ab}}(\mathcal{E}^{\mathrm{iso}}, Z \to G) \arrow[r, "{\Theta_{G,Z}}", "\sim"']
  & A(F, Z \to G) \arrow[d, hook] \\
B(G) \arrow[rr, "\kappa_G"'] & & \pi_1(G)_\Gamma.
\end{tikzcd}
\]
\end{corollary}

\begin{proof}
Suppose first that $G = Z = T$ is a torus, and set $Y \coloneqq X_*(T)$.
Let $b = [x, \nu] \in B(T)$ and choose $m \in Y$ representing $\kappa_T(b)$.
By \cite[Proposition~11.5]{Kot14} and \eqref{eq:kottwitz-local-norm}, for any finite Galois extension $E / F$ splitting $T$ we have
\[
\nu^* = \frac{1}{[E : F]} N_{E / F} m \quad \text{in } Y \otimes_{\mathbb{Z}} \mathbb{Q}.
\]
If $F$ is non-Archimedean or real, \cite[Propositions~3.13.5 and~3.13.6]{KT23} give
$\kappa_{\mathrm{rig}}\bigl(\operatorname{ab}(f^* b)\bigr) = [m - \nu^*]$.
The same equality holds over $\mathbb{C}$, since $\nu^* = m$ and the target of $\kappa_{\mathrm{rig}}$ is zero by \eqref{eq:local-rigid-targets}.
Under $\tilde{\pi}_1(T, T) = Y \otimes_{\mathbb{Z}} \mathbb{Q}$, the map $\iota$ is the identity.
Proposition~\ref{prop:kottwitz-local-recovery} therefore gives
\[
\Theta_{T,T}(\operatorname{ab}(b)) = [(m - \nu^*) + \nu^*] = [m] = \kappa_T(b).
\]
For any $F$-subgroup $Z \subseteq T$, naturality for $(Z \to T) \to (T \to T)$ gives the assertion on $B(Z \to T)$, since the induced map $A(F, Z \to T) \to Y_\Gamma$ is the inclusion in \eqref{eq:kottwitz-central-target-2}.

Now let $b \in B(Z \to G)$, and write $\nu \colon D^{\mathrm{iso}} \to Z$ for its restriction to the band.
Since $b$ is basic, it lifts to $b_T \in B(T)$ for an elliptic maximal $F$-torus $T \subseteq G$ if $F$ is non-Archimedean, and for a fundamental maximal $F$-torus $T \subseteq G$ if $F = \mathbb{R}$, by \cite[Proposition~13.1(2) and Lemma~13.2]{Kot14} respectively.
In either case, the composite of the band homomorphism of $b_T$ with $T \hookrightarrow G$ is $G(\bar{F})$-conjugate to $\nu$ and hence equals $\nu$, since $Z$ is central in $G$.
Thus $b_T \in B(Z \to T)$.
If $F = \mathbb{C}$, the central cocharacter $\nu \colon \mathbb{G}_m \to Z$ itself gives such a lift for any maximal torus $T \subseteq G$, since $\Gamma$ is trivial.
Writing $j \colon T \hookrightarrow G$ for the inclusion and using \eqref{eq:kottwitz-central-target-2}, we obtain
\[
\Theta_{G,Z}(\operatorname{ab}(b)) = j_* \Theta_{T,Z}(\operatorname{ab}(b_T)) = j_* \kappa_T(b_T) = \kappa_G(b),
\]
by naturality of Proposition~\ref{prop:kottwitz-local-recovery}, the torus case, and functoriality of $\kappa_G$.
\end{proof}

\begin{remark}
Suppose that $F$ is non-Archimedean.
By \eqref{eq:kottwitz-central-target}, a homomorphism $\nu \colon D^{\mathrm{iso}} \to Z$ occurs in $B(Z \to G)$ if and only if $j_Z(\nu) \in \operatorname{im}(N_G^\diamond)$.
For each such $\nu$, the Kottwitz invariants of the classes above $\nu$ form one coset of $\pi_1(G)_{\Gamma, \mathrm{tors}}$ in $\pi_1(G)_\Gamma$, since $\ker(N_G^\diamond) = \pi_1(G)_{\Gamma, \mathrm{tors}}$.
Taking $Z = Z(G)$ gives an isomorphism of abelian groups
\[
B(G)_{\mathrm{bas}}
\xrightarrow{\sim} \pi_1(G)_\Gamma,
\]
since $\pi_1(G / Z(G))$ is finite, as in \cite[Proposition~13.1(1)]{Kot14}.

For finite $Z$, the central isogeny $q \colon G \to G/Z$ induces an isomorphism
\[
q_* \otimes 1 \colon
\pi_1(G) \otimes_{\mathbb{Z}} \mathbb{Q}
\xrightarrow{\sim}
\pi_1(G/Z) \otimes_{\mathbb{Z}} \mathbb{Q}.
\]
Taking $\Gamma$-coinvariants yields an isomorphism
$\pi_1(G)_\Gamma \otimes_{\mathbb{Z}} \mathbb{Q}
\xrightarrow{\sim}
\pi_1(G/Z)_\Gamma \otimes_{\mathbb{Z}} \mathbb{Q}$.
Thus the kernel in \eqref{eq:kottwitz-padic-central-formula} is $\pi_1(G)_{\Gamma,\mathrm{tors}}$.
Since inflation identifies $H^1(F, G)$ with $B(Z \to G)$ by Proposition~\ref{prop:gerbe-inflation-restriction}, this recovers the isomorphism of abelian groups
\[
H^1(F,G) \xrightarrow{\sim} \pi_1(G)_{\Gamma,\mathrm{tors}}.
\]
\end{remark}

\subsection{Specialization to Fargues's extended Kottwitz gerbe}\label{subsec:local-extended-kottwitz}

We consider the Galois gerbe
\[
\mathcal{E}^{\mathrm{ext}}
\coloneqq
\mathcal{E}^{\mathrm{iso}} \times_\Gamma \mathcal{E}^{\mathrm{rig}}
\]
banded by $D^{\mathrm{iso}} \times u$.
For non-Archimedean $F$, this product represents Fargues's extended Kottwitz gerbe by \cite[Theorem~5.6]{Far22}.
For $F = \mathbb{R}$, it corresponds to the product used in \cite[Definition~13.8]{Far22}.

Let $f \colon \mathcal{E}^{\mathrm{rig}} \to \mathcal{E}^{\mathrm{iso}}$ be the morphism of gerbes fixed in \eqref{eq:kottwitz-comparison-map}, and let $\phi \colon u \to D^{\mathrm{iso}}$ be its band map, as in \eqref{eq:kottwitz-band-characters}.
It induces a morphism
\[
f^{\mathrm{ext}} \colon
\mathcal{E}^{\mathrm{rig}} \longrightarrow \mathcal{E}^{\mathrm{ext}},
\qquad
w \longmapsto (f(w), w),
\]
whose band map is $t \mapsto (\phi(t), t)$.

Every $F$-homomorphism $\eta \colon D^{\mathrm{iso}} \times u \to Z$ can be written uniquely as
\[
\eta(d, t) = \nu(d) \mu(t),
\qquad
\nu \in \operatorname{Hom}_F(D^{\mathrm{iso}}, Z),
\quad
\mu \in \operatorname{Hom}_F(u, Z).
\]
As in \eqref{eq:kottwitz-local-hom-identification}, we identify $\nu$ with its character map
\[
\nu^* \in \operatorname{Hom}_\Gamma\bigl(X^*(Z), X^*(D^{\mathrm{iso}})\bigr) \subseteq \operatorname{Hom}_\Gamma\bigl(X^*(Z), \mathbb{Q}\bigr).
\]

\begin{proposition}\label{prop:extended-kottwitz-formula}
Pullback along $f^{\mathrm{ext}}$ and restriction to $D^{\mathrm{iso}}$ induce a natural bijection of pointed sets
\begin{equation}\label{eq:extended-kottwitz-product}
H^1(\mathcal{E}^{\mathrm{ext}}, Z \to G)
\xrightarrow{\sim}
H^1(\mathcal{E}^{\mathrm{rig}}, Z \to G)
\times
\operatorname{Hom}_F(D^{\mathrm{iso}}, Z)
\end{equation}
and, after abelianization, a natural isomorphism of abelian groups
\begin{equation}\label{eq:extended-kottwitz-formula}
H^1_{\mathrm{ab}}(\mathcal{E}^{\mathrm{ext}}, Z \to G)
\xrightarrow{\sim}
\left(
\varinjlim_{E / F}
\frac{\tilde{\pi}_1(G, Z)^{N_{E / F}}}{I_\Gamma \pi_1(G)}
\right)
\times
\operatorname{Hom}_F(D^{\mathrm{iso}}, Z).
\end{equation}
For $(x, \eta) \in Z^1(\mathcal{E}^{\mathrm{ext}}, Z \to G)$, write $\eta(d, t) = \nu(d) \mu(t)$.
The two maps are given by
\begin{align*}
[x, \eta] &\longmapsto \bigl((f^{\mathrm{ext}})^*[x, \eta], \nu\bigr), \\
\operatorname{ab}[x, \eta] &\longmapsto \bigl(\kappa_{\mathrm{rig}}(\operatorname{ab}((f^{\mathrm{ext}})^*[x, \eta])), \nu\bigr).
\end{align*}
Thus abelianization corresponds to $\kappa_{\mathrm{rig}} \circ \operatorname{ab}$ on the first factor and the identity on the second.
A homomorphism $\eta(d, t) = \nu(d) \mu(t)$ lies in the image of $\mathcal{N}$ if and only if $(\nu \circ \phi) \mu$ belongs to $\operatorname{im}(\mathcal{N}_{\mathrm{ab}}^{\mathrm{rig}})$.
\end{proposition}

\begin{proof}
Apply Lemma~\ref{lem:cartesian} and Theorem~\ref{thm:main-bis} to $f^{\mathrm{ext}}$.
In \eqref{eq:extended-kottwitz-product}, let $\chi$ be the restriction of the first coordinate to $u$.
In \eqref{eq:extended-kottwitz-formula}, the first coordinate $r$ gives $\chi = \mathcal{N}_{\mathrm{ab}}^{\mathrm{rig}}(\kappa_{\mathrm{rig}}^{-1}(r))$.
Since the band map of $f^{\mathrm{ext}}$ is $t \mapsto (\phi(t), t)$, the fiber product condition is
\[
\chi = (\nu \circ \phi) \mu.
\]
For a given first coordinate and $\nu$, this determines $\mu = \chi (\nu \circ \phi)^{-1}$.
Forgetting $\mu$ gives the bijections in \eqref{eq:extended-kottwitz-product} and \eqref{eq:extended-kottwitz-formula}, together with the stated criterion for $\eta$.
In the abelianized case, the Newton map is a homomorphism by Proposition~\ref{prop:general-newton-fibers}, so both forgetting $\mu$ and its inverse preserve the group structure.
\end{proof}

For non-Archimedean $F$ and $Z = Z(G)$, \eqref{eq:extended-kottwitz-padic} below recovers Fargues's parametrization of $B^e(G)_{\mathrm{bas}}$ in \cite[Theorem~9.12]{Far22}, expressed in terms of the decomposition of $\pi_1(G)^e_\Gamma$ in \cite[Section~10.2]{Far22}.
Lemma~\ref{lem:fargues-basic-formula} identifies Fargues's extended Kottwitz map in these coordinates.
For a real torus $T$, taking $Z = G = T$ in \eqref{eq:extended-kottwitz-real} recovers Fargues's description of $B^e(T)$ in \cite[Proposition~13.10]{Far22}.

\begin{theorem}\label{thm:extended-kottwitz-local-cases}
\begin{enumerate}
\item
If $F$ is non-Archimedean, abelianization is bijective, and \eqref{eq:extended-kottwitz-formula} gives an isomorphism of abelian groups
\begin{equation}\label{eq:extended-kottwitz-padic}
H^1(\mathcal{E}^{\mathrm{ext}}, Z \to G)
\xrightarrow{\sim}
\left(
\frac{\tilde{\pi}_1(G, Z)}{I_\Gamma \pi_1(G)}
\right)_{\mathrm{tors}}
\times
\operatorname{Hom}_\Gamma(X^*(Z), \mathbb{Q}),
\end{equation}
where the second coordinate is obtained by sending $\nu$ to $\nu^*$.

\item
For $F = \mathbb{R}$, let $c$ denote complex conjugation.
There is an isomorphism of abelian groups
\begin{equation}\label{eq:extended-kottwitz-real}
H^1_{\mathrm{ab}}(\mathcal{E}^{\mathrm{ext}}, Z \to G)
\xrightarrow{\sim}
\left\{
[y] \in
\frac{\tilde{\pi}_1(G, Z)}{(c - 1) \pi_1(G)}
\; \middle| \;
(1 + c) y \in
i\bigl(\operatorname{Hom}_\Gamma(X^*(Z), \mathbb{Z})\bigr)
\right\},
\end{equation}
where $i$ is the injection in \eqref{eq:def-i} and we regard $\pi_1(G)$ as a subgroup of $\tilde{\pi}_1(G, Z)$.

\item
For $F = \mathbb{C}$, the Newton map induces an isomorphism of abelian groups
\[
H^1(\mathcal{E}^{\mathrm{ext}}, Z \to G)
\xrightarrow{\sim}
\operatorname{Hom}_{\mathbb{Z}}(X^*(Z), \mathbb{Z}),
\qquad
[x, \eta] \longmapsto \nu^*,
\]
where $\nu$ is the restriction of $\eta$ to $D^{\mathrm{iso}}$.
\end{enumerate}
All three isomorphisms are natural in $(Z \to G)$.
For non-Archimedean or complex $F$, the source has the group structure induced by abelianization.
\end{theorem}

\begin{proof}
For non-Archimedean $F$, abelianization is bijective by Corollary~\ref{cor:abelianization-bijective}.
Substituting \eqref{eq:local-rigid-targets} and $X^*(D^{\mathrm{iso}}) = \mathbb{Q}$ into \eqref{eq:extended-kottwitz-formula} yields \eqref{eq:extended-kottwitz-padic}.

Suppose that $F = \mathbb{R}$.
By \eqref{eq:local-rigid-targets} and Proposition~\ref{prop:extended-kottwitz-formula}, we have
\[
H^1_{\mathrm{ab}}(\mathcal{E}^{\mathrm{ext}}, Z \to G)
\simeq
\frac{
\{v \in \tilde{\pi}_1(G, Z) \mid (1 + c) v = 0\}
}{
(c - 1) \pi_1(G)
}
\times
\operatorname{Hom}_{\mathbb{R}}(D^{\mathrm{iso}}, Z).
\]
We show that the map $([v], \nu) \longmapsto [v + \iota(\nu^*)]$ identifies the group on the right with the target of \eqref{eq:extended-kottwitz-real}.
Since $\nu^* \in \operatorname{Hom}_\Gamma(X^*(Z), \tfrac{1}{2} \mathbb{Z})$ and the map $\iota$ in \eqref{eq:kottwitz-central-cocharacters} is $\Gamma$-equivariant and extends $i$, we have $(1 + c)\bigl(v + \iota(\nu^*)\bigr) = i(2 \nu^*)$.
Thus the stated map takes values in the target of \eqref{eq:extended-kottwitz-real}.

Conversely, let $[y]$ belong to that target.
By the defining condition on $[y]$ and the injectivity of $i$, there is a unique $\nu \in \operatorname{Hom}_{\mathbb{R}}(D^{\mathrm{iso}}, Z)$ satisfying $(1 + c) y = i(2 \nu^*)$.
Then $(1 + c)(y - \iota(\nu^*)) = 0$, hence the inverse is
\[
[y] \longmapsto \bigl([y - \iota(\nu^*)], \nu\bigr).
\]
Changing $y$ by an element of $(c - 1) \pi_1(G)$ leaves $\nu$ unchanged and changes $y - \iota(\nu^*)$ by the same element, since $(1 + c)(c - 1) = 0$.
Thus the homomorphism $([v], \nu) \mapsto [v + \iota(\nu^*)]$ has the stated inverse.

For $F = \mathbb{C}$, one has $\Gamma = 1$, $u = 1$, and $D^{\mathrm{iso}} = \mathbb{G}_m$.
The first factor of \eqref{eq:extended-kottwitz-formula} is trivial by \eqref{eq:local-rigid-targets}, and abelianization is bijective by Corollary~\ref{cor:abelianization-bijective}.
\end{proof}

We now take $Z = Z(G)$.
Following \cite[Definition~9.1]{Far22} for non-Archimedean $F$ and \cite[Definition~13.8]{Far22} for $F = \mathbb{R}$, we write
\[
B^e(G) \coloneqq \bigl\{ [x, \eta] \in H^1(\mathcal{E}^{\mathrm{ext}}, G) \mid \eta|_u \text{ factors through } Z(G) \bigr\}
\]
for Fargues's extended Kottwitz set, where $\eta$ is the restriction of $x$ to the band $D^{\mathrm{iso}} \times u$.
We use the same definition for $F = \mathbb{C}$.
Pullback along the projection $\mathcal{E}^{\mathrm{ext}} \to \mathcal{E}^{\mathrm{iso}}$ identifies $B(G)$ with the subset of $B^e(G)$ consisting of classes for which $\eta|_u$ is trivial.
A class $[x, \eta] \in B^e(G)$ is basic if and only if $\eta|_{D^{\mathrm{iso}}}$ also factors through $Z(G)$, as in \cite[Definition~9.4]{Far22}.
Thus
\begin{equation}\label{eq:extended-basic-cohomology}
B^e(G)_{\mathrm{bas}} = H^1(\mathcal{E}^{\mathrm{ext}}, Z(G) \to G).
\end{equation}

For a maximal $F$-torus $T \subseteq G$, let $\Phi$ and $\Phi^\vee$ denote the corresponding roots and coroots.
Following \cite[Definition~9.10]{Far22}, let
\begin{equation}\label{eq:fargues-extended-fundamental-group}
\pi_1(G)^e_\Gamma
\coloneqq
\frac{
\{\lambda \in X_*(T) \otimes_{\mathbb{Z}} \mathbb{Q}
\mid
\langle \alpha, \lambda \rangle \in \mathbb{Z}
\text{ for every } \alpha \in \Phi\}
}{
\mathbb{Z} \Phi^\vee + I_\Gamma X_*(T)
}.
\end{equation}

\begin{lemma}\label{lem:extended-fundamental-identification}
There is a canonical isomorphism
\[
\pi_1(G)^e_\Gamma
\simeq
\frac{\tilde{\pi}_1(G, Z(G))}{I_\Gamma \pi_1(G)}.
\]
\end{lemma}

\begin{proof}
For $n \geq 1$, set $A_n \coloneqq Z(G)[n]$.
Since $X^*(Z(G)) = X^*(T) / \mathbb{Z} \Phi$, pullback along $T \to T / A_n$ identifies
\[
X^*(T / A_n)
=
\mathbb{Z} \Phi + n X^*(T)
\subseteq X^*(T).
\]
Taking dual lattices in $X_*(T) \otimes_{\mathbb{Z}} \mathbb{Q}$ yields
\[
X_*(T / A_n)
=
\left\{
\lambda \in \tfrac{1}{n} X_*(T)
\; \middle| \;
\langle \alpha, \lambda \rangle \in \mathbb{Z}
\text{ for every } \alpha \in \Phi
\right\}.
\]
Under these embeddings, the coroot lattices of the groups $G / A_n$ are all identified with $\mathbb{Z} \Phi^\vee$.
Since the subgroups $A_n$ are cofinal among the finite $F$-subgroups of $Z(G)$, taking the direct limit gives
\[
\tilde{\pi}_1(G, Z(G))
=
\frac{
\{\lambda \in X_*(T) \otimes_{\mathbb{Z}} \mathbb{Q}
\mid
\langle \alpha, \lambda \rangle \in \mathbb{Z}
\text{ for every } \alpha \in \Phi\}
}{
\mathbb{Z} \Phi^\vee
}.
\]
The image of $I_\Gamma X_*(T)$ in this quotient is $I_\Gamma \pi_1(G)$.
Taking the quotient by this subgroup gives the asserted isomorphism.
\end{proof}

Suppose that $F$ is non-Archimedean.
Let
\[
\kappa_G^e \colon B^e(G) \longrightarrow \pi_1(G)^e_\Gamma
\]
denote Fargues's extended Kottwitz map of \cite[Section~9.5]{Far22}, where $\pi_1(G)^e_\Gamma$ is the group in \eqref{eq:fargues-extended-fundamental-group}.
Its restriction to $B(G)$ is the composite of the ordinary Kottwitz map $\kappa_G$ with the natural inclusion $\pi_1(G)_\Gamma \hookrightarrow \pi_1(G)^e_\Gamma$.

\begin{lemma}\label{lem:fargues-basic-formula}
Suppose that $F$ is non-Archimedean.
Equip $B^e(G)_{\mathrm{bas}}$ with the canonical group structure induced by abelianization via \eqref{eq:extended-basic-cohomology} and Proposition~\ref{prop:gerbe-abelian-group}.
Under the identification of Lemma~\ref{lem:extended-fundamental-identification}, Fargues's map restricts to an isomorphism of abelian groups given in the coordinates of \eqref{eq:extended-kottwitz-padic} by
\[
\kappa_G^e \colon B^e(G)_{\mathrm{bas}} \xrightarrow{\sim} \frac{\tilde{\pi}_1(G, Z(G))}{I_\Gamma \pi_1(G)},
\qquad
(r, \lambda) \longmapsto r + [\iota(\lambda)],
\]
where $\iota$ is the injection in \eqref{eq:kottwitz-central-cocharacters} and $[\iota(\lambda)]$ denotes the class of $\iota(\lambda)$ modulo $I_\Gamma \pi_1(G)$.
\end{lemma}

\begin{proof}
Let $\Psi_G$ send a class with coordinates $(r, \lambda)$ to $r + [\iota(\lambda)]$.
This is a homomorphism by Theorem~\ref{thm:extended-kottwitz-local-cases} and the additivity of $\iota$.
By \cite[Theorem~9.12]{Far22}, it suffices to prove that $\Psi_G = \kappa_G^e$ on $B^e(G)_{\mathrm{bas}}$.

Suppose first that $G = T$ is a torus, and set $Y \coloneqq X_*(T)$ and $Y_{\mathbb{Q}} \coloneqq Y \otimes_{\mathbb{Z}} \mathbb{Q}$.
Since $T^{\mathrm{sc}} = 1$, the group structure on $B^e(T)$ defined in Proposition~\ref{prop:gerbe-abelian-group} is induced by pointwise multiplication of cocycles.
Let $b = [x, \eta] \in B^e(T)$ have coordinates $([y], \lambda)$.
By Theorem~\ref{thm:local-rigid-formula} and \eqref{eq:kottwitz-local-band-map}, the image of $\Psi_T(b)$ in $Y_{\mathbb{Q}} / Y$ is
\[
(y + \lambda) + Y = \operatorname{ev}_1(\eta|_u).
\]
Under \eqref{eq:u-nonfinite}, the exact sequences of \cite[Proposition~9.7]{Far22} identify the restriction $B^e(T) \to \operatorname{Hom}_F(u, T)$ with the quotient $Y_{\mathbb{Q}} / I_\Gamma Y \to Y_{\mathbb{Q}} / Y$.
Thus $\kappa_T^e(b)$ has the same image as $\Psi_T(b)$ in $Y_{\mathbb{Q}} / Y$.
The difference $\Psi_T - \kappa_T^e$ is therefore a homomorphism from $B^e(T)$ to $Y / I_\Gamma Y = Y_\Gamma$.
The same proposition identifies $B^e(T)$ with the divisible group $Y_{\mathbb{Q}} / I_\Gamma Y$.
Since the finitely generated abelian group $Y_\Gamma$ has no nonzero divisible subgroup, this homomorphism is zero.
Hence $\Psi_T = \kappa_T^e$.

For general $G$, let $H$ be its quasi-split inner form.
By \cite[Corollary~9.9]{Far22}, we may choose $a \in B^e(H)_{\mathrm{bas}}$ such that $G \simeq H_a$.
Fix an elliptic maximal $F$-torus $T \subseteq H$ and write $j \colon T \hookrightarrow H$ for the inclusion.
By \cite[Lemma~9.11]{Far22}, choose $a_T \in B^e(T)$ whose image in $H^1(\mathcal{E}^{\mathrm{ext}}, H)$ is $a$, and let $s \colon \mathcal{E}^{\mathrm{ext}} \to T(\bar{F})$ represent $a_T$.
Realize $G \simeq H_a$ as the twist of $H$ by the adjoint cocycle $\overline{j \circ s}$, and let $t \colon B^e(G) \xrightarrow{\sim} B^e(H)$ be the corresponding twisting bijection, so that $t(1) = a$.
Since each $j(s_w)$ centralizes $T$, this twist contains $T$ with its original $F$-structure.
For $b \in B^e(G)_{\mathrm{bas}}$, the restriction of $t(b)$ to $D^{\mathrm{iso}}$ is the product of the restrictions of $b$ and $a$ under the identification $Z(G) \simeq Z(H)$.
Thus $t(b)$ is basic, and \cite[Lemma~9.11]{Far22} gives $c_T \in B^e(T)$ whose image in $H^1(\mathcal{E}^{\mathrm{ext}}, H)$ is $t(b)$.
If $z \colon \mathcal{E}^{\mathrm{ext}} \to T(\bar{F})$ represents $c_T$, then $w \mapsto z_w s_w^{-1}$ represents $c_T a_T^{-1}$ and has image $b$ in $H^1(\mathcal{E}^{\mathrm{ext}}, G)$.

Choose $v_a, v_c \in X_*(T) \otimes_{\mathbb{Z}} \mathbb{Q}$ representing $\kappa_T^e(a_T)$ and $\kappa_T^e(c_T)$ respectively.
The restrictions of $a_T$ and $c_T$ to $u$ take values in $Z(H)$, on which every root $\alpha \in \Phi(H, T)$ is trivial.
The formula for reduction modulo $X_*(T)$ in the torus case therefore yields $\langle \alpha, v_a \rangle, \langle \alpha, v_c \rangle \in \mathbb{Z}$.
Writing $\Phi^\vee$ for the corresponding coroot system, Fargues's definition in \cite[Sections~9.5.2--9.5.3]{Far22} gives
\[
\kappa_G^e(b) = \kappa_H^e(t(b)) - \kappa_H^e(a)
= (v_c - v_a) + \bigl( \mathbb{Z} \Phi^\vee + I_\Gamma X_*(T) \bigr).
\]
By the torus case, $v_c - v_a$ represents $\Psi_T(c_T a_T^{-1})$.
Since $a$ and $t(b)$ are basic, the restrictions of $c_T a_T^{-1}$ to $D^{\mathrm{iso}}$ and to $u$ both factor through $Z(G)$.
By naturality of Proposition~\ref{prop:extended-kottwitz-formula} and of $\iota$ for the maps from $(Z(G) \to T)$ to $(T \to T)$ and $(Z(G) \to G)$, we obtain
\[
\Psi_G(b) = (v_c - v_a) + \bigl( \mathbb{Z} \Phi^\vee + I_\Gamma X_*(T) \bigr) = \kappa_G^e(b).
\]

\end{proof}

\begin{remark}
Suppose that $F$ is non-Archimedean.
In the coordinates of Lemma~\ref{lem:fargues-basic-formula}, the split exact sequence
\[
0 \longrightarrow (\pi_1(G)^e_\Gamma)_{\mathrm{tors}}
\longrightarrow \pi_1(G)^e_\Gamma
\xrightarrow{N^\diamond} \operatorname{Hom}_\Gamma(X^*(Z(G)), \mathbb{Q})
\longrightarrow 0
\]
has $N^\diamond(r + [\iota(\lambda)]) = \lambda$ and section $\lambda \mapsto [\iota(\lambda)]$.
The corresponding projection onto the torsion subgroup is given by $\operatorname{proj}(r + [\iota(\lambda)]) = r$.

Let $p \colon \mathcal{E}^{\mathrm{ext}} \to \mathcal{E}^{\mathrm{iso}}$ be the projection, and let $j \colon \pi_1(G)_\Gamma \hookrightarrow \pi_1(G)^e_\Gamma$ be the natural inclusion.
Then the following diagram commutes.
\[
\begin{tikzcd}[column sep=large]
B(G)_{\mathrm{bas}}
  \arrow[r, "p^*"] \arrow[d, "\kappa_G"']
& B^e(G)_{\mathrm{bas}}
  \arrow[r, "{(f^{\mathrm{ext}})^*}"] \arrow[d, "\kappa_G^e"']
& H^1(\mathcal{E}^{\mathrm{rig}}, Z(G) \to G)
  \arrow[d, "\kappa_{\mathrm{rig}} \circ \operatorname{ab}"] \\
\pi_1(G)_\Gamma
  \arrow[r, hook, "j"']
& \pi_1(G)^e_\Gamma
  \arrow[r, "\operatorname{proj}"']
& (\pi_1(G)^e_\Gamma)_{\mathrm{tors}}.
\end{tikzcd}
\]

The left square follows from the compatibility of $\kappa_G^e$ with $\kappa_G$.
For the right square, let $a \in B^e(G)_{\mathrm{bas}}$ have coordinates $(r, \lambda)$ in \eqref{eq:extended-kottwitz-padic}.
By Lemma~\ref{lem:fargues-basic-formula} and the definition of $r$ in Proposition~\ref{prop:extended-kottwitz-formula}, we have
\[
\operatorname{proj}\bigl(\kappa_G^e(a)\bigr)
= r
= \kappa_{\mathrm{rig}}\bigl(\operatorname{ab}((f^{\mathrm{ext}})^* a)\bigr).
\]
Since $p \circ f^{\mathrm{ext}} = f$, the upper composite is $f^*$.
Thus the outer rectangle recovers \cite[Proposition~10.6]{Far22}.
\end{remark}

\section{Formulas over number fields}\label{sec:number-field}

Throughout this section, let $F$ be a number field, let $G$ be a connected reductive $F$-group, and let $Z \subseteq G$ be a central $F$-subgroup.
Let $\mathcal{E}$ be a Galois gerbe over $F$ banded by $D$.

\subsection{Kaletha's global rigid gerbe}\label{s:Kaletha-global-gerbe}

Let $V$ be the set of places of $F$.
For each $v \in V$, choose a place $\dot v$ of $\bar{F}$ above $v$ so that $\dot V \coloneqq \{ \dot v \mid v \in V \}$ satisfies \cite[Condition~3.3.1]{KT23}.
For a finite Galois extension $E / F$ and a finite subset $S \subseteq V$, let $S_E$ be the set of places of $E$ above $S$.
For each $v \in S$, let $\dot v|_E$ denote the restriction of $\dot v$ to $E$.
Then $\dot S_E \coloneqq \{ \dot v|_E \mid v \in S \}$ is a subset of $S_E$ containing exactly one place above each $v \in S$.
Following \cite[Conditions~2.4.1]{KT23}, we call $(E, S)$ \emph{admissible} if the following four conditions hold:
\begin{enumerate}
\item The set $S$ contains every Archimedean place of $F$ and every finite place of $F$ that ramifies in $E$.
\item The classes of the finite primes in $S_E$ generate the ideal class group of $E$.
\item For every $w \in V_E$, there exists $w' \in S_E$ with the same stabilizer as $w$ in $\Gamma_{E / F}$.
\item Every $\sigma \in \Gamma_{E / F}$ fixes some $w \in \dot S_E$.
\end{enumerate}
Fix an increasing sequence $(E_i)_{i \geq 0}$ of finite Galois extensions of $F$ with $\bigcup_{i \geq 0} E_i = \bar{F}$.
By \cite[Lemma~3.3.2(2)]{KT23}, we may choose an increasing sequence $(S_i)_{i \geq 0}$ of finite subsets of $V$ with $\bigcup_{i \geq 0} S_i = V$ such that $(E_i, S_i)$ is admissible for every $i$.
Enlarging each $S_i$ if necessary, we assume that $S_i$ contains all places of $F$ above a rational prime $p$ whenever it contains one of them.
Thus each $S_i$ is full in the sense of \cite[Section~2]{Kal18a}.

Let $P_{\dot V}$ be Kaletha's profinite $F$-group of multiplicative type, denoted by $P^{\mathrm{rig}}_{\dot V}$ in \cite[Section~3.3]{KT23}.
Its character group $X^*(P_{\dot V})$ is torsion.
Let $\xi^{\mathrm{rig}}_{\dot V} \in H^2(F, P_{\dot V})$ be the distinguished class of \cite[Definition~3.5.4]{Kal18a}, and choose a Galois gerbe
\[
1 \longrightarrow P_{\dot V}(\bar{F}) \longrightarrow \mathcal{E}^{\mathrm{rig}}_{\dot V} \longrightarrow \Gamma \longrightarrow 1
\]
whose isomorphism class corresponds to $\xi^{\mathrm{rig}}_{\dot V}$.
For fixed $\dot V$, the cohomology sets associated with Galois gerbes in the same isomorphism class are canonically identified by Remark~\ref{rmk:Kaletha-gerbe_surj}(1).

\begin{lemma}\label{lem:rig-nonfinite-global}
Let $Z$ be an $F$-group of multiplicative type and of finite type.
Every $F$-homomorphism $P_{\dot V} \to Z$ factors through $Z[n]$ for some positive integer $n$.
There is a natural isomorphism
\begin{equation}\label{eq:P-nonfinite}
\operatorname{Hom}_F(P_{\dot V}, Z) \simeq \left\{ (\lambda_{\dot v}) \in \bigoplus_{\dot v \in \dot V} L_{\dot v}(Z) \; \middle| \; \sum_{\dot v \in \dot V} \lambda_{\dot v} = 0 \right\},
\end{equation}
where
\[
L_{\dot v}(Z) \coloneqq
\begin{cases}
\operatorname{Hom}_{\mathbb{Z}}(X^*(Z), \mathbb{Q} / \mathbb{Z}) & \text{if $v$ is finite}, \\[3pt]
\{ \lambda \in \operatorname{Hom}_{\mathbb{Z}}(X^*(Z), \mathbb{Q} / \mathbb{Z}) \mid \lambda + \lambda \circ c_{\dot v} = 0 \} & \text{if $v$ is real}, \\[3pt]
0 & \text{if $v$ is complex}.
\end{cases}
\]
Here $c_{\dot v}$ denotes the nontrivial element of the decomposition subgroup $\Gamma_{\dot v}$ at a real place.
\end{lemma}

\begin{proof}
For an $F$-homomorphism $\nu \colon P_{\dot V} \to Z$, the image of $\nu^* \colon X^*(Z) \to X^*(P_{\dot V})$ is finite, since $X^*(Z)$ is finitely generated and $X^*(P_{\dot V})$ is torsion.
Choose a positive integer $n$ annihilating this image.
Then $\nu^*$ factors through $X^*(Z) / n X^*(Z) = X^*(Z[n])$, hence $\nu$ factors through $Z[n]$.

For each $n$, \cite[Lemma~3.3.6]{Kal18a} gives \eqref{eq:P-nonfinite} with $Z[n]$ in place of $Z$.
Precomposition with $X^*(Z) \to X^*(Z[n])$ identifies $L_{\dot v}(Z[n])$ with the subgroup of $L_{\dot v}(Z)$ annihilated by $n$.
Each family on the right of \eqref{eq:P-nonfinite} has finite support, and each of its coordinates has finite image since $X^*(Z)$ is finitely generated.
Thus a single positive integer annihilates all its coordinates.
Taking the direct limit over $n$, ordered by divisibility, yields \eqref{eq:P-nonfinite}.
\end{proof}

By Lemmas~\ref{lem:rig-nonfinite-global} and~\ref{lem:abelianization-colimit}, the inclusions of finite $F$-subgroups $A \subseteq Z$ induce bijections
\begin{equation}\label{eq:global-rigid-colimit}
\varinjlim_{A \subseteq Z} H^1(\mathcal{E}^{\mathrm{rig}}_{\dot V}, A \to G) \xrightarrow{\sim} H^1(\mathcal{E}^{\mathrm{rig}}_{\dot V}, Z \to G), \qquad
\varinjlim_{A \subseteq Z} H^1_{\mathrm{ab}}(\mathcal{E}^{\mathrm{rig}}_{\dot V}, A \to G) 
\xrightarrow{\sim} H^1_{\mathrm{ab}}(\mathcal{E}^{\mathrm{rig}}_{\dot V}, Z \to G).
\end{equation}
The transition maps in both direct systems are injective.

For each $v \in V$, let $\mathcal{E}^{\mathrm{rig}}_v$ be Kaletha's local rigid gerbe over $F_v$ with band $u_v$, as in Section~\ref{s:Kaletha-local-gerbe}.
The lift $\dot v$ determines a homomorphism $u_v \to (P_{\dot V})_{F_v}$ compatible with the local and global distinguished classes.
Choose a morphism $\ell_v \colon \mathcal{E}^{\mathrm{rig}}_v \to (\mathcal{E}^{\mathrm{rig}}_{\dot V})_v$ inducing this homomorphism on the bands.
By \cite[Section~3.6]{Kal18a}, localization to $(\mathcal{E}^{\mathrm{rig}}_{\dot V})_v$ followed by pullback along $\ell_v$ gives a map independent of the choice of $\ell_v$,
\[
\operatorname{loc}_v \colon H^1(\mathcal{E}^{\mathrm{rig}}_{\dot V}, Z \to G) \longrightarrow H^1(\mathcal{E}^{\mathrm{rig}}_v, Z_{F_v} \to G_{F_v}).
\]
Under \eqref{eq:P-nonfinite} and \eqref{eq:local-rigid-band}, restriction along $u_v \to (P_{\dot V})_{F_v}$ sends $(\lambda_{\dot w})_{\dot w \in \dot V}$ to $\lambda_{\dot v}$.

We now extend Kaletha's parametrization of $H^1_{\mathrm{ab}}(\mathcal{E}^{\mathrm{rig}}_{\dot V}, A \to G)$ from finite central $F$-subgroups $A$ to arbitrary central $F$-subgroups $Z$.
Fix a maximal $F$-torus $T \subseteq G$.
If $E$ splits $T$, then $E$ also splits $T / A$ for every finite $F$-subgroup $A \subseteq Z$.
Thus the $\Gamma$-action on each $\pi_1(G / A)$, and hence on their direct limit $\tilde{\pi}_1(G, Z)$, factors through $\Gamma_{E / F}$.
We may therefore omit finitely many terms of the chosen tower and assume that this action factors through $\Gamma_{E / F}$ for every remaining pair $(E, S)$.

Fix such a pair $(E, S)$.
For $M = \pi_1(G)$ or $M = \tilde{\pi}_1(G, Z)$, set
\[
M[S_E]_0 \coloneqq
\left\{
(m_w)_{w \in S_E} \in \bigoplus_{w \in S_E} M
\;\middle|\;
\sum_{w \in S_E} m_w = 0
\right\}.
\]
The diagonal action of $\Gamma_{E / F}$ on $M[S_E]_0$ is given by
$(\sigma m)_w = \sigma(m_{\sigma^{-1} w})$ for $\sigma \in \Gamma_{E / F}$ and $w \in S_E$.

For $M = \tilde{\pi}_1(G, Z)$, set
\[
M[S_E, \dot S_E]_0 \coloneqq
\left\{
m \in M[S_E]_0
\;\middle|\;
m_w \in \pi_1(G) \text{ for every } w \in S_E \setminus \dot S_E
\right\}.
\]
The superscript $N_{E / F}$ denotes the intersection with the kernel of the norm on the ambient module,
\[
M[S_E, \dot S_E]_0^{N_{E / F}} \coloneqq
M[S_E, \dot S_E]_0 \cap \ker\bigl( N_{E / F} \colon M[S_E]_0 \to M[S_E]_0 \bigr).
\]

\begin{theorem}[{\cite[Theorem~3.8.1]{Kal18a} for finite $Z$}]\label{thm:global-rigid-formula}
There is an isomorphism of abelian groups, natural in $(Z \to G)$,
\begin{equation}\label{eq:global-rigid-formula}
\kappa_{\mathrm{rig}, \dot V} \colon H^1_{\mathrm{ab}}(\mathcal{E}^{\mathrm{rig}}_{\dot V}, Z \to G) \xrightarrow{\sim} \varinjlim_{(E, S)} \frac{\tilde{\pi}_1(G, Z)[S_E, \dot S_E]_0^{N_{E / F}}}{I_\Gamma \bigl(\pi_1(G)[S_E]_0\bigr)}.
\end{equation}
The limit runs over the admissible pairs specified above, with the transition maps of \cite[Lemma~3.7.1 and Section~3.8]{Kal18a}.
The diagram
\[
\begin{tikzcd}[column sep=large, row sep=small]
H^1_{\mathrm{ab}}(\mathcal{E}^{\mathrm{rig}}_{\dot V}, Z \to G) \arrow[r, "{\kappa_{\mathrm{rig}, \dot V}}", "\sim"'] \arrow[d, "\mathcal{N}_{\mathrm{ab}}^{\mathrm{rig}}"'] & \displaystyle \varinjlim_{(E, S)} \frac{\tilde{\pi}_1(G, Z)[S_E, \dot S_E]_0^{N_{E / F}}}{I_\Gamma \bigl(\pi_1(G)[S_E]_0\bigr)} \arrow[d] \\
\operatorname{Hom}_F(P_{\dot V}, Z) \arrow[r, hook] & \displaystyle \bigoplus_{\dot v \in \dot V} L_{\dot v}(Z)
\end{tikzcd}
\]
commutes.
The bottom arrow is the inclusion given by \eqref{eq:P-nonfinite}.
Let $(E, S)$ be an admissible pair, and let $m = (m_w)_{w \in S_E} \in \tilde{\pi}_1(G, Z)[S_E, \dot S_E]_0^{N_{E / F}}$.
Write $[m]$ for the class of $m$ in the direct limit in \eqref{eq:global-rigid-formula}.
The $\dot v$-coordinate of the image of $[m]$ under the right vertical map is zero if $v \notin S$.
If $v \in S$, this coordinate is the image of $m_{\dot v|_E}$ under the composite
\[
\tilde{\pi}_1(G, Z) \longrightarrow \tilde{\pi}_1(G, Z) / \pi_1(G)
\xrightarrow{\sim} \operatorname{Hom}_{\mathbb{Z}}(X^*(Z), \mathbb{Q} / \mathbb{Z}),
\]
where the last isomorphism is that of Proposition~\ref{prop:colimit-quotient}.
\end{theorem}

\begin{proof}
Fix a finite $F$-subgroup $A \subseteq Z$ and a maximal $F$-torus $T \subseteq G$.
After omitting finitely many terms of the tower, we may assume that $E$ splits $T$ and that $S$ contains every finite place of $F$ whose residue characteristic divides the exponent of $X^*(A)$, as in \cite[Section~3.7]{Kal18a}.
For finite $A$, the equivalence relation \eqref{eq:abelianization-relation} defining $H^1_{\mathrm{ab}}(\mathcal{E}^{\mathrm{rig}}_{\dot V}, A \to G)$ agrees with the relation in \cite[Section~3.8]{Kal18a}.
Since $\pi_1(G / A)$ is the quotient of $X_*(T / A)$ by the image of the coroot lattice of $(G, T)$, the inverse of the bijection in \cite[Theorem~3.8.1]{Kal18a}, together with \cite[(3.30)]{Kal18a}, yields
\[
H^1_{\mathrm{ab}}(\mathcal{E}^{\mathrm{rig}}_{\dot V}, A \to G)
\xrightarrow{\sim}
\varinjlim_{(E, S)}
\frac{\pi_1(G / A)[S_E, \dot S_E]_0^{N_{E / F}}}{I_\Gamma \bigl(\pi_1(G)[S_E]_0\bigr)}.
\]

For a fixed finite $A$, write $\iota_{A, G}$ for the inverse of this bijection.
Given two elements $m_1, m_2$ of the displayed direct limit, point~(1) in the proof of \cite[Theorem~3.8.1]{Kal18a}, following \cite[(3.32)]{Kal18a}, gives a maximal $F$-torus $j \colon T \hookrightarrow G$ and lifts $m_{1, T}, m_{2, T}$ in the corresponding direct limit for $T$.
By \cite[Theorem~3.7.3]{Kal18a}, $\iota_{A, T}$ is an isomorphism of abelian groups for pointwise multiplication of cocycles.
By Lemma~\ref{lem:gerbe-abelian-functoriality}, $j_*$ preserves the group structure defined in Proposition~\ref{prop:gerbe-abelian-group}.
Naturality for $j$ now yields
\[
\iota_{A, G}(m_1 + m_2) = j_*\bigl(\iota_{A, T}(m_{1, T}) * \iota_{A, T}(m_{2, T})\bigr)
= \iota_{A, G}(m_1) * \iota_{A, G}(m_2).
\]
Naturality in \cite[Theorem~3.8.1]{Kal18a} also gives compatibility with the inclusions $(A \to G) \to (A' \to G)$ for finite $F$-subgroups $A \subseteq A' \subseteq Z$.

We now take the direct limit over the finite $F$-subgroups $A \subseteq Z$.
Fix an admissible pair $(E, S)$ and let $m = (m_w)_{w \in S_E} \in \tilde{\pi}_1(G, Z)[S_E, \dot S_E]_0^{N_{E / F}}$.
By \eqref{eq:tilde-pi} and Lemma~\ref{lem:finite-quotient}, $\tilde{\pi}_1(G, Z)$ is the directed union of the groups $\pi_1(G / A)$.
Since $S_E$ is finite, we may choose $A$ such that $m_w \in \pi_1(G / A)$ for every $w \in S_E$.
The inclusion $\pi_1(G / A) \hookrightarrow \tilde{\pi}_1(G, Z)$ is $\Gamma$-equivariant, so the relations $\sum_{w \in S_E} m_w = 0$ and $N_{E / F} m = 0$ already hold with coefficients in $\pi_1(G / A)$.
Moreover, $m_w \in \pi_1(G)$ for $w \notin \dot S_E$.
Thus $m \in \pi_1(G / A)[S_E, \dot S_E]_0^{N_{E / F}}$.

For each $(E, S)$, the denominator $I_\Gamma \bigl(\pi_1(G)[S_E]_0\bigr)$ is independent of $A$.
The transition maps in $(E, S)$ are compatible with the inclusions in $A$, and for each $A$ the splitting and residue characteristic conditions hold for all sufficiently large terms of the tower.
We may therefore interchange the two direct limits.
The preceding description of the numerator and \eqref{eq:global-rigid-colimit} now give \eqref{eq:global-rigid-formula}.
The bijection is natural in $(Z \to G)$ by the naturality of the bijections for finite $A$.

For finite $A$, the compatibility with restriction to $P_{\dot V}$ follows from \cite[Theorems~3.7.3 and~3.8.1]{Kal18a}.
If $w \notin \dot S_E$, then $m_w \in \pi_1(G)$, so its image in $\pi_1(G / A) / \pi_1(G)$ is zero.
For $v \in S$, the $\dot v$-coordinate is the image of $m_{\dot v|_E}$ in this quotient, and for $v \notin S$ it is zero.
These are the coordinates of the family in \eqref{eq:P-nonfinite} for $A$.
Passing to the direct limit over $A$ and applying Proposition~\ref{prop:colimit-quotient} gives the stated right vertical map and proves that the diagram commutes.
\end{proof}

\begin{remark}
If $F$ is totally imaginary, abelianization is bijective by Corollary~\ref{cor:abelianization-bijective}.
With the group structure transported by abelianization, $\kappa_{\mathrm{rig}, \dot V} \circ \operatorname{ab}$ identifies $H^1(\mathcal{E}^{\mathrm{rig}}_{\dot V}, Z \to G)$ with the direct limit in \eqref{eq:global-rigid-formula} as an abelian group.
\end{remark}

\subsection{Comparison with the global rigid gerbe}

We apply the comparison of gerbes in Proposition~\ref{prop:comparison-formal-ab} to a morphism from Kaletha's global rigid gerbe to $\mathcal{E}$.

\begin{theorem}\label{thm:global-abelianized-formula}
Let $f \colon \mathcal{E}^{\mathrm{rig}}_{\dot V} \to \mathcal{E}$ be a morphism of gerbes inducing an $F$-homomorphism $\phi \colon P_{\dot V} \to D$ on the bands.
Then $f$ induces an isomorphism of abelian groups, natural in $(Z \to G)$,
\begin{equation}\label{eq:global-abelianized-formula}
H^1_{\mathrm{ab}}(\mathcal{E}, Z \to G) \xrightarrow{\sim} \left( \varinjlim_{(E, S)} \frac{\tilde{\pi}_1(G, Z)[S_E, \dot S_E]_0^{N_{E / F}}}{I_\Gamma \bigl(\pi_1(G)[S_E]_0\bigr)} \right) \underset{\operatorname{Hom}_F(P_{\dot V}, Z)}{\times} \operatorname{Hom}_F(D, Z),
\end{equation}
given by
$\operatorname{ab}[x, \nu] \mapsto \bigl(\kappa_{\mathrm{rig}, \dot V}(\operatorname{ab}[x \circ f, \nu \circ \phi]), \nu\bigr)$.
The limit and its transition maps are those of Theorem~\ref{thm:global-rigid-formula}.
The maps to $\operatorname{Hom}_F(P_{\dot V}, Z)$ are $\mathcal{N}_{\mathrm{ab}}^{\mathrm{rig}} \circ \kappa_{\mathrm{rig}, \dot V}^{-1}$ and $\nu \mapsto \nu \circ \phi$ respectively.

Moreover, abelianization and pullback induce a bijection of pointed sets, natural in $(Z \to G)$,
\begin{equation}\label{eq:global-full-recovery}
H^1(\mathcal{E}, Z \to G) \xrightarrow{\sim} H^1_{\mathrm{ab}}(\mathcal{E}, Z \to G) \underset{H^1_{\mathrm{ab}}(\mathcal{E}^{\mathrm{rig}}_{\dot V}, Z \to G)}{\times} H^1(\mathcal{E}^{\mathrm{rig}}_{\dot V}, Z \to G),
\end{equation}
given by $[x, \nu] \mapsto (\operatorname{ab}[x, \nu], f^*[x, \nu])$.
The maps to the base of this fiber product are $f^*$ and abelianization respectively.
If $F$ is totally imaginary, \eqref{eq:global-abelianized-formula} also holds with $H^1(\mathcal{E}, Z \to G)$ on the left.

For fixed $\phi$, both bijections are independent of the choice of $f$ if $H^1(F, D) = 0$.
If $D$ is of finite type, there exists a morphism $f$ whose band map factors through a finite $F$-subgroup of $D$.
\end{theorem}

\begin{proof}
By \eqref{eq:comparison-formal-ab}, Lemma~\ref{lem:gerbe-abelian-functoriality}, and the additivity of $\mathcal{N}_{\mathrm{ab}}$, the map $\operatorname{ab}[x, \nu] \mapsto (\operatorname{ab}[x \circ f, \nu \circ \phi], \nu)$ identifies $H^1_{\mathrm{ab}}(\mathcal{E}, Z \to G)$ with
\[
H^1_{\mathrm{ab}}(\mathcal{E}^{\mathrm{rig}}_{\dot V}, Z \to G) \underset{\operatorname{Hom}_F(P_{\dot V}, Z)}{\times} \operatorname{Hom}_F(D, Z).
\]
Applying $\kappa_{\mathrm{rig}, \dot V}$ to the first factor gives \eqref{eq:global-abelianized-formula}, with the stated maps to the base, by Theorem~\ref{thm:global-rigid-formula}.
Applying \eqref{eq:comparison-full-recovery} to $f$ yields \eqref{eq:global-full-recovery}.
If $F$ is totally imaginary, abelianization is bijective by Corollary~\ref{cor:abelianization-bijective}.

Let $f' \colon \mathcal{E}^{\mathrm{rig}}_{\dot V} \to \mathcal{E}$ be another morphism inducing $\phi$.
If $H^1(F, D) = 0$, then $(f')^* = f^*$ on $H^1(\mathcal{E}, Z \to G)$ by Lemma~\ref{lem:comparison-dependence}, and the same equality holds after abelianization.
The existence assertion follows from Corollary~\ref{cor:comparison-exist-finite-type}.
\end{proof}

For $\nu \in \operatorname{Hom}_F(D, Z)$, the homomorphism $\nu \circ \phi$ determines a family in \eqref{eq:P-nonfinite}.
By Theorem~\ref{thm:global-abelianized-formula}, $\nu$ lies in the image of $\mathcal{N}$ if and only if this family is the image of some
\[
m \in \tilde{\pi}_1(G, Z)[S_E, \dot S_E]_0^{N_{E / F}}
\]
under the coordinate map of Theorem~\ref{thm:global-rigid-formula} for some admissible pair $(E, S)$.
By Corollary~\ref{cor:general-global-newton-fibers}, this is equivalent to $\delta_{\mathcal{E}, \mathrm{ab}}(\nu) = 0$ in $\mathcal{F}^2(\pi_1(G) / F)$.
For each such $\nu$, the classes $[m]$ giving this family form a coset of the subgroup corresponding under \eqref{eq:global-abelianized-formula} to $\operatorname{inf}_{\mathrm{ab}} H^1_{\mathrm{ab}}(F, G)$.
In particular,
\[
\# \mathcal{N}_{\mathrm{ab}}^{-1}(\nu) = \bigl\lvert \mathcal{F}^1(\pi_1(G) / F) \bigr\rvert.
\]
For totally imaginary $F$ the same formula counts the full fiber.
This cardinality can be infinite.

For $\mathcal{E} = \mathcal{E}^{\mathrm{rig}}_{\dot V}$ and $f = \operatorname{id}$, the first projection identifies the isomorphism in \eqref{eq:global-abelianized-formula} with $\kappa_{\mathrm{rig}, \dot V}$.
For finite $Z$, this is the parametrization of \cite[Theorem~3.8.1]{Kal18a}.

\subsection{Recovery from real localizations}

We first describe the effect of localization on the fiber of ordinary abelianization over the neutral class.

\begin{lemma}\label{lem:real-localization-image}
Let $H$ be a connected reductive $F$-group.
Localization induces a bijection of pointed sets
\[
\operatorname{im}\bigl[ H^1(F, H^{\mathrm{sc}}) \to H^1(F, H) \bigr] \xrightarrow{\sim} \prod_{v \text{ real}} \operatorname{im}\bigl[ H^1(F_v, H^{\mathrm{sc}}) \to H^1(F_v, H) \bigr].
\]
\end{lemma}

\begin{proof}
By \eqref{eq:ab-exact-seq}, the global and local images in the statement are the fibers of the respective abelianization maps over $0$.
Taking the fiber over $0 \in H^1_{\mathrm{ab}}(F, H)$ in the Cartesian square of \cite[Theorem~5.5.1(2)]{BK26} identifies the global fiber with the product of the fibers at the Archimedean places.
Each factor at a complex place is a singleton.
\end{proof}

For each real place $v$, choose an embedding $\bar{F} \hookrightarrow \overline{F_v}$ inducing $\dot v$, and identify $\Gamma_v \coloneqq \operatorname{Gal}(\overline{F_v} / F_v)$ with the corresponding decomposition subgroup of $\Gamma$.
Let $\mathcal{E}_v$ be the pushout of the restriction of $\mathcal{E}$ to $\Gamma_v$ along $D(\bar{F}) \to D(\overline{F_v})$, as in the construction following \cite[Fact~3.11.1]{KT23}.
We omit the base change to $F_v$ from the notation for $G$ and $Z$.

As in \cite[Section~2.10, Example~2.3]{Kot14}, localization induces a map
\[
\operatorname{loc}_v \colon H^1(\mathcal{E}, Z \to G) \longrightarrow H^1(\mathcal{E}_v, Z \to G),
\qquad
[x, \nu] \longmapsto [x_v, \nu_{F_v}],
\]
where $x_v$ is the unique cocycle extending the restriction of $x$ to the inverse image of $\Gamma_v$ in $\mathcal{E}$ and agreeing with $\nu_{F_v}$ on $D(\overline{F_v})$.
Under the inclusion $G^{\mathrm{ad}}(\bar{F}) \hookrightarrow G^{\mathrm{ad}}(\overline{F_v})$, we have $\overline{x_v} = \bar{x}|_{\Gamma_v}$.
Thus the twist $G_{x_v}$ identifies canonically with $(G_x)_{F_v}$.
Under this identification, restricting the cocycle $a \in Z^1(F, G_x^{\mathrm{sc}})$ in \eqref{eq:abelianization-relation} to $\Gamma_v$ yields the same relation over $F_v$.
Thus $\operatorname{loc}_v$ induces a map
\[
\operatorname{loc}_v \colon H^1_{\mathrm{ab}}(\mathcal{E}, Z \to G) \longrightarrow H^1_{\mathrm{ab}}(\mathcal{E}_v, Z \to G).
\]
For both cohomology sets, write $\operatorname{loc}_\infty(\zeta) \coloneqq (\operatorname{loc}_v(\zeta))_{v \text{ real}}$.

\begin{theorem}\label{prop:global-local-cartesian}
For every Galois gerbe $\mathcal{E}$ over $F$, the square
\[
\begin{tikzcd}[column sep=large]
H^1(\mathcal{E}, Z \to G) \arrow[r, "\operatorname{ab}"] \arrow[d, "\operatorname{loc}_\infty"'] & H^1_{\mathrm{ab}}(\mathcal{E}, Z \to G) \arrow[d, "\operatorname{loc}_\infty"] \\
\displaystyle \prod_{v \text{ real}} H^1(\mathcal{E}_v, Z \to G) \arrow[r, "{\prod_v \operatorname{ab}_v}"'] & \displaystyle \prod_{v \text{ real}} H^1_{\mathrm{ab}}(\mathcal{E}_v, Z \to G)
\end{tikzcd}
\]
is Cartesian.
\end{theorem}

\begin{proof}
Fix $c \in H^1_{\mathrm{ab}}(\mathcal{E}, Z \to G)$ and choose $(x, \nu) \in Z^1(\mathcal{E}, Z \to G)$ such that $\operatorname{ab}[x, \nu] = c$.
By \eqref{eq:abelianization-fiber}, twisting by $(x, \nu)$ and by $(x_v, \nu_{F_v})$ induces bijections
\[
\operatorname{im}\bigl[ H^1(F, G_x^{\mathrm{sc}}) \to H^1(F, G_x) \bigr] \xrightarrow{\sim} \operatorname{ab}^{-1}(c), \quad
\operatorname{im}\bigl[ H^1(F_v, G_{x_v}^{\mathrm{sc}}) \to H^1(F_v, G_{x_v}) \bigr] \xrightarrow{\sim} \operatorname{ab}_v^{-1}\bigl(\operatorname{loc}_v(c)\bigr).
\]
These bijections commute with localization by the calculation above.
Applying Lemma~\ref{lem:real-localization-image} to $G_x$ therefore gives a bijection
$\operatorname{loc}_\infty \colon \operatorname{ab}^{-1}(c) \xrightarrow{\sim} \prod_{v \text{ real}} \operatorname{ab}_v^{-1}\bigl(\operatorname{loc}_v(c)\bigr)$.
Since this holds for every $c$, the square is Cartesian.
\end{proof}

Consequently, for any $\nu \in \operatorname{Hom}_F(D, Z)$ the cardinality of the full Newton fiber is
\[
\# \mathcal{N}^{-1}(\nu)
= \sum_{c \in \mathcal{N}_{\mathrm{ab}}^{-1}(\nu)}
  \prod_{v \text{ real}} \#\operatorname{ab}_v^{-1}\bigl(\operatorname{loc}_v(c)\bigr).
\]
For totally imaginary $F$ the empty product is $1$, as in Corollary~\ref{cor:general-global-newton-fibers}.

For each real place $v$, use the morphism $\ell_v \colon \mathcal{E}^{\mathrm{rig}}_v \to (\mathcal{E}^{\mathrm{rig}}_{\dot V})_v$ chosen in Section~\ref{s:Kaletha-global-gerbe}.
For a morphism of gerbes $f \colon \mathcal{E}_{\dot V}^{\mathrm{rig}} \to \mathcal{E}$, let $f_{F_v} \colon (\mathcal{E}_{\dot V}^{\mathrm{rig}})_v \to \mathcal{E}_v$ be its localization and set
\[
f_v \coloneqq f_{F_v} \circ \ell_v \colon \mathcal{E}_v^{\mathrm{rig}} \longrightarrow \mathcal{E}_v.
\]

\begin{corollary}\label{prop:global-real}
Let $f \colon \mathcal{E}_{\dot V}^{\mathrm{rig}} \to \mathcal{E}$ be a morphism of gerbes inducing an $F$-homomorphism on the bands.
For each real place $v$, set $\operatorname{loc}_{v,f}^{\mathrm{rig}} \coloneqq f_v^* \circ \operatorname{loc}_v$, and write $\operatorname{loc}_{\infty,f}^{\mathrm{rig}} \coloneqq (\operatorname{loc}_{v,f}^{\mathrm{rig}})_{v \text{ real}}$.
Then the square
\[
\begin{tikzcd}[column sep=large]
H^1(\mathcal{E}, Z \to G)
  \arrow[r, "\operatorname{ab}"] \arrow[d, "{\operatorname{loc}_{\infty,f}^{\mathrm{rig}}}"']
& H^1_{\mathrm{ab}}(\mathcal{E}, Z \to G)
  \arrow[d, "{\operatorname{loc}_{\infty,f}^{\mathrm{rig}}}"] \\
\displaystyle \prod_{v \text{ real}} H^1(\mathcal{E}_v^{\mathrm{rig}}, Z \to G)
  \arrow[r, "{\prod_v \operatorname{ab}_v}"']
& \displaystyle \prod_{v \text{ real}} H^1_{\mathrm{ab}}(\mathcal{E}_v^{\mathrm{rig}}, Z \to G)
\end{tikzcd}
\]
is Cartesian.
For a fixed band map, Lemma~\ref{lem:comparison-dependence} describes how the vertical maps depend on $f$.
If $D$ is of finite type, such a morphism $f$ exists by Corollary~\ref{cor:comparison-exist-finite-type}.
\end{corollary}

\begin{proof}
For each real place $v$, applying \eqref{eq:comparison-full-recovery} to $f_v$ yields a Cartesian square
\[
\begin{tikzcd}[column sep=large]
H^1(\mathcal{E}_v, Z \to G) \arrow[r, "\operatorname{ab}_v"] \arrow[d, "f_v^*"'] & H^1_{\mathrm{ab}}(\mathcal{E}_v, Z \to G) \arrow[d, "f_v^*"] \\
H^1(\mathcal{E}_v^{\mathrm{rig}}, Z \to G) \arrow[r, "\operatorname{ab}_v"'] & H^1_{\mathrm{ab}}(\mathcal{E}_v^{\mathrm{rig}}, Z \to G).
\end{tikzcd}
\]
Taking the product over the real places and composing with the Cartesian square in Theorem~\ref{prop:global-local-cartesian} yields the desired square.
\end{proof}

\begin{remark}\label{cor:global-totally-imaginary}
Suppose that $F$ is a totally imaginary number field, and let $f \colon \mathcal{E}_{\dot V}^{\mathrm{rig}} \to \mathcal{E}$ be a morphism of gerbes with band map $\phi \colon P_{\dot V} \to D$.
By Corollary~\ref{cor:abelianization-bijective}, abelianization is bijective, so \eqref{eq:global-abelianized-formula} also parametrizes $H^1(\mathcal{E}, Z \to G)$.
If $D$ is of finite type, such a morphism $f$ exists by Corollary~\ref{cor:comparison-exist-finite-type}.
\end{remark}

\subsection{Specialization to Kottwitz's global gerbe}\label{subsec:global-kottwitz}

We use the construction of Kottwitz's global gerbe in \cite[Sections~2.4--2.5, 3.3, and~3.8]{KT23}.
Its band is the pro-torus $D^{\mathrm{iso}}$ with character group
\[
X^*(D^{\mathrm{iso}}) = \varinjlim_{E / F} \mathbb{Z}[V_E]_0,
\]
where $E / F$ runs over the finite Galois extensions in $\bar{F}$, $V_E$ is the set of places of $E$, and
\[
\mathbb{Z}[V_E]_0 \coloneqq
\left\{ \sum_{w \in V_E} n_w [w] \in \bigoplus_{w \in V_E} \mathbb{Z} [w]
\;\middle|\; \sum_{w \in V_E} n_w = 0 \right\}.
\]
For $E \subseteq K$, the transition map is induced by
\[
p_{K / E} \colon \mathbb{Z}[V_E] \longrightarrow \mathbb{Z}[V_K],
\qquad
[w] \longmapsto \sum_{u \mid w} [K_u : E_w] \, [u].
\]
Let $\mathcal{E}^{\mathrm{iso}}$ represent the class $\xi^{\mathrm{iso}} \in H^2(F, D^{\mathrm{iso}})$ determined by the global fundamental classes.
Kaletha and Taïbi use continuous Galois cochains for the inverse limit topology on $D^{\mathrm{iso}}(\bar{F})$ and on $D^{\mathrm{mid}}_{\dot V}(\bar{F})$, as in Section~\ref{s:Alg-cohomology} \cite[Sections~3.1 and~3.8--3.9]{KT23}.
A normalized continuous cocycle representing $\xi^{\mathrm{iso}}$ gives an extension on $D^{\mathrm{iso}}(\bar{F}) \times \Gamma$ with continuous section $\sigma \mapsto (1, \sigma)$.
By \cite[Lemma~3.12.1]{KT23}, its algebraic cohomology is Kottwitz's $B(G)$.
Set $B(Z \to G) \coloneqq H^1(\mathcal{E}^{\mathrm{iso}}, Z \to G)$.
Since $D^{\mathrm{iso}}$ is a pro-torus, every $F$-homomorphism $D^{\mathrm{iso}} \to Z$ factors through $Z^\circ$, so $B(Z \to G) = B(Z^\circ \to G)$.
For finite $Z$, inflation identifies $H^1(F, G)$ with $B(Z \to G)$.
For $Z = Z(G)$, we have $B(Z(G) \to G) = B(G)_{\mathrm{bas}}$.

We apply the maps $i$ and $\iota$ of \eqref{eq:def-i} and \eqref{eq:kottwitz-central-cocharacters} coordinatewise to families indexed by places.
Since $X^*(Z)$ is finitely generated and $X^*(D^{\mathrm{iso}})$ is torsion-free, passage to character groups gives a natural identification
\begin{equation}\label{eq:kottwitz-global-band-hom}
\operatorname{Hom}_F(D^{\mathrm{iso}}, Z)
\simeq
\bigl( \operatorname{Hom}_{\mathbb{Z}}(X^*(Z), \mathbb{Z}) \otimes_{\mathbb{Z}} X^*(D^{\mathrm{iso}}) \bigr)^\Gamma.
\end{equation}

\begin{lemma}\label{lem:global-kottwitz-comparison}
Let $D^{\mathrm{iso}} \xrightarrow{c^{\mathrm{iso}}} D^{\mathrm{mid}}_{\dot V} \xleftarrow{c^{\mathrm{rig}}} P_{\dot V}$ be the band maps of \cite[Section~3.4]{KT23}.
There exists an $F$-homomorphism $s \colon D^{\mathrm{mid}}_{\dot V} \to D^{\mathrm{iso}}$ such that $s \circ c^{\mathrm{iso}} = \operatorname{id}_{D^{\mathrm{iso}}}$.
For any such $s$, there exists a morphism of gerbes $f \colon \mathcal{E}^{\mathrm{rig}}_{\dot V} \to \mathcal{E}^{\mathrm{iso}}$ inducing $\phi \coloneqq s \circ c^{\mathrm{rig}}$ on the bands.
\end{lemma}

\begin{proof}
The splitting exists by \cite[Proposition~3.4.2]{KT23}.
By \cite[Corollary~3.8.6]{KT23}, we have $c^{\mathrm{rig}}_*(\xi^{\mathrm{rig}}_{\dot V}) = c^{\mathrm{iso}}_*(\xi^{\mathrm{iso}})$.
Applying $s_*$ yields $\phi_*(\xi^{\mathrm{rig}}_{\dot V}) = \xi^{\mathrm{iso}}$, so Lemma~\ref{lem:existence-comparison} gives $f$.
\end{proof}

\begin{remark}\label{rmk:global-kottwitz-lift-independence}
The splitting $s$ need not be canonical \cite[Remark~3.4.3]{KT23}.
For fixed $\phi$, the pullback on central cohomology is independent of the choice of $f$.
To see this, let $\nu \in \operatorname{Hom}_F(D^{\mathrm{iso}}, Z)$ and choose an admissible pair $(E, S)$ such that $\nu$ factors through the torus $D_{E, S}$ with character group $\mathbb{Z}[S_E]_0$.
Admissibility ensures that every $w \in V_E$ has the same stabilizer in $\Gamma_{E / F}$ as some $w' \in S_E$, so $H^1(F, D_{E, S}) = 0$ by \cite[Lemma~3.8.1]{KT23}.
Hence $\nu \circ \alpha$ is a coboundary for every $\alpha \in Z^1(F, D^{\mathrm{iso}})$, and Lemma~\ref{lem:comparison-dependence} gives the assertion before and after abelianization.
Applying the same argument to isomorphisms inducing the identity on the band gives canonical identifications of the central cohomology sets for different representatives of $\xi^{\mathrm{iso}}$.
\end{remark}

Fix $s$ and $f$ as in Lemma~\ref{lem:global-kottwitz-comparison}.
We describe the map $\nu \mapsto \nu \circ \phi$ in coordinates.
Let $\nu \in \operatorname{Hom}_F(D^{\mathrm{iso}}, Z)$ and choose an admissible pair $(E, S)$ such that the actions on $X^*(Z)$ and $\tilde{\pi}_1(G, Z)$ factor through $\Gamma_{E / F}$ and $(\nu \circ s)^*$ is represented by a function
\[
b_\nu \colon \Gamma_{E / F} \times S_E
\longrightarrow
\operatorname{Hom}_{\mathbb{Z}}\bigl( X^*(Z), \tfrac{1}{[E : F]} \mathbb{Z} \bigr)
\]
as in \cite[Section~2.4]{KT23}.
Set $\mu_{\nu, w} \coloneqq b_\nu(1, w)$ and $\mu_\nu \coloneqq (\mu_{\nu, w})_{w \in S_E}$.
The defining conditions on $b_\nu$ give $\mu_{\nu, w} = 0$ for $w \notin \dot S_E$ and $\sum_{w \in S_E} \mu_{\nu, w} = 0$.
By equivariance, $b_\nu(\sigma, w) = \sigma \mu_{\nu, \sigma^{-1} w}$.
Since $\nu \circ s \circ c^{\mathrm{iso}} = \nu$, formula~(2.4.1) of \cite{KT23} shows that $\nu^*$ is represented by
\begin{equation}\label{eq:kottwitz-global-correction}
\nu_E \coloneqq N_{E / F} \mu_\nu,
\qquad
(\nu_E)_w = \sum_{\sigma \in \Gamma_{E / F}} \sigma \mu_{\nu, \sigma^{-1} w}
\in \operatorname{Hom}_{\mathbb{Z}}(X^*(Z), \mathbb{Z}).
\end{equation}
Since $\phi = s \circ c^{\mathrm{rig}}$, the same formula shows that $(\nu \circ \phi)^*$ is represented by $-b_\nu \bmod \mathbb{Z}$.

\begin{corollary}\label{cor:kottwitz-global-recovery}
The morphism $f$ induces an isomorphism of abelian groups, natural in $(Z \to G)$,
\begin{equation}\label{eq:kottwitz-global-ab}
H^1_{\mathrm{ab}}(\mathcal{E}^{\mathrm{iso}}, Z \to G) 
\xrightarrow{\sim}
\left( \varinjlim_{(E, S)}
\frac{\tilde{\pi}_1(G, Z)[S_E, \dot S_E]_0^{N_{E / F}}}
     {I_\Gamma \bigl( \pi_1(G)[S_E]_0 \bigr)} \right)
\underset{\operatorname{Hom}_F(P_{\dot V}, Z)}{\times}
\operatorname{Hom}_F(D^{\mathrm{iso}}, Z).
\end{equation}
This isomorphism sends $\operatorname{ab}[x, \nu]$ to $\bigl( \kappa_{\mathrm{rig}, \dot V}(\operatorname{ab}(f^*[x, \nu])), \nu \bigr)$.
Suppose that $[y]$ is represented by $y = (y_w)_{w \in S_E}$ and that $\mu_\nu$ is defined at the same admissible pair $(E, S)$.
Then $([y], \nu)$ belongs to the fiber product if and only if
\begin{equation}\label{eq:kottwitz-global-integrality}
y_w + \iota(\mu_{\nu, w}) \in \pi_1(G) \qquad \text{for every } w \in S_E.
\end{equation}
\end{corollary}

\begin{proof}
Applying Theorem~\ref{thm:global-abelianized-formula} to $f$ gives \eqref{eq:kottwitz-global-ab}.
For $v \in S$, set $w \coloneqq \dot v \vert_E$.
By Theorem~\ref{thm:global-rigid-formula}, the $\dot v$-coordinate of $\mathcal{N}_{\mathrm{ab}}^{\mathrm{rig}}(\kappa_{\mathrm{rig}, \dot V}^{-1}([y]))$ corresponds to $y_w + \pi_1(G)$ under the identification of Proposition~\ref{prop:colimit-quotient}.
Under \eqref{eq:P-nonfinite}, the $\dot v$-coordinate of $\nu \circ \phi$ is obtained by evaluating $-b_\nu \bmod \mathbb{Z}$ at $(1, w)$, which gives
\[
-b_\nu(1, w) \bmod \mathbb{Z} = -\mu_{\nu, w} \bmod \mathbb{Z}.
\]
By the construction of $\iota$ in \eqref{eq:kottwitz-central-cocharacters}, reduction of $\lambda \in \operatorname{Hom}_{\mathbb{Z}}(X^*(Z), \mathbb{Q})$ modulo $\mathbb{Z}$ corresponds to the class of $\iota(\lambda)$ modulo $\pi_1(G)$.
Taking $\lambda = -\mu_{\nu, w}$, we see that the two $\dot v$-coordinates agree if and only if $y_w + \iota(\mu_{\nu, w}) \in \pi_1(G)$.

For $v \notin S$, both $\dot v$-coordinates are zero.
For $w \in S_E \setminus \dot S_E$, we have $y_w \in \pi_1(G)$ and $\mu_{\nu, w} = 0$.
Hence the fiber product condition is equivalent to \eqref{eq:kottwitz-global-integrality}.
\end{proof}

Let $H$ be a connected reductive $F$-group.
Following \cite[(11.8)]{Kot14}, define
\[
A(F, H) \coloneqq \varprojlim_{E / F} \bigl( \pi_1(H)[V_E]_0 \bigr)_{\Gamma_{E / F}},
\]
where $E / F$ runs over the finite Galois extensions in $\bar{F}$ through which the $\Gamma$-action on $\pi_1(H)$ factors.
Here $\pi_1(H)[V_E]_0 = \pi_1(H) \otimes_{\mathbb{Z}} \mathbb{Z}[V_E]_0$.
For $E \subseteq K$, the transition map on coinvariants is given by
\[
\bigl( \pi_1(H)[V_K]_0 \bigr)_{\Gamma_{K / F}}
\longrightarrow \bigl( \pi_1(H)[V_E]_0 \bigr)_{\Gamma_{E / F}}, \qquad
\bigl[ (n_u)_{u \in V_K} \bigr]
\longmapsto \left[ \left( \sum_{u \mid w} n_u \right)_{w \in V_E} \right].
\]
By \cite[Lemma~11.2]{Kot14}, each transition map is an isomorphism, hence $A(F, H)$ is identified with $\bigl( \pi_1(H)[V_E]_0 \bigr)_{\Gamma_{E / F}}$.

Choose a section $t \colon V_E \to V_K$ of the restriction map.
Following the notation of \cite[Section~3.7]{Kal18a}, write $t_!$ for extension by zero along $t$, so that
\[
t_! \left( \sum_{w \in V_E} m_w [w] \right)
= \sum_{w \in V_E} m_w [t(w)].
\]
For $m \in \pi_1(H)[V_E]_0$, the inverse transition map sends $[m]$ to $[t_! m]$, since summation over the places above each $w$ recovers $m_w$.
By \cite[Lemma~11.2]{Kot14}, the finite-level norms satisfy $N_{K / F}(t_! m) = p_{K / E}(N_{E / F} m)$, where $p_{K / E}$ acts on the place coordinate.
They therefore induce Kottwitz's global norm map of \cite[Definition~11.3]{Kot14},
\[
N_H^\diamond \colon A(F, H) \longrightarrow
\bigl( \pi_1(H) \otimes_{\mathbb{Z}} X^*(D^{\mathrm{iso}}) \bigr)^\Gamma.
\]
For $a \in A(F, H)$ represented by $m \in \pi_1(H)[V_E]_0$, the element $N_H^\diamond(a)$ is represented by $N_{E / F} m$.

Under \eqref{eq:kottwitz-global-band-hom}, set $j_Z \coloneqq (i \otimes 1)^\Gamma$.
Since $i$ is injective and $X^*(D^{\mathrm{iso}})$ is torsion-free, $j_Z$ is injective.
Thus the first projection identifies the fiber product of $N_G^\diamond$ and $j_Z$ with the subgroup
\begin{equation}\label{eq:kottwitz-global-central-target}
A(F, Z \to G) \coloneqq
\{ a \in A(F, G) \mid N_G^\diamond(a) \in \operatorname{im}(j_Z) \}.
\end{equation}
For $a \in A(F, Z \to G)$, let $\nu_a \in \operatorname{Hom}_F(D^{\mathrm{iso}}, Z)$ be the unique element satisfying $j_Z(\nu_a) = N_G^\diamond(a)$.

\begin{proposition}\label{prop:kottwitz-global-recovery}
The map from the fiber product in \eqref{eq:kottwitz-global-ab} to $A(F, Z \to G)$ given by
\begin{equation}\label{eq:kottwitz-global-coordinate}
([y], \nu) \longmapsto [y + \iota(\mu_\nu)]
\end{equation}
is an isomorphism of abelian groups, natural in $(Z \to G)$.
For $a \in A(F, Z \to G)$, choose a common admissible pair $(E, S)$ at which $m \in \pi_1(G)[S_E]_0$ represents $a$, $\mu_{\nu_a}$ is defined, and $N_{E / F} m = i((\nu_a)_E)$.
The inverse sends $a$ to $\bigl( [m - \iota(\mu_{\nu_a})], \nu_a \bigr)$.
Consequently, composition with \eqref{eq:kottwitz-global-ab} gives a natural isomorphism
\begin{equation}\label{eq:kottwitz-global-ab-target}
H^1_{\mathrm{ab}}(\mathcal{E}^{\mathrm{iso}}, Z \to G) \xrightarrow{\sim} A(F, Z \to G).
\end{equation}
\end{proposition}

\begin{proof}
Let $([y], \nu)$ belong to the fiber product in \eqref{eq:kottwitz-global-ab}, and choose an admissible pair $(E, S)$ at which $y$ represents $[y]$ and $\mu_\nu$ is defined.
Set $m \coloneqq y + \iota(\mu_\nu)$.
By \eqref{eq:kottwitz-global-integrality}, each $m_w$ belongs to $\pi_1(G)$.
Since $y$ and $\mu_\nu$ both have sum zero, we have $m \in \pi_1(G)[S_E]_0$.
Using $N_{E / F} y = 0$ and \eqref{eq:kottwitz-global-correction}, we obtain
\[
N_{E / F} m = \iota(N_{E / F} \mu_\nu) = i(\nu_E).
\]
Thus $N_G^\diamond([m]) = j_Z(\nu)$, so $[m] \in A(F, Z \to G)$ and $\nu_{[m]} = \nu$.

Conversely, let $a \in A(F, Z \to G)$ and choose an admissible pair $(E, S)$ at which $m \in \pi_1(G)[S_E]_0$ represents $a$ and $\mu_{\nu_a}$ is defined.
Since $N_G^\diamond(a) = j_Z(\nu_a)$, the elements $N_{E / F} m$ and $i((\nu_a)_E)$ have the same image in
\[
\pi_1(G) \otimes_{\mathbb{Z}} X^*(D^{\mathrm{iso}})
\simeq \varinjlim_{K / F} \pi_1(G)[V_K]_0.
\]
Their images under $p_{K / E}$ therefore agree for some $K$ containing $E$.
Since $p_{K / E}(N_{E / F} m) = N_{K / F}(t_! m)$ and $p_{K / E}((\nu_a)_E) = (\nu_a)_K$, we may enlarge $(E, S)$ and assume that $N_{E / F} m = i((\nu_a)_E)$.

Set $y \coloneqq m - \iota(\mu_{\nu_a})$.
Then $y$ has sum zero, and \eqref{eq:kottwitz-global-correction} yields
$N_{E / F} y = N_{E / F} m - i((\nu_a)_E) = 0$.
For $w \in S_E \setminus \dot S_E$, we have $\mu_{\nu_a, w} = 0$, so $y_w = m_w \in \pi_1(G)$.
Hence $y \in \tilde{\pi}_1(G, Z)[S_E, \dot S_E]_0^{N_{E / F}}$.
Since $y + \iota(\mu_{\nu_a}) = m$, the pair $([y], \nu_a)$ satisfies \eqref{eq:kottwitz-global-integrality}.

To check compatibility with enlargement of $E$, choose a section $t \colon V_E \to V_K$ such that $t(\dot v \vert_E) = \dot v \vert_K$ for every $v \in S$.
The transition maps send $[y]$ to $[t_! y]$ by \cite[Lemma~3.7.1]{Kal18a} and $\mu_\nu$ to $t_! \mu_\nu$ by \cite[Section~2.5]{KT23}.
Since $\iota$ acts coordinatewise, we have
\[
t_!(y + \iota(\mu_\nu)) = t_! y + \iota(t_! \mu_\nu),
\qquad
t_!(m - \iota(\mu_{\nu_a})) = t_! m - \iota(t_! \mu_{\nu_a}).
\]
Thus both formulas commute with enlargement of $E$.
Enlarging $S$ extends each family by zero.

At a fixed pair $(E, S)$, changing $y$ by an element of $I_\Gamma(\pi_1(G)[S_E]_0)$ changes $m$ by the same element and hence preserves $[m]$.
Conversely, let $m, m' \in \pi_1(G)[V_E]_0$ represent the same class $a$ in $\bigl( \pi_1(G)[V_E]_0 \bigr)_{\Gamma_{E / F}}$.
Then $m' - m = \sum_\sigma (\sigma - 1) m_\sigma$ for finitely many $m_\sigma \in \pi_1(G)[V_E]_0$.
Choose a finite set $S' \supseteq S$ containing the places below the supports of $m$, $m'$, and all the $m_\sigma$.
Extending $\mu_{\nu_a}$ by zero, we obtain
\[
\bigl( m' - \iota(\mu_{\nu_a}) \bigr) - \bigl( m - \iota(\mu_{\nu_a}) \bigr)
= \sum_\sigma (\sigma - 1) m_\sigma
\in I_\Gamma \bigl( \pi_1(G)[S'_E]_0 \bigr).
\]
Choose a pair $(E_j, S_j)$ in the tower with $E \subseteq E_j$ and $S' \subseteq S_j$.
By the compatibility with the transition maps proved above, the two representatives give the same class at $(E_j, S_j)$ and hence in \eqref{eq:global-rigid-formula}.

Both maps are therefore well defined.
They are mutually inverse since $\nu_{[y + \iota(\mu_\nu)]} = \nu$ and $(m - \iota(\mu_{\nu_a})) + \iota(\mu_{\nu_a}) = m$.
For two homomorphisms $\nu, \nu' \colon D^{\mathrm{iso}} \to Z$, choose a common admissible pair for their character maps.
Since $((\nu \nu') \circ s)^* = (\nu \circ s)^* + (\nu' \circ s)^*$, evaluation at $(1, w)$ yields $\mu_{\nu \nu', w} = \mu_{\nu, w} + \mu_{\nu', w}$ for each $w \in S_E$.
Thus \eqref{eq:kottwitz-global-coordinate} is additive.
Naturality follows from the construction of $\mu_\nu$ and the naturality of $\iota$.
\end{proof}

\begin{lemma}\label{lem:kottwitz-global-central-target}
Let $q \colon G \to G / Z$ be the quotient map.
Then
\[
A(F, Z \to G)
=
\ker \left[
A(F, G) \xrightarrow{N_{G / Z}^\diamond \circ q_*}
\bigl( \pi_1(G / Z) \otimes_{\mathbb{Z}} X^*(D^{\mathrm{iso}}) \bigr)^\Gamma
\right].
\]
\end{lemma}

\begin{proof}
By \eqref{eq:central-quotient-fundamental-group}, the kernel of $q_* \colon \pi_1(G) \to \pi_1(G / Z)$ is $\operatorname{im}(i)$.
Since $X^*(D^{\mathrm{iso}})$ is torsion-free, tensoring with it and taking $\Gamma$-invariants yields $\ker((q_* \otimes 1)^\Gamma) = \operatorname{im}(j_Z)$.
For $a \in A(F, G)$, functoriality of the norm gives
$(q_* \otimes 1)^\Gamma \bigl( N_G^\diamond(a) \bigr) = N_{G / Z}^\diamond(q_* a)$.
The assertion follows from \eqref{eq:kottwitz-global-central-target}.
\end{proof}

Let $\kappa_G \colon B(G) \to A(F, G)$ be Kottwitz's map of \cite[Section~11.5]{Kot14}.
For $Z = Z(G)$, the subgroup $A(F, Z \to G)$ is Kottwitz's $A_0(F, G)$, and \cite[Proposition~15.5]{Kot14} identifies it with the image of $B(G)_{\mathrm{bas}}$ under $\kappa_G$.
For arbitrary central $Z$, the compatibility of $\kappa_G$ with the Newton map \cite[Proposition~11.5]{Kot14} then identifies $A(F, Z \to G)$ with $\kappa_G(B(Z \to G))$.

\begin{theorem}\label{thm:kottwitz-global-explicit}
The isomorphism \eqref{eq:kottwitz-global-ab-target} is independent of $s$ and induces an isomorphism of abelian groups, natural in $(Z \to G)$,
\begin{equation}\label{eq:kottwitz-global-central-kernel}
H^1_{\mathrm{ab}}(\mathcal{E}^{\mathrm{iso}}, Z \to G)
\xrightarrow{\sim}
\ker \left[
A(F, G) \xrightarrow{N_{G / Z}^\diamond \circ q_*}
\bigl( \pi_1(G / Z) \otimes_{\mathbb{Z}} X^*(D^{\mathrm{iso}}) \bigr)^\Gamma
\right].
\end{equation}
Under the inclusion of the kernel in $A(F, G)$, composition with abelianization gives the restriction of $\kappa_G$ to $B(Z \to G)$.
If $F$ is totally imaginary, this restriction is an isomorphism from $B(Z \to G)$, equipped with the group structure induced by abelianization, onto the kernel.
\end{theorem}

\begin{proof}
For a fixed splitting $s$, let $\Theta_{G, Z}$ denote the isomorphism of Proposition~\ref{prop:kottwitz-global-recovery}.
Lemma~\ref{lem:kottwitz-global-central-target} identifies its target with the kernel in \eqref{eq:kottwitz-global-central-kernel}.

First take $G = Z = T$ for an $F$-torus $T$.
Let $b \in B(T)$ have band map $\nu$, and let $m$ represent $\kappa_T(b) \in A(F, T)$.
Write $\mathcal{E}^{\mathrm{mid}}_{\dot V}$ for the gerbe of \cite[Section~3.9]{KT23} with band $D^{\mathrm{mid}}_{\dot V}$.
By \cite[Corollary~3.8.6]{KT23} and Lemma~\ref{lem:existence-comparison}, we may lift $s$ to a morphism $\widetilde{s} \colon \mathcal{E}^{\mathrm{mid}}_{\dot V} \to \mathcal{E}^{\mathrm{iso}}$.
Lift $c^{\mathrm{iso}}$ and $c^{\mathrm{rig}}$ to morphisms of gerbes as in \cite[Section~3.9]{KT23}.
The resulting composites with $\widetilde{s}$ induce $s \circ c^{\mathrm{iso}} = \operatorname{id}$ and $s \circ c^{\mathrm{rig}} = \phi$ on the bands.
By Remark~\ref{rmk:global-kottwitz-lift-independence}, pullback along the first composite agrees with pullback along the identity, and pullback along the second agrees with $f^*$.
The Tate--Nakayama isomorphism for $T$ in \cite[Section~3.13, p.~800]{KT23} is the inverse of $\kappa_T$ by \cite[(9.1) and Section~11.5]{Kot14}.
Thus \cite[Proposition~3.13.10]{KT23} identifies $\widetilde{s}^* b$ with $([m], (\nu \circ s)^*)$.
The first coordinate is determined by pullback to $\mathcal{E}^{\mathrm{iso}}$ and the second by restriction to $D^{\mathrm{mid}}_{\dot V}$.
Choose $(E, S)$ so that $m$ and $\mu_\nu$ are both represented there.
In the notation of \cite[Proposition~3.13.15]{KT23}, evaluation of the second coordinate at $(1, w)$ gives $\mu_{\nu, w}$.
That proposition therefore identifies $f^* b$ with $[m - \mu_\nu]$ under $\kappa_{\mathrm{rig}, \dot V}$.
Since $\iota$ is the identity on $X_*(T) \otimes_{\mathbb{Z}} \mathbb{Q}$ for $(T \to T)$, \eqref{eq:kottwitz-global-coordinate} gives
\[
\Theta_{T, T}(\operatorname{ab}(b)) = [(m - \mu_\nu) + \mu_\nu] = \kappa_T(b).
\]

Now take $b \in B(Z \to G)$.
Its image in $G / Z$ defines a class in $H^1(F, G / Z)$.
By \cite[Lemma~A.1]{Kal18a}, this class is represented by a cocycle in a maximal $F$-torus $T / Z$ of $G / Z$, where $T$ is its inverse image in $G$.
Lifting a conjugating element from $(G / Z)(\bar{F})$ to $G(\bar{F})$ gives a representative of $b$ with values in $T(\bar{F})$ and the same band map.
It therefore defines $b_T \in B(Z \to T)$ mapping to $b$.
Write $h \colon T \hookrightarrow G$ for the inclusion.
Naturality for the maps from $(Z \to T)$ to $(T \to T)$ and $(Z \to G)$, the torus calculation, and functoriality of $\kappa_G$ give, in $A(F, G)$,
\[
\Theta_{G, Z}(\operatorname{ab}(b)) = h_*\kappa_T(b_T) = \kappa_G(b).
\]
Here $b_T$ is viewed in $B(T)$ when applying $\kappa_T$.
Since abelianization is surjective, this equality proves independence of $s$.
For totally imaginary $F$, abelianization is bijective by Corollary~\ref{cor:abelianization-bijective}.
\end{proof}

In particular, $\nu \in \operatorname{Hom}_F(D^{\mathrm{iso}}, Z)$ lies in the image of $\mathcal{N}$ if and only if $j_Z(\nu) \in \operatorname{im}(N_G^\diamond)$.
Under \eqref{eq:kottwitz-global-ab-target}, the abelianized fiber over such a $\nu$ is $(N_G^\diamond)^{-1}(j_Z(\nu))$ by Proposition~\ref{prop:kottwitz-global-recovery}.
It is a coset of $\ker(N_G^\diamond)$, which abelianized inflation identifies with $H^1_{\mathrm{ab}}(F, G) \simeq \mathcal{F}^1(\pi_1(G) / F)$ by Corollary~\ref{cor:general-global-newton-fibers}.

For each real place $v$, let $\mathcal{E}^{\mathrm{iso}}_v$ be Kottwitz's local gerbe over $F_v$, and set $B(F_v, Z \to G) \coloneqq H^1(\mathcal{E}^{\mathrm{iso}}_v, Z \to G)$.
On central cohomology and its abelianization, write $\operatorname{loc}_v$ for localization to the base change $(\mathcal{E}^{\mathrm{iso}})_v$ followed by pullback along a morphism $\mathcal{E}^{\mathrm{iso}}_v \to (\mathcal{E}^{\mathrm{iso}})_v$ inducing the localization map on the bands.
The induced map on cohomology is independent of the choice of this morphism by \cite[Section~3.11]{KT23}.
Set $\operatorname{loc}_\infty \coloneqq (\operatorname{loc}_v)_{v \text{ real}}$.

\begin{corollary}\label{cor:kottwitz-global-cartesian}
Abelianization and localization give a Cartesian square
\[
\begin{tikzcd}[column sep=large]
B(Z \to G) \arrow[r, "\operatorname{ab}"] \arrow[d, "\operatorname{loc}_\infty"']
& H^1_{\mathrm{ab}}(\mathcal{E}^{\mathrm{iso}}, Z \to G) \arrow[d, "\operatorname{loc}_\infty"] \\
\displaystyle \prod_{v \text{ real}} B(F_v, Z \to G) \arrow[r, "{\prod_v \operatorname{ab}_v}"']
& \displaystyle \prod_{v \text{ real}} H^1_{\mathrm{ab}}(\mathcal{E}^{\mathrm{iso}}_v, Z \to G).
\end{tikzcd}
\]
Thus \eqref{eq:kottwitz-global-central-kernel}, the formulas \eqref{eq:kottwitz-real-central-target} and \eqref{eq:kottwitz-real-full}, and localization on abelianized cohomology determine $B(Z \to G)$.
\end{corollary}

\begin{proof}
For each real place $v$, applying \eqref{eq:comparison-full-recovery} to $\mathcal{E}^{\mathrm{iso}}_v \to (\mathcal{E}^{\mathrm{iso}})_v$ gives a Cartesian square.
Taking the product of these squares over the real places and composing with the Cartesian square in Theorem~\ref{prop:global-local-cartesian} yields the desired Cartesian square.
\end{proof}

\end{document}